\documentclass[11pt,a4paper,reqno]{amsart}
\usepackage[english]{babel}
\usepackage[T1]{fontenc}
\usepackage{verbatim}
\usepackage[shortlabels]{enumitem}
\usepackage{palatino}
\usepackage{amsmath}
\usepackage{mathabx}
\usepackage{amssymb}
\usepackage{mathrsfs}
\usepackage{amsthm}
\usepackage{amsfonts}
\usepackage{graphicx}
\usepackage{esint}
\usepackage{color}
\usepackage{mathtools}
\usepackage{overpic}
\usepackage{tikz}
\usepackage{mathtools}

\usepackage[colorlinks = true, citecolor = red]{hyperref}
\title[Curvilinear two-ends Furstenberg inequality and Fourier decay]{Two-ends Furstenberg inequality for transversal families and applications to Fourier decay}

\author{William O'Regan}
\address{Department of Mathematics \\ University of British Columbia\\
	1984 Mathematics Rd \\Vancouver \\BC V6T 1Z2 \\Canada}
\email{woregan@math.ubc.ca}

\author{Shukun Wu}
\address{Department of Mathematics, Indiana University Bloomington}
\email{shukwu@iu.edu}

\author{Guangzeng Yi}
\address{Department of Mathematics and Statistics\\ University of Jyv\"askyl\"a,
	P.O. Box 35 (MaD)\\
	FI-40014 University of Jyv\"askyl\"a\\
	Finland}
\email{guangzeng.m.yi@jyu.fi}

\date{\today}

\thanks{WOR is supported in part by an NSERC Alliance grant administered by Pablo Shmerkin and Joshua Zahl.}
\thanks{S. Wu is partially supported by NSF2453583.}
\thanks{GY is supported by the Research Council of Finland via Tuomas Orponen’s project Approximate incidence geometry, grant no. 355453.}

\mathtoolsset{showonlyrefs}

\newcounter{stepctr}
\newif\ifinsteppage

\newcommand{\Ga}{\Gamma}

\newcommand{\les}{\lessapprox}

\newcommand{\cP}{\mathcal{P}}

\newcommand{\cD}{\mathcal{D}}

\newcommand{\cF}{\mathcal{F}}

\newcommand{\PP}{\mathbb{P}}

\newcommand{\cI}{\mathcal{I}}

\newcommand{\ga}{\gamma}

\newcommand{\De}{\Delta}

\newcommand{\bF}{\mathbf{F}}

\newcommand{\de}{\delta}

\newcommand{\dy}[2]{\cD_{{#1}}({#2})}

\newtheorem*{ack*}{Acknowledgment}

\newcommand{\R}{\mathbb{R}}

\newcommand{\N}{\mathbb{N}}

\newcommand{\Z}{\mathbb{Z}}

\newcommand{\diam}{\operatorname{diam}}

\newcommand{\dist}{\operatorname{dist}}

\def\Barint_#1{\mathchoice
	{\mathop{\vrule width 6pt height 3 pt depth -2.5pt
			\kern -8pt \intop}\nolimits_{#1}}%
	{\mathop{\vrule width 5pt height 3 pt depth -2.6pt
			\kern -6pt \intop}\nolimits_{#1}}%
	{\mathop{\vrule width 5pt height 3 pt depth -2.6pt
			\kern -6pt \intop}\nolimits_{#1}}%
	{\mathop{\vrule width 5pt height 3 pt depth -2.6pt
			\kern -6pt \intop}\nolimits_{#1}}}

\numberwithin{equation}{section}

\theoremstyle{plain}

\newtheorem{thm}[equation]{Theorem}
\newtheorem{"thm"}[equation]{"Theorem"}

\newtheorem{lemma}[equation]{Lemma}
\newtheorem{"lemma"}[equation]{"Lemma"}

\newtheorem{ex}[equation]{Example}
\newtheorem{cor}[equation]{Corollary}
\newtheorem{"proposition"}[equation]{"Proposition"}

\newtheorem{claim}[equation]{Claim}

\theoremstyle{definition}

\newtheorem{definition}[equation]{Definition}

\newtheorem{notation}[equation]{Notation}

\theoremstyle{remark}
\newtheorem{remark}[equation]{Remark}

\newcommand{\nref}[1]{(\hyperref[#1]{#1})}

\DeclareMathSymbol{\intop} {\mathop}{mathx}{"B3}

\begin{document}

	\begin{abstract}  
		We generalise the recent two-ends Furstenberg inequality due to Wang and the second author from lines to a family of transversal curves, and give a much simplified proof. 
		As an application, we present a result pertaining to the Fourier decay of fractal measures on convex curves.
	\end{abstract}
	
	\maketitle
	
	\tableofcontents

	\section{Introduction}
	A transversal curve family is a family of curves which prohibits tangencies--of which lines satisfy. 
	
	\begin{definition}[Transversal family]\label{def:transversality} Let $I \subset \R$ be a compact interval, and let $\mathcal{F} \subset C^{2}(I)$. We say that $\mathcal{F}$ is a \emph{$\mathfrak{T}$-transversal family} on $I$ if
		\begin{equation}\label{eq-trans1}
			\inf_{\theta\in I} (|f(\theta)- g(\theta)| + |f'(\theta) - g'(\theta)|) \geq \mathfrak{T}^{-1}\|f - g\|_{C^{2}(I)}, \quad f,g\in \mathcal{F}.
		\end{equation}
		Here and throughout this paper $\|f\|_{C^2(I)}:=\max_{x\in I}\sum_{k=0}^2|f^{(k)}(x)|$ and $\mathfrak{T}\geq1$.
	\end{definition}
	
	Besides the set of lines, a typical example of a transversal family is a set of translations of a single convex function, see \cite[Example 1.5]{OPY2025}. In the following, we give another important example arising from planar Fourier integral operators of H\"ormander type. A key motivating case is given by spectral projectors, which, after standard microlocalization, can be analyzed using oscillatory integrals (see \cite[(1.7)]{GWX}). 
	In this setting, functions are decomposed into small wave packets, each concentrated near a curve along which the main oscillation propagates.
	For spectral projectors, these curves are pieces of geodesics on the underlying surface. 
	When one separates into different frequency directions, the wave packets travel along curves whose tangents depend on the direction, so wave packets from different directions propagate along non-parallel curves. 
	As a result, after this decomposition, the core curves form a transverse family at the microlocal level.

	\begin{ex}\label{ex-1}
		For $x=(x_1,x_2)\in\R^2$ and $\xi\in\R$, consider the planar H\"ormander operator 
		\begin{equation}
			\label{eq:00}
			\mathcal H f(x):=\int e^{2\pi i\,\lambda\phi(x,\xi)}\,a(x,\xi)\,f(\xi)\,d\xi
		\end{equation}
		obeying the curvature condition
		\[\det(\phi_{x\xi}, \phi_{x\xi\xi})\not=0.\]
		Here $\phi:\R^3\to\R$, and $a$ is a bump function such that $\text{supp}_\xi a\subset[-1,1]$.
		For each $(\xi,v)\in \R^2$, we define
		\[\Gamma_{\xi,v}=\{x: \phi_\xi(x,\xi)=v\}.\]
		We show that, on sufficiently small neighborhoods, the curve family $\{\Gamma_{\xi,v}\}$ can be represented as the graphs of a transversal family.

		Due to the non-zero curvature condition, we may assume $\phi_{\xi x_2}\neq0$ (otherwise $\phi_{\xi x_1}\neq0$) in a small compact area $U\subset \R^3$. By implicit function theorem, $\phi_\xi(x,\xi)=v$ determines a smooth function $y=F(t,\xi,v)$ and a small cube $I_1\times I_2\times I_3\subset \R^3$ such that\[\phi_\xi(t,F(t,\xi,v),\xi)=v, \quad (t,\xi,v)\in I_1\times I_2\times I_3.\] For fixed $(\xi,v)\in I_2\times I_3$, we write $f_{\xi,v}(t):=F(t,\xi,v)$. Let $K=I_2\times I_3$, then 
		we claim that $\{f_{\xi,v}\}_{(\xi,v)\in K}$ is a $C_{K,I_1}$-transversal family on $I_1$. To see this, it follows from some basic calculations that
		\begin{equation}
			\max_{(\xi,v)\in K} \|\partial_\xi f_{\xi,v}\|_{C^2(I_1)}+\|\partial_v f_{\xi,v}\|_{C^2(I_1)} \leq C_{K,I_1}.
		\end{equation}
		Moreover, by using $\det(\phi_{x\xi}, \phi_{x\xi\xi})\not=0$, the Jacobian of the map $(\xi,v)\mapsto (f_{\xi,v}(t), f_{\xi,v}'(t))$ is non-zero, thus it is Bi-Lipschitz on $K$ for each $t\in I_1$ with constant depending on $K, I_1$. Consequently, the transversality condition \eqref{eq-trans1} follows immediately.
	\end{ex}

	Here is the main result of the paper, a transversal generalisation of the two-ends Furstenberg inequality which was established in \cite{WW1} for $t=1$ and in \cite{WW2} for all $t\in (0,2)$. The reader may check Section \ref{sec2} for some related definitions and notations.
	
	\begin{thm}\label{main-twoends}
		Let $t\in (0,2)$, $\mathfrak{T}\geq1$ and $\lambda\in [0,1]$. For any $\epsilon\in (0,1)$, there exists a small constant $\delta_0=\delta_0(t,\epsilon,\mathfrak{T})>0$ such that the following holds for any $\delta\in (0,\delta_0]$.
		
		Let $(\mathcal{F},\mathcal{P})_\delta$ be a given configuration. Here $\mathcal{F}\subset B_{C^2}(1)$ is $\mathfrak{T}$-transversal on $[-2,2]$. Assume that
		\begin{itemize}
			\item $\mathcal{F}$ is $\delta$-separated and a $(\delta,t)$-KT set;
			\item for each $f\in \mathcal{F}$, $\mathcal{P}(f)$ is $(\epsilon_1, \epsilon_2)$-two-ends and $\lambda$-dense.
		\end{itemize}
		Write $t^\ast=\{t,2-t\}$ and $\mathbf{E}_{\mathcal{F},\mathcal{P}}:=\cup_{f\in \mathcal{F}} \mathcal{P}(f)$. Then,
		\begin{equation}\label{equ-twoendsss}
			|\mathbf{E}_{\mathcal{F},\mathcal{P}}|\geq \delta^{\epsilon+t\epsilon_1/2}  \delta^{(t-1)/2} \ga_{\cP,t^\ast}^{-1/2}\cdot \lambda^{1/2}\sum_{f\in\mathcal{F}} |\mathcal{P}(f)|.
		\end{equation}
	\end{thm}
	
	The two-ends Furstenberg inequality has already found many applications; see, for example, \cite{doprod,OpHrit,WW1,Wu-weighted-L2}. We hope that the curvilinear version will lead to further applications in the future. In view of Example \ref{ex-1}, Theorem \ref{main-twoends} provides a potential tool for studying eigenfunctions of the Laplace–Beltrami operator on manifolds. For instance, one may seek to establish an analogue of the main result in \cite{Wu-weighted-L2}, with the Fourier extension operator replaced by the Hörmander-type operator \eqref{eq:00}. Such a result could be a useful ingredient in the study of eigenfunction distribution on manifolds. We will also present an application of Theorem \ref{main-twoends} in Subsection \ref{sub2}.
	
	Our proof of Theorem \ref{main-twoends} is much simpler than the one in \cite{WW2}. For the reader’s convenience, we include a brief sketch of the argument below.

	\subsection{Sketch proof of Theorem \ref{main-twoends} and structure of the paper}
	We will first reduce Theorem \ref{main-twoends} to the case when the two-ends condition for $\mathcal{P}(f)$ is replaced by $(\delta, \epsilon^2)$-set condition, see Definition \ref{def:deltarho}. Up to dyadic pigeonholing, we assume that $\cP(f)$ is uniform for all $f\in \mathcal{F}$ and they have the same branching function (view Definition \ref{def:branching} and Definition \ref{def:uniformset}). We apply an induction on scales strategy. 
	
	There are two main ingredients in our simplified proof. First, by applying the by now well-known multi-scale decomposition to each $\mathcal{P}(f)$, we obtain a scale $r\in (\delta,1)$ such that $[\cP(f)]_{r \rightarrow \de}$ is a $(\de/r,s)$-set with cardinality $\approx (r/\de)^s.$ The notation $[\cP(f)]_{r \rightarrow \de}$ denotes $\cP(f)$ when viewed between scales $r$ and $\de$. The second key tool is an intermediate scale selection lemma, see Lemma \ref{lem-main}. It replaces the use of \cite[Theorem 5.4]{DW2024} in \cite{WW2}, and is one of the main reasons why our proof is simpler. Namely, under the same assumptions as in Theorem \ref{main-twoends}, we can find a scale $\Delta\in [\delta,1]$ such that either the number of curves intersecting each $\delta$-cube has a small upper bound, or we can find a $\Delta$-cube family $\mathcal{Q}_\Delta$ so that the cardinality of $\mathbf{E}_{\mathcal{F},\mathcal{P}}\cap Q$ for each $Q\in\mathcal{Q}_\Delta$ has large lower bound, where \[\mathbf{E}_{\mathcal{F},\mathcal{P}}=\bigcup_{f\in \mathcal{F}} \mathcal{P}(f).\] We note that Lemma \ref{lem-main} is essentially an application of the curvilinear Furstenberg set estimates established in \cite[Theorem 1.11]{OPY2025} and an incidence estimate under Katz-Tao condition established in \cite[Theorem 1.4]{OS2026}.
	
	In light of the previous paragraph, we now work inside each $r$-cube which intersects $\mathbf{E}_{\mathcal{F},\mathcal{P}}$, and call it $\mathbf{Q}$. We consider those functions in $\mathcal{F}$ whose shadings $\mathcal{P}(f)$ intersect $\mathbf{Q}$ and let $\mathcal{F}_\mathbf{Q}$ be an incomparable subset of those functions. Now for each $f\in \mathcal{F}_{\mathbf{Q}}$, 
	\[\mathcal{P}_\mathbf{Q}(f):=\mathcal{P}(f)\cap \mathbf{Q}\]
	is contained in a curved $\sim \delta\times r$-rectangle which we call tube segments. This is where the proof now diverges from the one in \cite{WW2}. We apply our scale selection lemma (Lemma \ref{lem-main}) to each $(\mathcal{F}_\mathbf{Q}, \mathcal{P}_\mathbf{Q})$ (precisely to its refinement) after scaling by $r^{-1}$. This tells us that 
	\begin{enumerate}
		\item either the number of $\de \times r$ tube segments which intersect each $p \in \mathbf{E}_{\mathcal{F},\mathcal{P}} \cap \mathbf{Q}$ is sufficiently bounded to ensure that we obtain a suitable lower bound for the cardinality $|\mathbf{E}_{\mathcal{F},\mathcal{P}} \cap \mathbf{Q}|$, thus \eqref{equ-twoendsss} follows roughly from the fact \[|\mathbf{E}_{\mathcal{F},\mathcal{P}}|\gtrapprox |\mathbf{E}_{\mathcal{F},\mathcal{P}} \cap \mathbf{Q}|\cdot |\mathcal{D}_r(\mathbf{E}_{\mathcal{F},\mathcal{P}})|;\]
		
		\item or we obtain another scale $\De\in [\delta,r]$ so that for a large collection of $\De$-cubes $\mathcal{G}_\mathbf{Q} \subset \mathbf{Q}$ the cardinality $|\mathbf{E}_{\mathcal{F},\mathcal{P}} \cap Q|$ is sufficiently large. In this case, we define for each $f$ in a $(\Delta,t)$-KT subset $\widetilde{\mathcal{F}}\subset\mathcal{F}$ that \[\widetilde{\mathcal{P}}(f):=\bigcup_{\mathbf{Q}\in \mathcal{D}_r(\mathbf{E}_{\mathcal{F},\mathcal{P}})} \mathcal{G}_\mathbf{Q}(f),\]
		where $\mathcal{G}_\mathbf{Q}(f)$ is roughly those $\Delta$-cubes in $\mathcal{G}_\mathbf{Q}$ intersecting $\mathcal{P}(f)$. Finally, applying induction hypothesis at scale $\Delta$ to $(\widetilde{\mathcal{F}}, \widetilde{\mathcal{P}})$, \eqref{equ-twoendsss} follows roughly from
		\[|\mathbf{E}_{\mathcal{F},\mathcal{P}}|\gtrsim |\mathbf{E}_{\widetilde{\mathcal{F}},\widetilde{\mathcal{P}}}|\cdot \min_{Q\in \cup_\mathbf{Q}\mathcal{G}_\mathbf{Q}}|\mathbf{E}_{\mathcal{F},\mathcal{P}} \cap Q|, \]
		which fulfils our aims.
	\end{enumerate}
	
	The outline above explains what happens in Section \ref{sec3}. Additionally, Section \ref{sec2} contains preliminaries, and Section \ref{sec4} contain the details of one application
	of our curvilinear Furstenberg two-ends theorem, introduced in Section \ref{sub2}.

	\begin{subsection}{A Fourier-analytic application}\label{sub2}
		One motivation for this paper is to extend \cite[Theorem 1.8]{doprod} to general convex curves besides the parabolas. Here it is:
		
		\begin{thm}\label{thm-L6decay}
			Let $s\in (0, 2/3]$ and $\mathbf{c}\geq \mathbf{d}>0$. Fix $f\in C^3([-6,6])$ such that 
			\[\|f\|_{C^3([-6,6])}\leq \mathbf{c}\quad \text{and} \quad \min_{x\in [-6,6]}|f''(x)|\geq \mathbf{d}.\]
			Write \[\Gamma:=\{(x,f(x)):x\in [-1,1]\}.\]  Let $\mu$ be an $s$-Frostman measure supported on $\Gamma$. Then for any $\epsilon \in (0,1)$ and $R\geq1$ we have
			\begin{equation}\label{equ-decay}
				\|\hat{\mu}\|_{L^6(B_R)}^6 \leq C(\mathbf{c},\mathbf{d},s,\epsilon) R^{2-5s/2+\epsilon}.
			\end{equation}
		\end{thm}
		
		The problem was first studied by Orponen in \cite{O1} where he conjectured that the sharp decay should be $\min\{1+s,3s\}$ for all $s\in (0,1]$, see the sharp example \cite[Example 1.8]{O1}. 
		
		Up to now, most results for this problem were obtained via incidence estimates. For $s \in [2/3,1]$, the sharp result was obtained in \cite{O2} for parabolas and in \cite{Yi} for general convex curves by using an upper bound incidence estimate. When $s\in (0,2/3]$, the problem is more difficult and remains open. Demeter-Wang \cite{DW2024} proved the decay $2-9s/4$ when $s\in (0,1/2]$ by using their Szemer\'{e}di-Trotter type incidence estimate. Very recently, after Wang and the second author \cite{WW1} established the two-ends Furstenberg inequality, Demeter and the first author \cite{doprod} improved the decay to $2-5s/2$ for parabolas and they also obtained some partial results with decoupling method. Moreover, we note that when considering $L^p$-decay, the sharp result has been established for general convex curves in \cite[Theorem 1.16]{OPY2025}.
	\end{subsection}

	\subsection*{Acknowledgements}
	We thank Tuomas Orponen for carefully reading the draft and providing numerous comments that helped improve the exposition of this paper. GY would like to thank Tuomas Orponen for his constant support. WOR also thanks Tuomas Orponen and the University of Jyv\"askyl\"a for their hospitality during his visit in the fall of 2025, where this project began; thanks are extended to Pablo Shmerkin and Joshua Zahl for their support.

	\section{Preliminaries}\label{sec2}
	\subsection{Notations and definitions} We adopt the standard notations $\lesssim$, $\gtrsim$, $\sim$. For example, $A\lesssim B$ means $A\leq CB$ for some constant $C>0$, while $A\lesssim_rB$ stands for $A\leq C(r)B$ for a positive function $C(r)$. We will denote $A \lessapprox_\delta B$, $A \gtrapprox_\delta B$, $A \approx_\delta B$ or $A\approx B$ to hide slowly growing functions of $\delta$ such as $\log(1/\delta)$ and $\delta^{-\epsilon}$. The precise meaning of the $\lessapprox$ notation will always be explained separately.
	
	When $A$ is a finite set, $|A|$ always denotes the cardinality of $A$.
	
	For $\delta \in 2^{-\N}$, dyadic $\delta$-cubes in $\R^2$ are denoted $\mathcal{D}_{\delta}(\R^2)$. Elements of $\mathcal{D}_{\delta}(\R^2)$ are typically denoted with letters $p,q$. For $P \subset \R^2$, we write $\mathcal{D}_{\delta}(P) := \{p \in \mathcal{D}_{\delta}(\R^2) : P \cap p \neq \emptyset\}$. In particular, we abbreviate $\mathcal{D}_{\delta}:=\mathcal{D}_{\delta}([-1,1]^2)$.

	The following are some basic definitions for our problem.
	
	\begin{definition}[Covering numbers] 
		Let $(X,d)$ be a metric space. For $P \subset X$ and $r>0$, we write $|P|_r$ as the minimum number of $r$-balls in $X$ needed to cover $P$.
	\end{definition}
	
	The following two spacing conditions are now well-known in this area.
	
	\begin{definition}[$(\delta,s,C)$-set]\label{def:spacingcondition}
		Let $(X,d)$ be a metric space. Let $\delta\in (0,1]$ and $s\geq 0$. We say a bounded subset $P\subset X$ is a $(\delta,s,C)$-set if
		\begin{equation}\label{def 1}
			|P\cap B(x,r)|_\delta\leq Cr^s |P|_\delta, \qquad x \in X, \, \delta \leq r \leq 1. 
		\end{equation}
		Here $B(x,r)$ refers to a ball in $(X,d)$. A $(\delta,s,C)$-set is called a $(\delta,s)$-set if the value of the constant $C > 0$ is irrelevant. 
	\end{definition}
	
	\begin{definition}[Katz-Tao $(\delta,s,C)$-set]\label{def:katzTaoSet}
		Let $(X,d)$ be a metric space. Let $\delta\in (0,1]$ and $s\geq 0$. We say a bounded subset $P\subset X$ is a Katz-Tao $(\delta,s,C)$-set if
		\begin{equation}\label{def 1}
			|P\cap B(x,r)|_\delta\leq C\left(\tfrac{r}{\delta}\right)^s, \qquad x \in X, \, \delta \leq r \leq 1. 
		\end{equation}
		A Katz-Tao $(\delta,s,C)$-set is called a $(\delta,s)$-KT set if the value of $C > 0$ is irrelevant. 
	\end{definition}

	\begin{definition}[Upper $(s,C)$-regular set]\label{def:Ahlfors}
		Let $(X,d)$ be a metric space and let $C, s>0$. We say $P\subset X$ is upper $(s,C)$-regular, if for any $x\in X$ and $0<r\leq R<\infty$,
		\[|P\cap B(x,R)|_r\leq C\left(\tfrac{R}{r}\right)^s.\]
		We also say that $P$ is \emph{upper $(\delta,s,C)$-regular} if the estimate above holds for $\delta \leq r \leq R < \infty$.
	\end{definition}

	\subsection{Transversal families} 
	
	This part was taken from \cite[Section 2]{OPY2025} and one can check the related proofs in that paper.
	
	\begin{definition}[Transversal family]\label{def:transversality} Let $I \subset \R$ be a compact interval, and let $\mathcal{F} \subset C^{2}(I)$. We say that $\mathcal{F}$ is a \emph{$\mathfrak{T}$-transversal family} on $I$ if
		\begin{equation}\label{eq-trans}
			\inf_{\theta\in I} (|f(\theta)- g(\theta)| + |f'(\theta) - g'(\theta)|) \geq \mathfrak{T}^{-1}\|f - g\|_{C^{2}(I)}, \quad f,g\in \mathcal{F}.
		\end{equation}
		Here and throughout this paper $\|f\|_{C^2(I)}:=\max_{x\in I}\sum_{k=0}^2|f^{(k)}(x)|$ and $\mathfrak{T}\geq1$.
	\end{definition}
	
	We record that transversal families are upper $2$-regular subsets of $C^{2}(I)$:
	\begin{lemma}\label{lem-Ahlforsregularity}\label{lem.indentify}
		Every $\mathfrak{T}$-transversal family $\mathcal{F} \subset C^{2}(I)$ admits a $\sqrt{2}\mathfrak{T}$-bi-Lipschitz embedding into $\R^{2}$ via the map $A_{x_0}(f) = (f(x_0),f'(x_0)),$ where $x_0 \in I$ is arbitrary. In particular, $\mathcal{F}$ is upper $(2,20\mathfrak{T}^2)$-regular. 
	\end{lemma}

	\begin{definition}[Rescaling map $T_B$] \label{def:rescalingmap} Let $B := B(f_{0},r_{0}) \subset C^{2}(I)$. We define 
		\begin{displaymath} 
			T_{B}(f) :=T_{f_{0},r_{0}}:= r_{0}^{-1}(f - f_{0}), \qquad f \in C^{2}(I). 
		\end{displaymath} 
	\end{definition}
	
	Transversal families behave well under rescaling maps in $\mathcal{F}$.
	\begin{lemma}\label{lem1} Let $\mathcal{F} \subset C^{2}(I)$ be a $\mathfrak{T}$-transversal family, let $f_{0} \in \mathcal{F}$, and let $B := B(f_{0},r) \subset C^{2}(I)$. Then
		\begin{displaymath} \mathcal{F}_{B} := T_B(\mathcal{F}):=\{T_B(f): f \in \mathcal{F}\} \subset C^{2}(I) \end{displaymath}
		is a $\mathfrak{T}$-transversal family.
	\end{lemma}

	Besides the rescaling operation $f \mapsto (f - f_{0})/r$, we will also need to consider another one, where $f \mapsto (f(rx +x_{0}) - y_{0})/r$. Transversality is also preserved by this operation, but the intervals and constants of transversality may change, as follows:
	\begin{lemma}\label{lem2}
		Let $\mathcal{F}\subset C^{2}(I)$ be a $\mathfrak{T}$-transversal family. Let $x_{0},y_{0} \in \R$ and $r \in (0,1]$. Then,
		\[T_{(x_{0},y_{0}),r}(\mathcal{F}):=\{f_{(x_{0},y_{0}),r}:  f\in \mathcal{F}\} \subset C^{2}(J)\]
		is a $(|J|\mathfrak{T}+1)$-transversal family on any compact interval $J \subset (I - x_{0})/r$, where $f_{(x_{0},y_{0}),r}(x)=(f(rx+x_0)-y_0)/r$ and $|J|$ is the length of $J$.
	\end{lemma}

	Now we introduce the dyadic cubes associated with a transversal family.
	
	\begin{definition}[Dyadic system]\label{def-pullbackdyadic}
		Let $\mathcal{F}\subset C^2(I)$ be a $\mathfrak{T}$-transversal family, and let $x_{0}$ be the midpoint of $I$. Let $A:\mathcal{F}\to \R^2$, $A(f)=(f(x_0),f'(x_0))$ be the bi-Lipschitz map in the proof of Lemma \ref{lem-Ahlforsregularity}. The \emph{dyadic system associated with $\mathcal{F}$}, denoted $\mathcal{D}(\mathcal{F})$, is defined as $\mathcal{D}(\mathcal{F}):=\bigcup_{r\in 2^{\Z}} \mathcal{D}_r(\mathcal{F})$, where
		\[\mathcal{D}_r(\mathcal{F}):= \{A^{-1}(p) : p \in \mathcal{D}_{r}(A(\mathcal{F}))\}.\]
		Thus, cubes associated with $\mathcal{F}$ are pull-backs of dyadic squares in $\R^{2}$. We will use boldface symbols $\mathbf{F},\mathbf{F}'$ etc. to denote dyadic cubes in $\mathcal{F}$. For any $r\in 2^{\Z}$ and $p\in \mathcal{D}_r(A(\mathcal{F}))$, we say $\mathbf{F}=A^{-1}(p)\in  \mathcal{D}_r(\mathcal{F})$ is a dyadic cube with side length $r$.
		
		If $\mathcal{F}' \subset \mathcal{F}$ and $r \in 2^{\Z}$, we also define $\mathcal{D}_{r}(\mathcal{F}') := \{\mathbf{F} \in \mathcal{D}_{r}(\mathcal{F}) : \mathbf{F} \cap P \neq \emptyset\}$. 
	\end{definition}
	
	\begin{remark} If $\mathcal{F}$ is a transversal family, and $\mathcal{F}' \subset \mathcal{F}$ is an arbitrary subset, then $\mathcal{F}'$ is a transversal family on its own, and Definition \ref{def-pullbackdyadic} yields an associated dyadic system $\mathcal{D}(\mathcal{F}') = \bigcup_{r \in 2^{\Z}} \mathcal{D}_{r}(\mathcal{F}')$. Fortunately, as is clear from the construction, the families $\mathcal{D}_{r}(\mathcal{F}')$ coincide with the subsets of $\mathcal{D}_{r}(\mathcal{F})$ defined on the last line of Definition \ref{def-pullbackdyadic}. \end{remark}

	The following result says that the dyadic covering number and ordinary ball covering number in a transversal family are comparable.
	\begin{cor}\label{cor-covernumber}
		Let $ P\subset \mathcal{F}$ be a bounded and non-empty subset. Then for any dyadic number $r\leq \diam(P)$, we have
		\[\frac{|P|_r}{|\mathcal{D}_r(P)|}\in [\tfrac{1}{16},40\mathfrak{T}^4].\]
	\end{cor}

	\subsection{Uniform sets and branching functions.}
	
	This part is mainly stated for transversal families, see \cite[Section 2]{OPY2025} for the proofs. We should also mention that the corresponding definitions and results for sets of dyadic cubes in $\mathbb R^2$ have been established in \cite{PO2023}. In what follows, we consider a $\mathfrak{T}$-transversal family $\mathcal{F} \subset B_{C^2}(1)$ over $[-2,2]$. 
	
	\begin{definition}[Uniform sets]\label{def:uniformset}
		Let $n \ge 1$, and let $0<\delta = \Delta_n < \Delta_{n-1} < \cdots < \Delta_1 \le \Delta_0 = 1$ be a sequence of scales in $2^{-\N}$. We say $\mathcal{F}$ is $\{\Delta_j\}_{j=1}^n$-\emph{uniform} if there is a sequence $\{N_j\}_{j=1}^n$ with $N_j \in 2^{\mathbb{N}}$ such that \[|\mathcal{D}_{\Delta_j}(\mathcal{F} \cap \mathbf{F})| = N_j\quad\text{for all}\; j \in \{1,\dots,n\} \;\text{and for all}\; \mathbf{F} \in \mathcal{D}_{\Delta_{j-1}}(\mathcal{F}).\]
		When $\Delta_j=|\log \delta|^{-j}$ and $n=\log_{|\log\delta|}\delta^{-1}$, we simply say that $\mathcal{F}$ is \emph{uniform}.
	\end{definition}
	
	It turns out that we can always find “dense uniform subsets” in a transversal family.
	
	\begin{lemma}\label{lem-uniformsubset}
		For each $\mathfrak{T}\geq1$, there exists $\delta_0=\delta_0(\mathfrak{T})>0$ such that the following holds for all $\delta\in (0,\delta_0]$.
		
		Let $\mathcal{F} \subset B_{C^2}(1)$ be a $\mathfrak{T}$-transversal family. Then there exists a uniform sub-family $\mathcal{F}'\subset\mathcal{F}$ such that 
		\[|\mathcal{F}'|_\delta \gtrsim (\log |\log \delta |)^{\tfrac{-|\log\delta|}{\log |\log \delta |}} |\mathcal{F}|_\delta.\]
		
	\end{lemma}

	Lemma \ref{lem-uniformsubset} has the following useful corollary.
	
	\begin{cor}\label{cor-uniformsets}
		For every $s\in(0,2]$ and $\epsilon\in(0,1)$, there exists $\delta_0=\delta_0(\mathfrak{T},\epsilon)>0$ such that the following holds for all $\delta\in(0,\delta_0]$. Assume $\mathcal{F}$ is a $(\delta,s,\delta^{-\epsilon})$-set. Then, there exists a uniform subset $\mathcal{F}'\subset\mathcal{F}$ so that $|\mathcal{F}'|_\delta \geq\delta^\epsilon|\mathcal{F}|_\delta$. In particular, $\mathcal{F}'$ is a $(\delta,s,\delta^{-2\epsilon})$-set.
	\end{cor}
	
	\begin{definition}[Branching function]\label{def:branching}
		Let $\mathcal{F}\subset B_{C^2}(1)$ be a uniform $\mathfrak{T}$-transversal family. We define the \emph{branching function} $\beta: [0,m] \to [0,\infty)$ by setting $\beta(0)=0$ and
		\[\beta(j) := \frac{\log(|\mathcal{D}_{\Delta_j}(\mathcal{F})|}{\log|\log \delta|}, \qquad j \in \{1, \dots, m\},\]
		and then interpolating linearly. Here $\Delta_j=|\log \delta |^{-j}$ and $m=\log_{|\log\delta|}\delta^{-1}$. \end{definition}

	We end this subsection by stating the following lemma for a collection of cube families in $\mathbb R^2$, see \cite[Lemma 2.13]{WW2} for its proof.
	
	\begin{lemma}\label{lem:uniform branching}
		Let $\delta\in(0,1)$, $M=|\log\de|$, and  $n=\lceil \log_M \delta^{-1} \rceil$. Suppose $\mathcal{G}=\{\mathcal{P}_k\}_{k\in I}$, where $\mathcal{P}_k\subset \mathcal{D}_\delta$ and $I$ is a finite index set. Let $\beta_k$ be the branching function of $\mathcal{P}_k$.
		
		Then there is subset $\mathcal{G}'\subset\mathcal{G}$ with $|\mathcal{G}'|\gtrsim (n \log M)^{-n}|\mathcal{G}|$ and a uniform branching function $\beta'$ such that for any $\mathcal{P}_k\in\mathcal{G}'$, $|\beta_k(j)-\beta'(j)|\leq  |\log \delta|^{-1}$ for all $0\leq j\leq n$. In particular, $|\mathcal{G}'|\gtrapprox|\mathcal{G}|$.
	\end{lemma}

	\subsection{Some results from incidence geometry}
	\begin{definition}[$\delta$-configuration]\label{def:shading}
		Let $\delta\in(0,1]\cap 2^{\N}$. We say $(\mathcal{F},\mathcal{P})$ is a {\bf{$\delta$-configuration}} (or simply a configuration), if $\mathcal{F}\subset B_{C^2}(1)\subset C^2([-2,2])$ is a finite $\mathfrak{T}$-transversal family, and for each $f\in \mathcal{F}$, there is a set of dyadic $\delta$-cubes or $\delta$-balls $\mathcal{P}(f)$ such that 
		\begin{itemize}
			\item [(i)] $\cup\mathcal{P}(f)\subset [-1,1]^2$,
			\item [(ii)] $p\cap \Gamma_f\neq\emptyset$ for all $p\in \mathcal{P}(f)$.
		\end{itemize}
		Write $\mathbf{E}_{\mathcal{F},\mathcal{P}}:=\cup_{f\in\mathcal{F}} \mathcal{P}(f)$. We call the mapping $\mathcal{P}:\mathcal{F}\to \mathbf{E}_{\mathcal{F},\mathcal{P}}$ a {\bf{shading}} in the terminology of \cite{WW1}.
		
		We say a $\delta$-configuration $(\mathcal{F},\mathcal{P})$ is {\bf $\lambda$-dense} if $|\mathcal{P}(f)|\geq \lambda \delta^{-1}$ for any $f\in \mathcal{F}$. We also writes $(\mathcal{F},\mathcal{P})_\delta$ to emphasize the dependence on $\delta$. 
	\end{definition}

	\begin{definition}
		Let $(\mathcal{F},\mathcal{P})_\delta$ be a given configuration. We say a $\delta$-configuration $(\mathcal{F}',\mathcal{P}')_\delta$ is a {\bf refinement} of $(\mathcal{F},\mathcal{P})_\delta$, if $\mathcal{F}'\subset \mathcal{F}$, $\mathcal{P}'(f)\subset \mathcal{P}(f)$ for all $f\in \mathcal{F}'$, and moreover
		\begin{equation}
			\sum_{f'\in \mathcal{F}'}|\mathcal{P}'(f)|\gtrapprox\sum_{f\in \mathcal{F}}|\mathcal{P}(f)|.
		\end{equation}
	\end{definition}
	
	For natational convenience, we define the cover of a shading by curved tubes.
	
	\begin{definition}\label{tube-segment}
		Let $(\mathcal{F},\mathcal{P})_\delta$ be a given configuration. Fix $f\in\mathcal{F}$ and $r\in (0,1]\cap 2^{-\N}$. For each $Q\in \mathcal{D}_r(\mathcal{P}(f))$, let $u_Q:=\Gamma_{f|_{I_Q}}$ be the graph of $f$ above $I_Q=\pi_x(Q)$, so we have $\mathcal{P}(f)\cap Q\subset u_Q(6\delta)$. Here $u_Q(6\delta):=\{(x,y)\in I_Q\times \R: |y-f(x)|\leq 6\delta\}$ is the vertical $6\delta$-neighborhood of $u_Q$. We define
		\begin{equation}\label{equ-curvedrectangle}
			[\mathcal{P}(f)]_r:=\{u_Q(6\delta): Q\in \mathcal{D}_r(\mathcal{P}(f))\}.
		\end{equation}
		In this paper, we will use $J$ to denote the element in the family $[\mathcal{P}(f)]_r$. 
	\end{definition}
	
	\begin{remark}
		Note that each element in $[\mathcal{P}(f)]_r$ is a curved rectangle of dimensions $\sim (\delta\times r)$ and each cube in $\mathcal{P}(f)$ is contained in a unique element of $[\mathcal{P}(f)]_r$. Moreover, since $f\in \mathcal{F}\subset B_{C^2}(1)$, each $u_Q$ can intersect at most two dyadic $r$-cubes. If $\mathcal{P}(f)$ is uniform, we have
		\[|\mathcal{P}(f)\cap J|\sim |\mathcal{P}(f)\cap Q|, \quad J\in [\mathcal{P}(f)]_r.\]
	\end{remark}
	
	\begin{definition}\label{twoendsreduction}
		Let $v>0$, $C>0$, and let $\delta\in(0,1)\cap 2^{\N}$. Let $(\mathcal{F},\mathcal{P})_\delta$ be a given configuration. For each $f\in\mathcal{F}$, $\mathcal{P}(f)\subset \mathcal{D}_\delta$ is uniform. Then we define the scale 
		\begin{equation}
			\label{twoendsreduction1}
			\rho:=\min\{r\in[\delta,1]\cap 2^{-\N}:|\mathcal{P}(f)|_r\leq C^{-1}r^{-v}\}.
		\end{equation}
		Consequently, since $\mathcal{P}(f)$ is uniform, for all $r\in[\delta,\rho]\cap 2^{-\N}$ and all $J\in [\mathcal{P}(f)]_\rho$, 
		\begin{equation}\label{twoendsreduction2}
			|\mathcal{P}(f)\cap J|_r\approx \frac{|\mathcal{P}(f)|_r}{|\mathcal{P}(f)|_\rho}\gtrsim (\rho/r)^v.
		\end{equation}
	\end{definition}
	\begin{remark}\label{remk-1}
		For $J=u_Q(6\delta)\in [\mathcal{P}(f)]_\rho$, let $S_J$ be the rescaling map taking $Q$ to $[-1,1)^2$. By the definition of $\rho=\rho(f;v,C)$, we see that $S_J(\mathcal{P}(f)\cap J)$ is a $(\delta/\rho,v, C')$-set for all $J\in [\mathcal{P}(f)]_\rho$, for some $C'\lessapprox1$.
	\end{remark}

	\begin{definition}[Two-ends condition]\label{def:twoends}
		Let $(\mathcal{F},\mathcal{P})_\delta$ be a given configuration. Let $0<\epsilon_2<\epsilon_1<1$.
		For $f\in \mathcal{F}$, we say $\mathcal{P}(f)$ is \emph{ $(\epsilon_1 ,\epsilon_2, C)$-two-ends} if
		\begin{equation}
			|\mathcal{P}(f)\cap B(x, \delta^{\epsilon_1})|_\delta\leq C\delta^{\epsilon_2} |\mathcal{P}(f)|, \quad x\in \R^2.
		\end{equation}
		We say $\mathcal{P}$ is $(\epsilon_1 ,\epsilon_2, C)$-two-ends if $\mathcal{P}(f)$ is $(\epsilon_1 ,\epsilon_2, C)$-two-ends for all $f\in \mathcal{F}$. When $C$ is not important in the context, we say $\mathcal{P}$ is \emph{$(\epsilon_1 ,\epsilon_2)$-two-ends}, or simply \emph{two-ends}.
	\end{definition}
	
	By the two definitions above, we can get the following result, see \cite[Lemma 1.22]{WW1}.
	
	\begin{lemma}\label{lem-rho}
		Let $\delta\in (0,1)$. Let $(\mathcal{F},\mathcal{P})_\delta$ be a given configuration. Let $0<\epsilon_2<\epsilon_1<1$, let $v<\epsilon_2$ and $C>0$.
		Suppose $\mathcal{P}(f)$ is $(\epsilon_1,\epsilon_2, C)$-two-ends, and let $\rho=\rho(f;v,C)$ be the scale given by Definition \ref{twoendsreduction}.
		Then $\rho\geq \delta^{\epsilon_1}$.
	\end{lemma}
	
	Next, we state a Katz-Tao subset selection lemma taken from \cite[Section 2]{orpabckt}.
	
	\begin{lemma}\label{lemma-KT} Let $s \geq 0$, $\delta\in (0,1)$, and $C \geq 1$. Let $P \subset [-1,1]^{d}$ be a Katz-Tao $(\delta,s,C)$-set, and let $\delta \leq \rho \leq 1$. Then $P$ contains a Katz-Tao $(\rho,s,C_{d})$-subset $P'$ of cardinality $|P'| \gtrsim_{d} (\delta/\rho)^{s}|P|/C$, where $C_{d} \geq 1$ is a constant depending only on $d$.  \end{lemma}
	
	We will also need the following pigeonholing lemma for a $\delta$-configuration.
	\begin{lemma}\label{refine1}
		Let $\delta\in(0,1)$ and let $(\mathcal{F},\mathcal{P})_\delta$ be a given configuration. There exist $M\geq1$, a sub-family $\mathbf{E}_M\subset \mathbf{E}_{\mathcal{F},\mathcal{P}}$, and a refinement $(\mathcal{F}',\mathcal{P}')_\delta$ of $(\mathcal{F},\mathcal{P})_\delta$ so that 
		\begin{enumerate}
			\item [(i)] $\mathcal{P}'(f)$ is a refinement of $\mathcal{P}(f)$ for all $f\in \mathcal{F}'$, where $\mathcal{P}'(f)= \mathbf{E}_M\cap \mathcal{P}(f)$.
			\item [(ii)] $|\{f\in\mathcal{F}: p\in \mathcal{P}'(f)\}|\sim M$ for all $p\in \mathbf{E}_M$.
			\item [(iii)] $M\sim \sum_{f\in \mathcal{F}}|\mathcal{P}'(f)|/|\mathbf{E}_M|\approx \sum_{f\in \mathcal{F}'}|\mathcal{P}'(f)|/|\mathbf{E}_{\mathcal{F}',\mathcal{P}'}|$.
		\end{enumerate}
	\end{lemma}
	\begin{proof}
		First, we pigeonhole an integer $M\in 2^{\mathbb N}$ and a sub-family $\mathbf{E}_M\subset \mathbf{E}_{\mathcal{F},\mathcal{P}}$ such that
		\begin{enumerate}
			\item $|\mathcal{F}(p)|:=|\{f\in\mathcal{F}: p\in\mathcal{P}(f)\}|\sim M$ for any $p\in\mathbf{E}_M$,
			\item $\sum_{p\in\mathbf{E}_M}|\mathcal{F}(p)|\gtrapprox \sum_{p\in\mathbf{E}_{\mathcal{F},\mathcal{P}}}|\mathcal{F}(p)|=\sum_{f\in\mathcal{F}}|\mathcal{P}(f)|$.
		\end{enumerate}
		Then we define $\mathcal{P}'(f):=\mathbf{E}_M\cap\mathcal{P}(f)$ for each $f\in\mathcal{F}$, thus $|\{f\in\mathcal{F}: p\in \mathcal{P}'(f)\}|\sim M$ for all $p\in \mathbf{E}_M$. Since $\sum_{f\in\mathcal{F}}|\mathcal{P}'(f)|=\sum_{p\in\mathbf{E}_M}|\mathcal{F}(p)|\gtrapprox \sum_{f\in\mathcal{F}}|\mathcal{P}(f)|$, we can pigeonhole a sub-family $\mathcal{F}'\subset \mathcal{F}$ such that $\mathcal{P}'(f)$ is a refinement of $\mathcal{P}(f)$ for all $f\in\mathcal{F}'$ and
		\[M|\mathbf{E}_{\mathcal{F}',\mathcal{P}'}|\gtrsim\sum_{f\in\mathcal{F}'}|\mathcal{P}'(f)|\gtrapprox \sum_{f\in\mathcal{F}}|\mathcal{P}(f)|\geq M|\mathbf{E}_M|\sim\sum_{f\in\mathcal{F}}|\mathcal{P}'(f)|,\]
		which implies that $|\mathbf{E}_M|\approx |\mathbf{E}_{\mathcal{F}',\mathcal{P}'}|$ and $(\mathcal{F}',\mathcal{P}')_\delta$ is a refinement of $(\mathcal{F},\mathcal{P})_\delta$. Then property (iii) follows easily from the inequality above.
	\end{proof}

	\begin{definition}[Katz-Tao constant]\label{gamma-Y-def}
		Let $\de\in(0,1]$ and $t\in [0,1]$. Let $(\cF,\cP)_\de$ be a $\de$-configuration. For each $f \in \cF$, define 
		\begin{equation}
			\label{gamma-y}
			\ga_{\cP,t}(f):=\sup_{\substack{ r\in[\de,1]\\ x\in \cup\cP(f)}}\Big(\frac{\de}{r}\Big)^{t}\cdot|\cP(f)\cap B(x,r)|_\delta,
		\end{equation}
		and define $\ga_{\cP,t}:=\sup_{f\in \cF}\ga_{\cP,t}(f)$. We note that $\ga_{\cP,t}$ depends on the scale $\delta$ which is always clear from the definition of the shading $\mathcal{P}$.
	\end{definition}
	
	The lemma below follows from the definition of $\gamma_{\mathcal{P},t}$, see \cite[Lemma 2.22]{WW2}.
	
	\begin{lemma}\label{lem-rdelta}
		Let $0<\de\leq r\leq1$ and let $f$ be a function in a transversal family. Let $\mathcal{P}(f)\subset \mathcal{D}_\delta$ and $\widetilde {\mathcal{P}}(f)=\mathcal{D}_r(\mathcal{P}(f))$. Suppose that each $r$-cube in $\widetilde {\mathcal{P}}(f)$ contains $\gtrsim d$ many $\de$-cubes in $\mathcal{P}(f)$ for some $d\in[1,r/\de]$.
		Then $\ga_{\widetilde {\mathcal{P}},t}(f)\lesssim(\de/r)^{-t}d^{-1}\ga_{\mathcal{P},t}(f)$.
	\end{lemma}
	
	Finally, we state the following result whose proof is the same as \cite[Lemma 2.23]{WW2}.
	\begin{lemma}\label{multi-scale-lem}
		For $\eta>0$, there exists $\delta_0=\delta_0(\eta)>0$ such that the following holds for all $\delta\in (0,\delta_0]$.
		
		Let $(\cF,\cP)_\de$ be a $\de$-configuration such that $\{\cP(f): f \in \cF\}$ is a set of uniform shadings and have a uniform branching function as in Lemma \ref{lem:uniform branching}.
		Let $\eta_0=\eta_0(\eta)=\eta^{2\eta^{-1}}$ and $t\in[0,1]$. Then there exist an $r\in[\delta^{1-\eta_0\eta^{-1}},1]$ and an $s\in[0,1]$ such that the following is true: 
		
		For each $f\in \cF$, let $\tilde{\cP}(f) := [\cP(f)]_r$, then we have
		\begin{enumerate}
			\item [\textup{(i)}] $\ga_{\tilde \cP, t}(f)\lesssim \ga_{\cP,t}(f)$; 
			\item [\textup{(ii)}] for each $J=u_Q(6\delta)\in [\cP(f)]_r$, let $S_J$ be the rescaling map taking $Q$ to $[-1,1)^2$, then $S_J(\cP(f)\cap J)$ is a $(\delta/r, s, (\delta/r)^{-9\eta})$-set, and
			$$\log_{1/\de}\Big(\frac{|\cP(f)|_{\de}}{|\cP(f)|_{r}}\Big)\leq (s+9\eta) \log_{1/\de}\big(\frac{r}{\delta}\big).$$
		\end{enumerate}
	\end{lemma}
	
	\section{Two-ends Furstenberg inequality}\label{sec3}
	
	In this section, we extend the two-ends Furstenberg inequality from \cite{WW1,WW2}, stated below, to transversal families, and the proof is considerably simplified. We refer the reader to Definition \ref{def:twoends} for the two-ends condition.
	
	\begin{thm}\label{thm-twoends}
		Let $t\in (0,2)$, $\mathfrak{T}\geq1$ and $\lambda\in [0,1]$. For any $\epsilon\in (0,1)$, there exists a small constant $\delta_0=\delta_0(t,\epsilon,\mathfrak{T})>0$ such that the following holds for any $\delta\in (0,\delta_0]$.
		
		Let $(\mathcal{F},\mathcal{P})_\delta$ be a given configuration. Here $\mathcal{F}\subset B_{C^2}(1)$ is $\mathfrak{T}$-transversal over $[-2,2]$. Assume that
		\begin{itemize}
			\item $\mathcal{F}$ is $\delta$-separated and a $(\delta,t)$-KT set;
			\item for each $f\in \mathcal{F}$, $\mathcal{P}(f)$ is $(\epsilon_1, \epsilon_2)$-two-ends and $\lambda$-dense.
		\end{itemize}
		Write $t^\ast=\{t,2-t\}$. Then,
		\begin{equation}\label{main--twoends}
			|\mathbf{E}_{\mathcal{F},\mathcal{P}}|\geq \delta^{\epsilon+t\epsilon_1/2}  \delta^{(t-1)/2} \ga_{\cP,t^\ast}^{-1/2}\cdot \lambda^{1/2}\sum_{f\in\mathcal{F}} |\mathcal{P}(f)|.
		\end{equation}
	\end{thm}

	As shown in \cite{WW2}, the Katz-Tao constant factor $\ga_{\cP,t^\ast}^{-1/2}$ is necessary, see the sharp examples given in \cite[Section 3]{WW2}. We next introduce two additional definitions for convenience.
	
	\begin{definition}[Vertical $r$-neighbourhood]\label{def:vertical} Let $I \subset \R$ be an interval, $f\in C^2(I)$, $r>0$. The \emph{vertical $r$-neighborhood of $f$} is the set
		\[\Gamma_f(r):=\{(x,y)\in I \times \R :|y-f(x)|< r\}.\]
	\end{definition}
	
	\begin{definition}\label{def:deltarho}
		Let $\delta\in (0,1]$, $s\geq 0$ and $\rho^\ast\in [\delta,1]$. We say a bounded subset $P\subset \R^2$ is a $(\delta,s,C;\rho^\ast)$-set if
		\begin{equation}
			|P\cap B(x,r)|_\delta\leq Cr^s |P|_\delta, \qquad x \in \R^2, \, \rho^\ast\leq r \leq 1. 
		\end{equation}
	\end{definition}

	The two-ends Furstenberg inequality can be reduced to the following Theorem.
	
	\begin{thm}\label{thm-reduction2}
		Let $t\in (0,2)$, $\mathfrak{T}\geq1$, $C\geq1$ and $\lambda\in [0,1]$. For any $\epsilon\in (0,1)$, there exist $\eta=\eta(\epsilon)>0$ and $C_{\mathfrak{T},t,\epsilon}>0$ such that the following holds for all $\delta\in (0,1)$.
		
		Let $(\mathcal{F},\mathcal{P})_\delta$ be a given configuration. Here $\mathcal{F}$ is $\mathfrak{T}$-transversal on $[-2,2]$. Assume that
		\begin{itemize}
			\item $\mathcal{F}$ is $\delta$-separated and a $(\delta,t)$-KT set;
			\item for each $f\in \mathcal{F}$, $\mathcal{P}(f)$ is uniform, $\lambda$-dense and a $(\delta,\epsilon^2,C;\rho^\ast)$-set for some $\rho\ast\in [\delta,\delta^\eta]$.
		\end{itemize}
		Write $t^\ast=\min\{t, 2-t\}$. Then
		\begin{equation}\label{equu-15}
			|\mathbf{E}_{\mathcal{F},\mathcal{P}}|\geq  C_{\mathfrak{T},t,\epsilon}\delta^{\epsilon}  C^{-\eta^{-2}} \delta^{(t-1)/2}\gamma_{\mathcal{P},t^\ast}^{-1/2} \lambda^{1/2}\sum_{f\in\mathcal{F}}|\mathcal{P}(f)|.
		\end{equation}
	\end{thm}
	
	\begin{remark}
		In the proof of Theorem \ref{thm-twoends}, we only need to apply Theorem \ref{thm-reduction2} when $\rho^\ast=\delta$.
	\end{remark}
	
	\begin{notation}
		In this section, we write $A \lessapprox B$, $A \gtrapprox B$, or $A\approx B$ to hide slowly growing functions of $\delta$ such as $\log(1/\delta)$ or $\bar{h}(|\log\delta|)$ for some function $\bar{h}$, arising from dyadic pigeonholing. We write $A \lesssim B$, $A \gtrsim B$, or $A\sim B$ to hide absolute constants and also the transversality constant $\mathfrak{T}$.
	\end{notation}

	\subsection{Proof of Theorem \ref{thm-twoends}}
	\begin{proof}[Proof of Theorem \ref{thm-twoends} using Theorem \ref{thm-reduction2}]
		It suffices to prove Theorem \ref{thm-twoends} when $\epsilon<\sqrt{\epsilon_2}$. By a dyadic pigeonholing argument, we may assume that $|\mathcal{P}(f)|\sim \lambda \delta^{-1}$ for all $f\in \mathcal{F}$. Applying Lemma \ref{lem:uniform branching}, we find a refinement $(\mathcal{F}_1,\mathcal{P}_1)_\delta$ of $(\mathcal{F},\mathcal{P})_\delta$ such that $\mathcal{P}_1(f)$ is a uniform refinement of $\mathcal{P}(f)$ for all $f\in\mathcal{F}_1$ and there is a uniform branching function for $\{\mathcal{P}_1(f)\}_{f\in \mathcal{F}_1}$.
		
		Since $\mathcal{P}(f)$ is $(\epsilon_1,\epsilon_2)$-two-ends, $\mathcal{P}_1(f)$ is $(\epsilon_1,\epsilon_2, K)$-two-ends for some $K\lessapprox 1$. Recalling Definition \ref{twoendsreduction} and Remark \ref{remk-1}, there exists $\rho(f):=\rho(f,\epsilon^2,K)\in[\delta,1]\cap 2^{-\N}$ for each $f\in\mathcal{F}_1$ such that
		\begin{itemize}
			\item [(i)] $|\mathcal{P}_1(f)|_{\rho(f)}\leq K^{-1} \rho(f)^{-\epsilon^2}$;
			\item [(ii)] for all $J\in [\mathcal{P}_1(f)]_{\rho(f)}$, $S_J(\mathcal{P}_1(f)\cap J)$ is a $(\delta/\rho(f),\epsilon^2,C_1)$-set for some $C_1\lessapprox1$.
		\end{itemize}
		Since $\epsilon^2<\epsilon_2$, by Lemma \ref{lem-rho} $\rho(f)\geq \delta^{\epsilon_1}$ for all $f\in\mathcal{F}_1$. By dyadic pigeonholing, we can further find a uniform $\rho\geq \delta^{\epsilon_1}$ and a refinement $\mathcal{F}_2$ of $\mathcal{F}_1$ such that $\rho(f)\sim \rho$ for all $f\in\mathcal{F}_2$. Let $\mathcal{P}_2=\mathcal{P}_1$ (i.e., $\mathcal{P}_2(f)=\mathcal{P}_1(f)$ for all $f\in\mathcal{F}_1$), then $(\mathcal{F}_2,\mathcal{P}_2)_\delta$ is a refinement of $(\mathcal{F},\mathcal{P})_\delta$ and for each $f\in\mathcal{F}_2$ we have
		\begin{itemize}
			\item [(i)] $|\mathcal{P}_2(f)|_{\rho}\lesssim K^{-1} \rho^{-\epsilon^2}\leq \rho^{-\epsilon^2}$;
			\item [(ii)] for all $J\in [\mathcal{P}_2(f)]_{\rho}$, $S_J(\mathcal{P}_2(f)\cap J)$ is a $(\delta/\rho,\epsilon^2,C_2)$-set for some $C_2\lessapprox1$.
		\end{itemize}
		In the sequel, we may also assume that $\rho<1/2$ since otherwise Theorem \ref{thm-twoends} follows from Theorem \ref{thm-reduction2} when $\rho^\ast=\delta$.

		Next, recall a basic construction from \cite[Appendix A]{OPY2025}. For each $Q\in\mathcal{D}_\rho(\mathbf{E}_{\mathcal{F}_2,\mathcal{P}_2})$, let \[\Gamma(Q)=\{\Gamma_{f|_{I_Q}}: \Gamma\cap Q\neq \emptyset, f\in\mathcal{F}_2\}.\] 
		Here $I_Q=\pi_x(Q)$ and $\Gamma_{f|_{I_Q}}$ means the graph of $f$ above $I_Q$. Two curve segments $\Gamma_{f|_{I_{Q}}},\Gamma_{g|_{I_{Q}}}$ are \emph{comparable} if $|f(x) - g(x)| \leq 3\delta$ for all $x \in I_{Q}$. Let $\mathbb{S}_{\rho,Q}$ be a maximal set of \emph{incomparable} curve segments in $\Gamma(Q)$ and write $\mathbb{S}_\rho:=\cup_Q \mathbb{S}_{\rho,Q}$. We recall the following two properties for incomparable segments whose proof can be found in \cite[Appendix A]{OPY2025}. 
		\begin{itemize}
			\item [(i)] Each $u\in \mathbb{S}_\rho$ belongs to at most two different $\mathbb{S}_{\rho,Q}$ since the functions $f \in \mathcal{F}$ are $1$-Lipschitz, and therefore any graph segment $\Gamma_{f|_{I_{Q}}}$ can intersect at most two (vertically aligned) $\rho$-cubes.
			
			\item [(ii)] For each $Q\in \mathcal{D}_\rho(\mathbf{E}_{\mathcal{F}_2,\mathcal{P}_2})$, each element of $\Gamma(Q)$ is contained in the vertical $3\delta$-neighborhood of one and at most $O(\mathfrak{T}^{4})$ segments in $\mathbb{S}_{\rho,Q}$.
		\end{itemize}

		For each $u\in\mathbb{S}_{\rho,Q}$, let $\mathcal{P}_{Q}(u):=\mathcal{P}_2(f_u)\cap Q$, where $f_u\in\mathcal{F}_2$ such that $\Gamma_{f|_{I_Q}}=u$. Since $\mathcal{P}_2(f_u)$ is uniform, $\lambda$-dense and $|\mathcal{P}_2(f)|_{\rho}\lesssim \rho^{-\epsilon^2}$,
		\begin{equation}\label{equa-16}
			|\mathcal{P}_{Q}(u)|\sim \frac{|\mathcal{P}_2(f_u)|}{|\mathcal{P}_2(f_u)|_\rho}\gtrapprox \lambda \delta^{-1}\rho^{\epsilon^2}=\lambda\rho^{-1+\epsilon^2} \cdot (\delta /\rho)^{-1}.
		\end{equation}
		Let $(x_{0},y_{0})$ be the lower left corner of $Q$ and recall the notation  \[f_{(x_{0},y_{0}),\rho}(x)=(f(\rho x+x_0)-y_0)/\rho.\] Lemma \ref{lem2} stated that the family 
		\begin{displaymath} \mathcal{F}_{Q} := \mathcal{F}_{(x_0,y_0),\rho}:=\{(f_u)_{(x_0,y_0),\rho}: u\in \mathbb{S}_{\rho,Q}\} \end{displaymath}
		is transversal with constant $(4\mathfrak{T} + 1)$ on the interval $[-2,2]$ (note that $[-2,2]\subset ([-2,2]-x_0)/\rho$ since $x_0\in [0,1]$ and $\rho<1/2$). It has been verified in \cite[Appendix A]{OPY2025} that $\mathcal{F}_{Q}$ is $\sim_{\mathfrak{T}} \delta/\rho$-separated. We aim to apply Theorem \ref{thm-reduction2}, but $\mathcal{F}_{Q}$ may not be a $(\delta/\rho,t)$-KT set. In the following, we use random selection to find a KT-subset $\mathcal{F}_Q''\subset \mathcal{F}_Q$. Before that, we need the following claim.
		
		\begin{claim}\label{claimm1}
			A subset $\mathcal{F}_Q''\subset \mathcal{F}_Q$ is a $(\delta/\rho,t)$-KT set if 
			\begin{equation}\label{equ-17}
				|\{u_f\in\mathbb{S}_{\rho,Q}'': u_f\subset u(2\tau)\}|\lessapprox (\tau/\delta)^t, \quad \forall u\in \mathbb{S}_{\rho,Q}'', \quad \forall \tau\in [\delta,\rho],
			\end{equation}
			where $u(2\tau)$ is the vertical $2\tau$-neighborhood of $u$ (recall Definition \ref{def:vertical}) and
			\[\mathbb{S}_{\rho,Q}'':=\{u\in \mathbb{S}_{\rho,Q}: f_{(x_{0},y_{0}),\rho} \in \mathcal{F}_Q''\}.\]
		\end{claim}
		\begin{proof}
			Let $r_0\in [\delta/\rho,1]$, then if $(f_{u_1})_{(x_0,y_0),\rho}$ and $(f_{u_2})_{(x_0,y_0),\rho}$ (here $u_1,u_2\in\mathbb{S}_{\rho,Q}''$) belong to the same $r_0$-ball, we have
			\[|(f_{u_1})_{(x_0,y_0),\rho}-(f_{u_2})_{(x_0,y_0),\rho}|_{L^\infty([-1,1])}\leq 2r_0 \Longrightarrow |f_{u_1}-f_{u_2}|_{L^\infty(I_Q)}\leq 2r_0\rho,\]
			which implies $u_2\subset u_1(2r_0\rho)$. Hence from \eqref{equ-17} we deduce that
			\[|\mathcal{F}''_Q\cap B_{r_0}|_{\delta/\rho}\leq \max_{u_1\in \mathbb{S}_{\rho,Q}''}|\{u_2\in \mathbb{S}_{\rho,Q}'':u_2\subset u_1(2r_0\rho)|\lessapprox \Big(\frac{r_0}{\delta/\rho}\Big)^t,\]
			which proves the claim.
			
		\end{proof}

		Now we start selecting the sub-family $\mathcal{F}_Q''\subset \mathcal{F}_Q$. Write 
		\[\mathcal{F}_2(Q)=\{f\in\mathcal{F}_2: \Gamma_f\cap Q\neq \emptyset\}\]
		and 
		\[\mathcal{F}_2(u)=\{f\in\mathcal{F}_2(Q): u_f\subset u(3\delta)\}, \quad u\in\mathbb{S}_{\rho,Q},\]
		then $\mathcal{F}_2(Q)=\cup_{u\in \mathbb{S}_{\rho,Q}}\mathcal{F}_2(u)$. Write $m(u)=|\mathcal{F}_2(u)|$. By dyadic pigeonholing, we can find $m_Q$ and $\mathbb{S}_{\rho,Q}'\subset \mathbb{S}_{\rho,Q}$ such that
		\begin{itemize}
			\item [(1)] $m(u)\sim m_Q$ for all $u\in \mathbb{S}_{\rho,Q}'$,
			\item [(2)] $m_Q\cdot |\mathbb{S}_{\rho,Q}'|\gtrapprox \sum_{u\in\mathbb{S}_{\rho,Q}} m(u)=\sum_{f\in \mathcal{F}_2(Q)}|\{u\in \mathbb S_{\rho,Q}: u_f\subset u(3\delta)\}|\sim |\mathcal{F}_2(Q)|$. This follows from the property (ii) above.
		\end{itemize}
		To pick a subset $\mathcal{F}_Q''$ of $\mathcal{F}_Q$, define
		\begin{equation}\label{equ-18}
			\sigma:=\sup_{\tau\in[\delta,\rho]}\Big(\frac{\delta}{\tau}\Big)^t\cdot \Big(\sup_{u\in \mathbb{S}'_{\rho,Q}}|\{u_1\in \mathbb{S}'_{\rho,Q}: u_1\subset u(2\tau)\}|\Big),
		\end{equation}
		then $\sigma\geq1$ (take $\tau=\delta$). For $\tau\in [\delta,\rho]$ and $u\in \mathbb{S}_{\rho,Q}'$, $\mathcal{F}_2(u;\tau):=\{f\in\mathcal{F}_2: u_f\subset u(2\tau)\}$ is contained in a ball with radius $\sim_\mathfrak{T} \tau/\rho$, see \cite[Claim A.15]{OPY2025} for the proof. Since $\mathcal{F}_2$ is a $(\delta,t)$-KT set, we have 
		\[|\{u_1\in \mathbb{S}'_{\rho,Q}: u_1\subset u(2\tau)\}|\cdot m_Q\lesssim_\mathfrak{T} |\mathcal{F}(u;\tau)|\lessapprox \Big(\frac{\tau}{\delta \rho}\Big)^t,\]
		where the first inequality follows from the property (ii) for incomparable segments above. This implies 
		\begin{equation}\label{eqqqq}
			\sigma \cdot m_Q\lessapprox \rho^{-t}.
		\end{equation}
		Let $\mathbb{S}_{\rho,Q}''\subset \mathbb{S}_{\rho,Q}'$ be a uniform random sample with probability $\sigma^{-1}$. Then, with high probability $|\mathbb{S}_{\rho,Q}''|\gtrapprox \sigma^{-1}|\mathbb{S}_{\rho,Q}'|\gtrsim_\mathfrak{T} m_Q\rho^t |\mathbb{S}_{\rho,Q}'|$ and
		\[|\{u_1\in \mathbb{S}_{\rho,Q}'': u_1\subset u(2\tau)\}|\lessapprox  \Big(\frac{\tau}{\delta}\Big)^t, \quad \forall ~u\in\mathbb{S}_{\rho,Q}''.\]
		Let $\mathcal{F}_Q''=\{(f_{u})_{(x_0,y_0),\rho}: u\in\mathbb{S}_{\rho,Q}''\}$, then by Claim \ref{claimm1} the inequality above shows that $\mathcal{F}_Q''$ is a $(\delta/\rho,t,C_\mathfrak{T})$-KT set for some $C_\mathfrak{T}\lessapprox_\mathfrak{T}1$.

		Recall $\mathcal{P}_{Q}(u):=\mathcal{P}_2(f_u)\cap Q$ for $u\in \mathbb{S}_{\rho,Q}$. For notational convenience, we define a new shading:
		\begin{equation}\label{equ-f}
			\mathcal{P}_Q''(u)=\left\{
			\begin{array}{ll}
				\mathcal{P}_Q(u), &~\text{if}~u\in \mathbb{S}_{\rho,Q}'',\\
				\emptyset, &~\text{otherwise}.
			\end{array}\right.
		\end{equation}
		Let $T_Q$ be the rescaling map taking $Q$ to $[-1,1]^2$. For $\bar{f}=(f_{u})_{(x_0,y_0),\rho}\in \mathcal{F}_Q''$, define
		\begin{displaymath} \overline{\mathcal{P}}_{Q} (\bar{f}) := \{T_Q(p):p\in \mathcal{P}_Q''(u)\} \subset \mathcal{D}_{\delta/\rho}. \end{displaymath}
		Recalling \eqref{equa-16}, we infer that $\overline{\mathcal{P}}_{Q}(\bar{f})$ is a uniform $(\delta/\rho,\epsilon^2, C_3)$-set ($C_3\lessapprox1$) and is $\lambda\rho^{-1+\epsilon^2}$-dense for all $\bar{f}\in\mathcal{F}_Q''$. Therefore, we are ready to apply Theorem \ref{thm-reduction2} to $(\mathcal{F}_Q'',\overline{\mathcal{P}}_{Q})_{\delta/\rho}$ with $\epsilon$ replaced by $\epsilon/2$. This gives $\eta=\eta(\epsilon)$ and $C_{\mathfrak{T},t,\epsilon}>0$ such that
		\begin{equation}\label{equ-19}
			|\mathbf{E}_{\mathcal{F}_Q'',\overline{\mathcal{P}}_{Q}}|\gtrapprox  C_{\mathfrak{T},t,\epsilon}(\delta/\rho)^{\epsilon/2}\cdot C_3^{-\eta^{-2}}(\delta/\rho)^{(t-1)/2} \cdot \gamma_{\overline{\mathcal{P}}_{Q},t^\ast}^{-1/2} \cdot (\lambda\rho^{-1+\epsilon^2})^{1/2}\sum_{\bar{f}\in\mathcal{F}_Q''}|\overline{\mathcal{P}}_{Q}(\bar{f})|.
		\end{equation}
		We may choose $\delta>0$ small enough so that $C_{\mathfrak{T},t,\epsilon}C_3^{-\eta^{-2}}>\delta^{\epsilon^2}$. After rescaling back, we get
		\begin{equation}\label{equu-19}
			\begin{split}
				|\mathbf{E}_{\mathcal{F},\mathcal{P}}\cap Q|\geq |\mathbf{E}_{\mathbb{S}_Q'',\mathcal{P}''_{Q}}|\gtrapprox & (\delta/\rho)^{\epsilon/2} \delta^{\epsilon^2}\cdot   (\delta/\rho)^{(t-1)/2} \gamma_{\mathcal{P}''_{Q},t^\ast}^{-1/2} \\& \cdot  (\lambda\rho^{-1+\epsilon^2})^{1/2}\sum_{u\in\mathbb{S}_Q''}|\mathcal{P}''_{Q}(u)|.
			\end{split}
		\end{equation}
		Recall $|\mathcal{F}_Q''|=|\mathbb{S}_{\rho,Q}''|\gtrapprox \sigma^{-1} |\mathbb{S}_{\rho,Q}'|$ and $m_Q\cdot |\mathbb{S}_{\rho,Q}'|\gtrapprox |\mathcal{F}_2(Q)|$ for each $u\in\mathbb{S}_{\rho,Q}'$. Since $|\mathcal{P}_2(f)\cap Q|$ is approximately constant for all $f\in \mathcal{F}_1$ and $Q$,
		\begin{equation}\label{equ-cubeinQ}
			\begin{split}
				\sum_{u\in\mathbb{S}_Q''}|\mathcal{P}''_{Q}(u)|&\gtrapprox \sigma^{-1} \sum_{u\in\mathbb{S}_Q'}|\mathcal{P}_{Q}(u)|\gtrapprox (\sigma m_Q)^{-1} \sum_{f\in\mathcal{F}_2(Q)}|\mathcal{P}_2(f)\cap Q|\\
				&\gtrapprox \rho^t \sum_{p\in \mathbf{E}_{\mathcal{F}_2,\mathcal{P}_2}\cap Q} |\mathcal{F}_2(p)|,
			\end{split}
		\end{equation}
		where we use \eqref{eqqqq} and $\mathcal{F}_2(p)=\{f\in \mathcal{F}_2:p\in \mathcal{P}_2(f)\}$. Substituting \eqref{equ-cubeinQ} into \eqref{equu-19} gives
		\begin{equation}\label{equu-22}
			\begin{split}
				|\mathbf{E}_{\mathcal{F},\mathcal{P}}\cap Q|&\gtrapprox (\delta/\rho)^{\epsilon/2} \delta^{\epsilon^2}\cdot   (\delta/\rho)^{(t-1)/2} \gamma_{\mathcal{P}''_{Q},t^\ast}^{-1/2}  \cdot  (\lambda\rho^{-1+\epsilon^2})^{1/2} \rho^t \sum_{p\in \mathbf{E}_{\mathcal{F}_2,\mathcal{P}_2}\cap Q} |\mathcal{F}_2(p)|\\
				&\geq \delta^{\epsilon/2+\epsilon^2}\cdot \rho^{t/2} \delta^{(t-1)/2}  \gamma_{\mathcal{P}''_{Q},t^\ast}^{-1/2} \lambda^{1/2}\sum_{p\in \mathbf{E}_{\mathcal{F}_2,\mathcal{P}_2}\cap Q} |\mathcal{F}_2(p)|.
			\end{split}
		\end{equation}
		By using $\rho\geq \delta^{\epsilon_1}$ and $\gamma_{\mathcal{P}''_{Q},t^\ast}\leq \gamma_{\mathcal{P},t^\ast}$, we deduce by summing up all $Q\in\mathcal{D}_\rho(\mathbf{E}_{\mathcal{F}_2,\mathcal{P}_2})$ that
		\begin{equation}\label{equ-20}
			\begin{split}
				|\mathbf{E}_{\mathcal{F},\mathcal{P}}|&\gtrapprox  \delta^{\epsilon/2+\epsilon^2} \delta^{t\epsilon_1/2}\cdot \delta^{(t-1)/2}  \gamma_{\mathcal{P}''_{Q},t^\ast}^{-1/2} \lambda^{1/2}\sum_{p\in \mathbf{E}_{\mathcal{F}_2,\mathcal{P}_2}} |\mathcal{F}_2(p)|\\
				&\gtrapprox \delta^{\epsilon/2+\epsilon^2} \delta^{t\epsilon_1/2} \cdot\delta^{(t-1)/2} \gamma_{\mathcal{P},t^\ast}^{-1/2}  \lambda^{1/2}\sum_{f\in\mathcal{F}}|\mathcal{P}(f)|.
			\end{split}
		\end{equation}
		Finally, by choosing $\delta>0$ sufficiently small, we easily conclude \eqref{main--twoends}. 
	\end{proof}

	\subsection{A scale selection lemma}
	
	The remaining task of this section is to prove Theorem \ref{thm-reduction2}. Our simplified proof for Theorem \ref{thm-reduction2} is mainly due to the following scale selection lemma. The lemma itself is an application of the Furstenberg set estimates in Theorem \ref{thm.furstenberg} and \ref{thm-incidence-KT} stated below.
	
	\begin{lemma}\label{lem-main}
		Let $s\in (0,1]$, $t\in(0,2]$ and $\mathfrak{T}\geq1$. Then, for every $\nu > 0$, there exist $\eta=\eta(\nu,s)>0$ and $\delta_0=\delta_0(s,t,\nu,\mathfrak{T})>0$ such that the following holds for all $\delta \in (0,\delta_{0}]$. 
		
		Let $(\mathcal{F},\mathcal{P})_\delta$ be a given configuration. Here $\mathcal{F}\subset B_{C^2}(1)$ is $\mathfrak{T}$-transversal over $[-2,2]$. Assume that:
		\begin{itemize}
			\item $\mathcal{F}$ is uniform and $\delta$-separated;
			\item for all $f\in\mathcal{F}$, $\mathcal{P}(f)$ is a $(\delta,s,\delta^{-\eta})$-set with $|\mathcal{P}(f)|\lesssim\delta^{-s-\eta}$;
			\item recall $\mathcal{F}(p)=\{f\in\mathcal{F}:p\in \mathcal{P}(f)\}$, then we have \begin{equation}\label{eq.average}|\mathcal{F}(p)|\lessapprox \sum_{f\in \mathcal{F}}|\mathcal{P}(f)|/|\mathbf{E}_{\mathcal{F},\mathcal{P}}|.
			\end{equation}
		\end{itemize}
		Write
		\begin{equation}\label{equ-f}
			f(s,t)=\left\{
			\begin{array}{lll}
				0, \quad &\text{if}~~t\in(0,s],\\
				(s-t)/2, &~\text{if}~t\in[s,2-s],\\
				s-1, &~\text{if}~t\in [2-s,2].
			\end{array}\right.
		\end{equation}
		Then there exists a scale $\Delta\in [\delta,1]$ such that the following is true:
		\begin{itemize}
			\item [(i)] We have
			\begin{equation}\label{equ-multiplicity}
				|\mathcal{F}(p)|\leq \delta^{-\nu} \Big(\frac{\delta}{\Delta}\Big)^{-O(1)}\delta^{f(s,t)}.
			\end{equation}
			
			\item [(ii)] If $\Delta \geq \delta^{1-\sqrt{\eta}}$, then there is a refinement $\mathcal{F}^\ast\subset \mathcal{F}$ and an integer $d\geq1$ such that for all $f\in\mathcal{F}^\ast$ there is a sub-family $\mathcal{Q}(f)\subset \mathcal{D}_\Delta\cap \Gamma_f(O(\Delta))$ with $|\mathcal{Q}(f)|\sim d$ satisfying
			\begin{equation}\label{equ-largecardi}
				|\mathbf{E}_{\mathcal{F},\mathcal{P}}\cap Q|\geq \Big(\frac{\delta}{\Delta}\Big)^\nu \delta^{-2s+f(s,t)}\cdot \Delta^{s-f(s,t)}\cdot d^{-1}, \quad\forall Q\in\mathcal{Q}(f).
			\end{equation}
			
		\end{itemize}
	\end{lemma}

	Before proving this lemma, we first recall the following Furstenberg set estimate for transversal families which was established in \cite[Theorem 1.11]{OPY2025}, see also \cite{orpshabc,renwang} for the solution of Furstenberg set conjecture in $\R^2$.
	
	\begin{thm}\label{thm.furstenberg} Let $s \in (0,1]$ and $t \in [0,2]$. Then, for every $\mathfrak{T},\nu > 0$, there exist $\eta,\delta_{0} > 0$ such that the following holds for all $\delta \in 2^{-\N} \cap (0,\delta_{0}]$.
		
		Let $\mathcal{F} \subset B_{C^{2}}(1)$ be a non-empty $\mathfrak{T}$-transversal family over $[-2,2]$. Assume that $\mathcal{F}$ is a $(\delta,t,\delta^{-\eta})$-set. For each $f \in \mathcal{F}$, assume that $\mathcal{P}(f)$ is a non-empty $(\delta,s,\delta^{-\eta})$-set of dyadic $\delta$-squares which intersect the graph $\Gamma_{f}$, are contained in $[-1,1]^{2}$, and satisfy $|\mathcal{P}(f)| \geq M$ for some $M \in \N$ independent of $f$. Then
		\begin{equation}\label{form35} |\mathbf{E}_{\mathcal{F},\mathcal{P}}| \geq \delta^{\nu} \cdot \min\{\delta^{-t},\delta^{-(s + t)/2},\delta^{-1}\} \cdot M.\end{equation} \end{thm}
	
	\begin{remark}
		When $t\in [0,s]$ or $t\in [2-s,2]$, we have $\eta=\eta(\nu)$ and $\delta_0=\delta_0(\mathfrak{T},\nu)$. When $t\in (s,2-s)$, we have $\eta=\eta(\nu,s)$ and $\delta_0=\delta_0(\mathfrak{T},s,t,\nu)$.
	\end{remark}

	Below is the analogue of \cite[Theorem 1.4]{OS2026} for transversal families; more precisely, it is a dual formulation of that theorem in our setting. The proof uses multi-scale decomposition which has also been established for transversal families in \cite[Section 5]{OPY2025}, so here we shall not repeat the proof. We note estimate \eqref{equ-incd-KT} under the Katz-Tao condition of $\mathcal{F}$ is sharp, see examples given in \cite[Section 5]{OS2026}. 
	
	\begin{thm}\label{thm-incidence-KT}
		Let $s\in (0,1]$, $t\in [s,2-s]$ and $\mathfrak{T}\geq1$. For each $\nu \in (0,1)$, there exist small constants $\eta=\eta(\nu,s)>0$ and $\delta_0=\delta_0(s,t,\nu,\mathfrak{T})>0$ such that the following holds for all $\delta \in (0, \delta_0]$.
		
		Let $(\mathcal{F},\mathcal{P})_\delta$ be a given configuration, where $\mathcal{F}\subset B_{C^2}(1)$ is $\mathfrak{T}$-transversal. Assume that
		\begin{itemize}
			\item $\mathcal{F}$ is $\delta$-separated and a $(\delta,t,\delta^{-\eta})$-KT set;
			\item for each $f\in \mathcal{F}$, there is a $(\delta,s,\delta^{-\eta})$-set $\mathcal{P}(f)\subset \mathcal{D}_\delta$ with $|\mathcal{P}(f)|\geq M \geq1$.
		\end{itemize}
		Then 
		\begin{equation}\label{equ-incd-KT}
			|\mathbf{E}_{\mathcal{F},\mathcal{P}}|\geq \delta^\nu M |\mathcal{F}|^{(s+t)/2t}.
		\end{equation}
	\end{thm}

	The following is a direct corollary of Theorem \ref{thm-incidence-KT} and will be used in the proof of Theorem \ref{thm-reduction2}.
	
	\begin{cor}\label{cor2}
		Let $s \in (0,1]$, $t \in [0,2]$ and $\mathfrak{T}\geq1$. Then, for every $\nu > 0$, there exist small constants $\eta=\eta(\nu,s)>0$ and $\delta_0=\delta_0(s,t,\nu,\mathfrak{T})>0$ such that the following holds for all $\delta \in (0,\delta_{0}]$.
		
		Let $\mathcal{F} \subset B_{C^{2}}(1)$ be a $\delta$-separated $\mathfrak{T}$-transversal family. Assume that $\mathcal{F}$ is a $(\delta,t,\delta^{-\eta})$-KT set. For each $f \in \mathcal{F}$, assume that $\mathcal{P}(f)$ is a non-empty $(\delta,s,\delta^{-\eta})$-set of dyadic $\delta$-squares which intersect the graph $\Gamma_{f}$, are contained in $[-1,1]^{2}$, and satisfy $\lambda \delta^{-1}\sim |\mathcal{P}(f)| \leq \delta^{-s-\eta}$. Then,
		\begin{equation}\label{uq35} |\mathbf{E}_{\mathcal{F},\mathcal{P}}| \geq \delta^{\nu} \cdot  \delta^{(t-1)/2}\gamma_{\mathcal{P},t^\ast}^{-1/2}\lambda^{1/2}\sum_{f\in \mathcal{F}} |\mathcal{P}(f)|.\end{equation}
	\end{cor}
	
	\begin{proof}
		When $t\in (0,s]$, $\mathcal{F}$ is also a $(\delta,s,\delta^{-\eta})$-KT set. Applying Theorem \ref{thm-incidence-KT} with $t=s$ and $\nu$ replaced by $\nu/2$, we get
		\begin{equation}\label{equ-es1}
			|\mathbf{E}_{\mathcal{F},\mathcal{P}}|\gtrsim \delta^{\nu/2} \lambda \delta^{-1} |\mathcal{F}| \geq \delta^{\nu/2+3\eta/2} \delta^{(s-t)/2} \delta^{(t-1)/2}\lambda^{1/2}\sum_{f\in\mathcal{F}}|\mathcal{P}(f)|,
		\end{equation}
		where we also use $\lambda \delta^{-1}\gtrsim \delta^{\eta-s}$ by the $(\delta,s,\delta^{-\eta})$-condition of $\mathcal{P}(f)$.
		
		When $t\in [s,2-s]$, Theorem \ref{thm-incidence-KT} (with $\nu$ replaced by $\nu/2$) gives
		\begin{equation}\label{equ-es2}
			\begin{split}
				|\mathbf{E}_{\mathcal{F},\mathcal{P}}| &  \gtrsim \delta^{\nu/2}  |\mathcal{F}|^{s/2t-1/2}\cdot \lambda \delta^{-1} |\mathcal{F}| \geq \delta^{\nu/2}  \delta^{(t-s)/2+s-1+(1-s)/2+3\eta/2}\cdot \lambda^{1/2}\sum_{f\in\mathcal{F}}|\mathcal{P}(f)|\\
				&=\delta^{\nu/2+3\eta/2}  \delta^{(t-1)/2}\cdot \lambda^{1/2}\sum_{f\in\mathcal{F}}|\mathcal{P}(f)|,
			\end{split}
		\end{equation}
		where we use $|\mathcal{F}|\leq \delta^{-\eta-t}$ and $\lambda \delta^{-1}\gtrsim \delta^{\eta-s}$.
		
		When $t \in (2-s,2]$, we randomly choose a subset $\mathcal{F}'$ of $\mathcal{F}$ with probability $\delta^{s+t-2}$. Then with high probability, $\mathcal{F}'$ is a $(\delta,2-s,\delta^{-2\eta})$-KT set and $|\mathcal{F}'|\gtrsim \delta^{s+t-2}|\mathcal{F}|$. By applying Theorem \ref{thm-incidence-KT} to $(\mathcal{F}',\mathcal{P)}$ with $\nu$ replaced by $\nu/2$, we get
		\begin{equation}
			\begin{split}
				|\mathbf{E}_{\mathcal{F},\mathcal{P}}| &  \gtrsim \delta^{\nu/2} \lambda \delta^{-1} |\mathcal{F}'|^{1/(2-s)}\gtrsim \delta^{\nu/2} \lambda \delta^{t/(2-s)-2}|\mathcal{F}|^{1/(2-s)}.
			\end{split}
		\end{equation}
		By similar calculations as the former two cases using $|\mathcal{F}|\leq \delta^{-\eta-t}$ and $\lambda \delta^{-1}\gtrsim \delta^{\eta-s}$, we get
		\begin{equation}\label{equ-es3}
			|\mathbf{E}_{\mathcal{F},\mathcal{P}}|\geq \delta^{\nu/2+5\eta} \delta^{1-(s+t)/2} \cdot \delta^{(t-1)} \lambda^{1/2}\sum_{f\in\mathcal{F}}|\mathcal{P}(f)|.
		\end{equation}
		
		Recall Definition \ref{gamma-Y-def}. Since $|\mathcal{P}(f)|\in [\delta^{-s+\eta},\delta^{-s-\eta}]$, we have 
		\begin{equation}\label{equ-lbg}
			\gamma_{\mathcal{P},t^\ast}\geq \delta^{10\eta} \cdot \delta^{\min\{t-s,0,2-t-s\}}.
		\end{equation}
		Therefore, by combining \eqref{equ-es1}, \eqref{equ-es2}, \eqref{equ-es3} and using \eqref{equ-lbg}, we infer
		\[|\mathbf{E}_{\mathcal{F},\mathcal{P}}|\geq \delta^{\nu+O(\eta)} \cdot \lambda^{1/2} \delta^{(t-1)/2}\gamma_{\mathcal{P},t^\ast}^{-1/2}\sum_{f\in \mathcal{F}} |\mathcal{P}(f)|.\]
		Finally, we conclude \eqref{uq35} by choosing $O(\eta)<\nu/2$.
		
	\end{proof}

	\begin{proof}[Proof of Lemma \ref{lem-main}]
		Write
		\[u:= \begin{cases}
			s & \text{if } s \geq t\\
			t & \text{if } s \leq t \leq 2-s\\
			2-s & \text{otherwise.}
		\end{cases}\]
		Hence $u\in [s,2-s]$. Given $\nu > 0$, the small parameters $\delta, \eta$ will be determined later.

		First, choose $\Delta\in [\delta,1]$ so that the branching function of $\cF$ lies above the line with slope $u$ at the point corresponding to scale $\Delta$. In other words, for each $\bF \in \dy \Delta \cF$, write $\cF[\bF]:=\mathcal{F}\cap \mathbf{F}$, then 
		\begin{itemize}
			\item [(G1)] \phantomsection \label{G1} the rescaled family $T_\mathbf{F}(\cF[\bF])$ is a $(\de/\Delta, u, C_1)$-set for some $C_1\lessapprox1$, where $T_\mathbf{F}(f)=(f-f_\mathbf{F})/\Delta$ for a fixed $f_\mathbf{F}\in \cF[\bF]$.
			\item [(G2)] \phantomsection \label{G1} $\cF_\Delta:=\mathcal{D}_\Delta(\mathcal{F})$ is a $(\Delta,u,C_2)$-KT-set for some $C_2\lessapprox1$. 
		\end{itemize}
		By \cite[Lemma 2.8]{DW2024}, we are able to find a subset of each $\cF[\bF]$ of size $\approx (\Delta/\de)^u$. After replacing each $\cF[\bF]$ by this subset and using also KT-condition of $\mathcal{F}_\Delta$, we have found a subset $\cF^\ast \subset \cF$ which is a $(\de,u,C_3)$-KT set for some $C_3\lessapprox 1$. Since $|\mathcal{F}[\mathbf{F}]|\lesssim(\Delta/\delta)^2$ by Lemma \ref{lem.indentify}, we infer
		\[|\mathcal{F}^\ast|\approx \frac{(\Delta/\delta)^u}{|\mathcal{F}[\mathbf{F}]|}\cdot |\mathcal{F}|\gtrsim \Big(\frac{\delta}{\Delta}\Big)^{2-u}|\mathcal{F}|.\]
		Choose $\delta,\eta$ small enough so that Theorem \ref{thm-incidence-KT} is applicable to $(\mathcal{F}^\ast,\mathcal{P})$ with $\nu$ replaced by $\nu/2$, then we get
		\begin{equation}
			\begin{split}
				|\mathbf{E}_{\mathcal{F},\mathcal{P}}|&\geq \delta^{\nu/2} \delta^{\eta-s} |\mathcal{F}^\ast|^{(s+u)/2u}=\delta^{\nu/2+2\eta}  |\mathcal{F}^\ast|^{s/2u-1/2} \cdot \delta^{-\eta-s}|\mathcal{F}^\ast|\\
				&\gtrapprox \delta^{\nu/2+2\eta}  \delta^{(u-s)/2} \cdot \Big(\frac{\delta}{\Delta}\Big)^{2-u}\sum_{f\in \mathcal{F}}|\mathcal{P}(f)|,
			\end{split}
		\end{equation}
		where we use $|\mathcal{F}^\ast|\lessapprox \delta^{-u}$ and $|\mathcal{F}^\ast|\gtrapprox (\tfrac{\delta}{\Delta})^{2-u}|\mathcal{F}|$. Since $|\mathcal{F}(p)|\lessapprox \sum_{f\in \mathcal{F}}|\mathcal{P}(f)|/|\mathbf{E}_{\mathcal{F},\mathcal{P}}|$, we infer
		\[|\mathcal{F}(p)|\lessapprox \delta^{-\nu/2-2\eta} \delta^{f(s,t)} \cdot \Big(\frac{\Delta}{\delta}\Big)^{2-u},\]
		which proves property (i) once we choose $8\eta<\nu$ and $\delta$ small enough.

		Property (ii) will be deduced from Theorem \ref{thm.furstenberg}. From now on we assume $\Delta\geq \delta^{1-\sqrt{\eta}}$. Up to a pigeonholing argument, we may assume that each $\cP(f)$ is uniform. By using Lemma \ref{lem:uniform branching},  we can find a uniform refinement $\mathcal{F}'$ of $\cF$ such that there is a uniform branching function for $\{\cP(f)\}_{f\in \mathcal{F}'}$. 
		
		Next, for each $\mathbf{F}\in \mathcal{D}_\Delta(\mathcal{F}')$, we pigeonhole a cube family $\mathcal{Q}_\mathbf{F}\in \mathcal{D}_\Delta$ such that
		\begin{itemize}
			\item [(H1)] \phantomsection \label{H1} $|\mathbf{E}_{\mathcal{F}'[\mathbf{F}],\mathcal{P}}\cap Q|$ are constant for all $Q\in \mathcal{Q}_\mathbf{F}$;
			\item [(H2)] \phantomsection \label{H1} if $\mathcal{P}'(f):=\mathcal{P}(f)\cap (\cup \mathcal{Q}_\mathbf{F})$, $(\mathcal{F}'[\mathbf{F}], \mathcal{P}')$ is a refinement of $(\mathcal{F}'[\mathbf{F}], \mathcal{P})$.
		\end{itemize}
		Let $d_\mathbf{F}=|\mathcal{Q}_\mathbf{F}|$. By another pigeonholing, we find a sub-family $\mathcal{G}_\Delta\subset \mathcal{D}_\Delta(\mathcal{F}')$ and a uniform $d\geq1$ such that $d_\mathbf{F}\sim d$ for all $\mathbf{F} \in \mathcal{G}_\Delta$. Now we define $\mathcal{F}'':=\bigsqcup_{\mathbf{F}\in \mathcal{G}_\Delta} \mathcal{F}'[\mathbf{F}]$, then $\mathcal{F}''$ is a refinement of $\mathcal{F}$. We may also choose a uniform refinement $\mathcal{F}^\ast \subset \mathcal{F}''$ such that $|\mathcal{P}'(f)|\gtrapprox |\mathcal{P}(f)|$ for all $f\in \mathcal{F}^\ast$. In addition, we take a uniform refinement $\mathcal{P}^\ast(f)$ of $\mathcal{P}'(f)$ for each $f\in \mathcal{F}^\ast$.

		For each $\bF \in \mathcal{D}_\Delta(\mathcal{F}^\ast)$, since $\cF^\ast[\bF]:=\mathcal{F}^\ast\cap\mathbf{F}$ is a refinement of $\cF[\bF]$, we infer from property \nref{G1} that $T_\mathbf{F}(\cF^\ast[\bF])$ is a $(\de/\Delta,u, C)$-set for some $C\lessapprox1$. Let $M\in 2^{\N}$ such that 
		\[|{\cP^\ast(f)}|_{\delta/\Delta}\approx|{\cP(f)}|_{\delta/\Delta} \approx M, \quad f\in \mathcal{F}^\ast.\]
		We need to use the same construction as in the proof of Theorem \ref{thm-twoends}. For each $Q\in\mathcal{D}_{\delta/\Delta}(\mathbf{E}_{\cF^\ast[\bF],\cP^\ast})$, let \[\Gamma(Q)=\{\Gamma_{f|_{I_Q}}: \Gamma\cap Q\neq \emptyset, f\in\cF^\ast[\bF]\}.\] 
		Here $I_Q=\pi_x(Q)$ and $\Gamma_{f|_{I_Q}}$ means the graph of $f$ above $I_Q$. Two curve segments $\Gamma_{f|_{I_{Q}}},\Gamma_{g|_{I_{Q}}}$ are \emph{comparable} if $|f(x) - g(x)| \leq 3\delta$ for all $x \in I_{Q}$. Let $\mathbb{S}_{Q}$ be a maximal set of \emph{incomparable} curve segments in $\Gamma(Q)$ and write $\mathbb{S}:=\cup_Q \mathbb{S}_{Q}$. Recall the following properties for incomparable segments. 
		\begin{itemize}
			\item  Each $u\in \mathbb{S}$ belongs to at most two different $\mathbb{S}_{Q}$.
			
			\item For each $Q\in \mathcal{D}_{\delta/\Delta}(\mathbf{E}_{\cF^\ast[\bF],\cP^\ast})$, each segment in $\Gamma(Q)$ is contained in the vertical $3\delta$-neighborhood of one and at most $O(\mathfrak{T}^{4})$ segments in $\mathbb{S}_{Q}$.
			
			\item $\{u(6\delta):u\in \mathbb{S}\}$ form a curved tube cover for $\mathbf{E}_{\cF^\ast[\bF],\cP^\ast}$.
		\end{itemize}
		In the following, for $u\in \mathbb{S}_Q$, we use $f_u$ to denote the function such that $u=\Gamma_{f_u|_{I_Q}}$. Now for each $f\in \cF^\ast[\bF]$ and $Q\in \mathcal{D}_{\delta/\Delta}(\mathcal{P}^\ast(f))$, choose one $u_{Q,f}\in \mathbb{S}_Q$ (at most $O(1)$ choices) such that $\mathcal{P}^\ast(f)\cap Q\subset u_{Q,f}(6\delta)$ and define
		\[\widetilde{\mathcal{P}}(f):=\{u_{Q,f}(6\delta):Q\in \mathcal{D}_{\delta/\Delta}(\mathcal{P}^\ast(f))\}.\]
		Then $\widetilde{\mathcal{P}}(f)$ is a $(\de/\Delta,s, \de^{-O(\eta)})$-set since $\mathcal{P}^\ast(f)$ is a uniform $(\delta,s,\delta^{-O(\eta)})$-set. We abuse notation to define 
		\[T_\mathbf{F}(x,y)=\Big(x,\frac{y-f_\mathbf{F}(x)}{\Delta}\Big), \quad (x,y)\in \R^2.\]
		Our plan is to apply Theorem \ref{thm.furstenberg} at scale $\delta/\Delta$ after rescaling the configuration $(\cF^\ast[\bF], \widetilde{\mathcal{P}})$ by $T_\mathbf{F}$. Since $\mathcal{F}\subset B_{C^2}(1)$, for each $J\in \widetilde{\mathcal{P}}(f)$, $T_\mathbf{F}(J)$ is roughly a cube of side length $\sim \delta/\Delta$. We may need to choose one cube in $\mathcal{D}_{\delta/\Delta}(T_\mathbf{F}(J))$ so that $T_\mathbf{F}(\widetilde{\mathcal{P}}(f))$ becomes a $\delta/\Delta$-cube family. Here the details are skipped since one can check the same process in \cite[Appendix A]{OPY2025}.
		Let $\eta_1, \delta_1$ be the constants determined by Theorem \ref{thm.furstenberg}. Since $\Delta\geq \delta^{1-\sqrt{\eta}}$, we take $\delta<\delta_1^{1/\sqrt{\eta}}$ (hence $\delta/\Delta<\delta_1$) and $O(\eta)<\eta_1$. Recall $u\in [s,2-s]$, then we are able to apply Theorem \ref{thm.furstenberg} with $\nu$ replaced by $\nu/2$ to obtain
		\begin{equation}\label{eq-lbofP}
			|\mathbf{E}_{\mathcal{F}^\ast[\mathbf{F}],\widetilde{\mathcal{P}}}| \gtrapprox (\delta/\Delta)^{\nu/2} M (\Delta/\de)^{(s+u)/2},
		\end{equation}
		where $|\mathbf{E}_{\mathcal{F}^\ast[\mathbf{F}],\widetilde{\mathcal{P}}}|=|\cup_{f\in \mathcal{F}^\ast[\mathbf{F}]}\widetilde{\mathcal{P}}(f)|$ denotes the number of $\sim \delta\times \delta/\Delta$ curved tubes.
		
		We claim that each $\delta$-cube $p\in \mathbf{E}_{\mathcal{F}^\ast[\mathbf{F}],\mathcal{P}^\ast}$ belongs to one and at most $\lesssim_\mathfrak{T}1$ curved tubes in $\mathbf{E}_{\mathcal{F}^\ast[\mathbf{F}],\widetilde{\mathcal{P}}}$. To see this, let $p\subset Q\in \mathcal{D}_{\delta/\Delta}(\mathcal{P}^\ast(f))$ and for $u,v\in \mathbb{S}_Q$ such that $\|f_u-f_v\|_{C^2([-2,2])}\leq 2\Delta$, we infer
		\begin{displaymath} d_{V}(u,v) := \min_{x\in I_{Q}}|f_u(x)-f_v(x)|>\delta. \end{displaymath}
		Indeed, if $|f_u(x_{0}) - f_v(x_{0})| \leq \delta$ for some $x_{0} \in I_{Q}$, then $|f_u(x) - f_v(x)| \leq 3\delta$ for all $x \in I_{Q}$ by using the mean value theorem, contradicting the hypothesis that $u,v \in \mathbb{S}_{Q}$ are incomparable. This means for any $u\in\mathbb{S}_Q$ such that $p\subset u(6\delta)$, we have
		\begin{equation}\label{eq-bdednum}
			|\{f_v\in B_{C^2}(f_u,\Delta): p\subset v(6\delta) \text{ with } v\in\mathbb{S}_Q\}| \lesssim1.
		\end{equation}
		Moreover, since $\mathcal{F}^\ast[\mathbf{F}]$ can be covered by $O_\mathfrak{T}(1)$ many $\Delta$-balls, we have finished the proof of the claim.

		Since $|\mathcal{P}^\ast(f)\cap J| \gtrapprox \de^{O(\eta)} \de^{-s}/M$ for each $J\in \widetilde{\mathcal{P}}(f)$, we infer by using \eqref{eq-lbofP} and the claim above that
		\begin{equation}\label{equ-lbd}
			\begin{split}
				|\mathbf{E}_{\mathcal{F}^\ast[\mathbf{F}],\mathcal{P}^\ast}|&\gtrsim\sum_{p\in \mathbf{E}_{\mathcal{F}^\ast[\mathbf{F}],\mathcal{P}^\ast}} |\{J\in \mathbf{E}_{\mathcal{F}^\ast[\mathbf{F}],\widetilde{\mathcal{P}}}: p\subset J\}|=\sum_{J\in \mathbf{E}_{\mathcal{F}^\ast[\mathbf{F}],\widetilde{\mathcal{P}}}}|\mathbf{E}_{\mathcal{F}^\ast[\mathbf{F}],\mathcal{P}^\ast}\cap J|\\
				&\gtrapprox (\delta/\Delta)^{\nu/2} \de^{O(\eta)-s} (\Delta/\de)^{(s+u)/2}.
			\end{split}
		\end{equation}
		For any $f\in \mathcal{F}^\ast$, let $\mathbf{F}\in\mathcal{D}_{\Delta}(\mathcal{F})$ such that $f\in\mathcal{F}^\ast[\mathbf{F}]$. We take $\mathcal{Q}(f)=\mathcal{Q}_\mathbf{F}$. By \eqref{equ-lbd} and property \nref{H1}, for any $Q\in \mathcal{Q}(f)$ we have 
		\[|\mathbf{E}_{\mathcal{F},\mathcal{P}} \cap Q|\geq |\mathbf{E}_{\mathcal{F}'[\mathbf{F}],\mathcal{P}} \cap Q| \gtrapprox (\delta/\Delta)^{\nu/2} \de^{O(\eta)-s} (\Delta/\de)^{(s+u)/2}d^{-1}.\]
		Finally, note $f(s,t)=(s-u)/2$, then by taking $O(\eta)<\nu^4$ and $\delta>0$ small enough, we conclude \eqref{equ-largecardi} and finish the proof.
		
	\end{proof}
	
	\begin{remark}\label{re-portion}
		Our proof shows that for any $f\in \mathcal{F}^\ast$, the $\Delta$-cube family $\mathcal{Q}(f)$ actually contains a proportion $\gtrapprox1$ of the shading $\mathcal{P}(f)$, in the sense that
		\[|\mathcal{P}(f)\cap (\cup \mathcal{Q}(f))|\gtrapprox |\mathcal{P}(f)|.\]
	\end{remark}

	\subsection{Proof of Theorem \ref{thm-reduction2}} We restate the theorem:
	\begin{thm}\label{thm-reduction3}
		Let $t\in (0,2)$, $\mathfrak{T}\geq1$, $C\geq1$ and $\lambda\in [0,1]$. For any $\epsilon\in (0,1)$, there exist $\eta=\eta(\epsilon)>0$ and $C_{\mathfrak{T},t,\epsilon}>0$ such that the following holds for all $\delta\in (0,1)$.
		
		Let $(\mathcal{F},\mathcal{P})_\delta$ be a given configuration. Here $\mathcal{F}$ is $\mathfrak{T}$-transversal on $[-2,2]$. Assume that
		\begin{itemize}
			\item $\mathcal{F}$ is $\delta$-separated and a $(\delta,t)$-KT set;
			\item for each $f\in \mathcal{F}$, $\mathcal{P}(f)$ is uniform, $\lambda$-dense and a $(\delta,\epsilon^2,C;\rho^\ast)$-set for some $\rho^\ast\in [\delta,\delta^\eta]$.
		\end{itemize}
		Write $t^\ast=\min\{t, 2-t\}$. Then
		\begin{equation}\label{equu-15}
			|\mathbf{E}_{\mathcal{F},\mathcal{P}}|\geq  C_{\mathfrak{T},t,\epsilon}\delta^{\epsilon}  C^{-\eta^{-2}} \delta^{(t-1)/2}\gamma_{\mathcal{P},t^\ast}^{-1/2} \lambda^{1/2}\sum_{f\in\mathcal{F}}|\mathcal{P}(f)|.
		\end{equation}
	\end{thm}
	
	\begin{proof}
		We use a backward induction on $\delta$. The proof is organized in five steps.
		
		\subsection*{Step 1: initial reductions.} It suffices to prove Theorem \ref{thm-reduction3} for $\epsilon \in (0,1/100)$. Given $\epsilon\in (0,1/100)$, $\eta=\eta(\epsilon)>0$ will be determined later. Let $\delta_0=\delta_0(\epsilon,t,\mathfrak{T})>0$ be an initial scale to be chosen later. For the base case $\delta\geq \delta_0$, we can choose $C_{\mathfrak{T},t,\epsilon}>0$ small enough such that \eqref{equu-15} is true. Moreover, when $C\geq \delta^{-2\eta^2}$, we have $\delta^2\geq C^{-\eta^{-2}}$, thus \eqref{equu-15} follows directly. 
		
		In the following, we assume $0<\delta<\delta_0$ and $1\leq C< \delta^{-2\eta^2}$. 
		
		After dyadic pigeonholing and applying Lemma \ref{lem:uniform branching}, we may assume that $|\mathcal{P}(f)|$ are constant for all $f\in\mathcal{F}$, and there is a uniform branching function for $\{\mathcal{P}(f)\}_{f\in\mathcal{F}}$. Without loss of generality, we assume $|\mathcal{P}(f)|\sim \lambda \delta^{-1}$.
		
		Let $\eta_1=\eta\epsilon^3$ and recall $t^\ast=\min\{t,2-t\}$. Let $\delta_1=\delta_1(\eta_1)=\delta_1(\epsilon)$ be the constant given by Lemma \ref{multi-scale-lem}, then take $\delta<\delta_0<\delta_1$. Applying Lemma \ref{multi-scale-lem} to $\{\mathcal{P}(f)\}_{f\in\mathcal{F}}$ with $(t,\eta)=(t^\ast,\eta_1)$, there exist an $r\in[\delta^{1-\eta_0\eta_1^{-1}},1]$ (here $\eta_0:=\eta_1^{2\eta_1^{-1}}$) and an $s\in (0,1]$ such that the following is true: 
		
		For each $f\in \cF$, let $\tilde{\cP}(f) := [\cP(f)]_r$, then we have
		\begin{enumerate}
			\item [\textup{(B1)}] \phantomsection \label{B1} $\ga_{\tilde \cP, t^\ast}(f)\lesssim \ga_{\cP,t^\ast}(f)$; 
			\item [\textup{(B2)}] \phantomsection \label{B2} for each $J=u_Q(6\delta)\in [\cP(f)]_r$, $T_Q(\cP(f)\cap J)$ is a $(\delta/r, s, (\delta/r)^{-9\eta_1})$-set, where $T_Q$ is the map taking $Q$ to $[-1,1]^2$. Moreover,
			$$\log_{1/\de}\Big(\frac{|\cP(f)|_{\de}}{|\cP(f)|_{r}}\Big)\leq (s+9\eta_1) \log_{1/\de}\Big(\frac{r}{\delta}\Big).$$
		\end{enumerate}
		We note that $s>0$ can be seen from the proof of \cite[Lemma 2.23]{WW2} by using $C< \delta^{-2\eta^2}$ and choosing $\eta<\epsilon^4$. 
		
		Now if $r\geq \delta^\eta$, we infer from property \nref{B2} that $\mathcal{P}(f)$ is a $(\delta,s,\delta^{-2\eta})$-set. Let $\eta'$ and $\delta_2$ be the small constants determined by Corollary \ref{cor2} with $\nu=\epsilon$. Choose 
		\[\delta<\delta_0<\min\{\delta_1,\delta_2\}, \quad \eta<\min\{\eta'/2,\epsilon^4\},\]
		then applying Corollary \ref{cor2} to $(\mathcal{F},\mathcal{P})_\delta$ gives \eqref{equu-15}. Therefore, in the sequel we assume \[r\in[\delta^{1-\eta_0\eta_1^{-1}},\delta^\eta].\]
		We note that $r\leq \delta^\eta$ will be used for only one time, see {\bf Case II} in {\bf Step 5}.

		\subsection*{Step 2: refining $(\mathcal{F},\mathcal{P})_\delta$ at scale $r$.} First,  let us recall the construction below.
		
		For each $Q\in\mathcal{D}_r(\mathbf{E}_{\mathcal{F},\mathcal{P}})$, let $\Gamma(Q)=\{\Gamma_{f|_{I_Q}}: \Gamma\cap Q\neq \emptyset, f\in\mathcal{F}\}$. Two curve segments $\Gamma_{f|_{I_{Q}}},\Gamma_{g|_{I_{Q}}}$ are \emph{comparable} if $|f(x) - g(x)| \leq 3\delta$ for all $x \in I_{Q}$. Let $\mathbb{S}_Q$ be a maximal set of \emph{incomparable} curve segments in $\Gamma(Q)$ and write $\mathbb{S}:=\cup_Q \mathbb{S}_Q$. We have two properties for incomparable segments. 
		\begin{itemize}
			\item Each $u\in \mathbb{S}$ belongs to at most two different $\mathbb{S}_Q$.
			
			\item For each $Q\in\mathcal{D}_r(\mathbf{E}_{\mathcal{F},\mathcal{P}})$, each segment in $\Gamma(Q)$ is contained in the vertical $3\delta$-neighborhood (recall Definition \ref{def:vertical}) of one and at most $O(\mathfrak{T}^{4})$ segments in $\mathbb{S}_Q$. 
		\end{itemize}
		For the proofs of these facts, see \cite[Appendix A]{OPY2025}. In the sequel, when $u\in \mathbb{S}_Q$, we use $f_u\in \mathcal{F}$ to denote the function such that $u=\Gamma_{f_u|_{I_Q}}$.
		
		Now, for each $Q\in \mathcal{D}_r(\mathcal{P}(f))$, choose one $u_Q\in\mathbb{S}_Q$ (at most $O(\mathfrak{T}^4)$ choices) such that $\Gamma_{f|_{I_Q}}\subset u_Q(3\delta)$ and thus $\mathcal{P}(f)\cap Q\subset u_Q(6\delta)$ since $f\in\mathcal{F}\subset B_{C^2}(1)$. For each $f\in \mathcal{F}$, define
		\begin{equation}\label{equ-curvedrectangle}
			\tilde{\mathcal{P}}(f):=\{u_Q(6\delta): Q\in \mathcal{D}_r(\mathcal{P}(f))\}.
		\end{equation}
		In the following, we will use $J$ to denote the elements in $\tilde{\mathcal{P}}(f)$. Note that each $J\in \tilde{\mathcal{P}}(f)$ is a curved rectangle of dimensions $\sim (\delta\times r)$ and each $\delta$-cube in $\mathcal{P}(f)$ is contained in a unique element of $\tilde{\mathcal{P}}(f)$. 
		
		The goal of this step is to prove the following claim.
		
		\begin{claim}\label{claimrefine1}
			There exist a refinement $(\mathcal F_1,\mathcal P_1)_\delta$ of $(\mathcal F,\mathcal P)_\delta$ satisfying the following properties.
			
			\begin{itemize}
				\item [(D1)] \phantomsection \label{D1} $\mathcal{P}_1(f):=\cup_{J\in \tilde{\mathcal{P}}'(f)} (\mathcal{P}(f)\cap J)$ for each $f\in \mathcal{F}_1$, where $\tilde{\mathcal{P}}'(f)$ is a refinement of $\tilde{\mathcal{P}}(f)$ and
				\begin{equation}
					|\tilde{\mathcal{P}}'(f)|\gtrapprox |\tilde{\mathcal{P}}(f)|\gtrapprox \tilde{\lambda} r^{-1}, \quad \tilde{\lambda}:=\lambda (r/\delta)^{1-s-9\eta_1}.
				\end{equation}
				
				\item [(D2)] \phantomsection \label{D2} For each $f\in \mathcal F_1$ and each $J=u_Q(6\delta)\in \widetilde{\mathcal P}'(f)$,
				\begin{equation}\label{eqqq-upper}
					|\mathcal P_1(f)\cap J|\lesssim ( \delta/r)^{-(s+10\eta_1)}.
				\end{equation}
				Moreover, $T_Q(\mathcal P_1(f)\cap J)$ is a $(\delta/r,s,(\delta/r)^{-10\eta_1})$-set, where $T_Q$ is the map taking $Q$ to $[-1,1]^2$.
				
				\item [(D3)] \phantomsection \label{D3} There is an integer $M\geq1$ such that for every curved rectangle $J\in \tilde{\mathcal P}'(f)$,
				\begin{equation}
					|\{f\in \mathcal F_1: J\in \tilde{\mathcal P}'(f)\}|\sim_\mathfrak{T} M \lessapprox C_{\mathfrak{T},t,\epsilon}^{-1} C^{\eta^{-2}}r^{-\epsilon} \tilde{\lambda}^{-1/2} r^{(1-t)/2}\gamma_{\tilde{\mathcal{P}},t^\ast}^{1/2}.
				\end{equation}
				
				\item [(D4)] \phantomsection \label{D4} $|\mathcal F_1|\approx |\mathcal F|$.
			\end{itemize}
		\end{claim}
		
		\begin{proof}
			For $f\in\mathcal{F}$ and $Q\in \mathcal{D}_r(\mathcal{P}(f))$, it follows from \nref{B2} that
			\begin{equation}\label{equ-21}
				|\mathcal{P}(f)\cap Q|\approx \frac{|\mathcal{P}(f)|}{|\mathcal{D}_r(\mathcal{P}(f))|}\leq \Big(\frac{\delta}{r}\Big)^{-(s+9\eta_1)}.
			\end{equation}
			Since $|\mathcal{P}(f)|\sim \lambda \delta^{-1}$, we use \eqref{equ-21} to deduce
			\[|\mathcal{D}_r(\mathcal{P}(f))|\approx \frac{|\mathcal{P}(f)|}{|\mathcal{P}(f)\cap Q|}\gtrapprox \lambda\delta^{-1}\cdot\Big(\frac{\delta}{r}\Big)^{(s+9\eta_1)}=\lambda\Big(\frac{r}{\delta}\Big)^{(1-s-9\eta_1)}r^{-1},\]
			which implies that $|\tilde{\mathcal{P}}(f)|=|\mathcal{D}_r(\mathcal{P}(f))|\gtrapprox \tilde{\lambda} r^{-1}$, where $\tilde{\lambda}:=\lambda (r/\delta)^{1-s-9\eta_1}$. 
			
			For $\mathbf{F}\in \mathcal{D}_{\delta/r}(\mathcal{F})$, write $\mathcal{F}[\mathbf{F}]:=\mathcal{F}\cap \mathbf{F}$. Then we apply Lemma \ref{refine1} to $(\mathcal{F}[\mathbf{F}],\tilde{\mathcal{P}}(f))$ at scale $r$, which gives $M_\mathbf{F}\geq1$, a sub-family $\mathbf{E}_{\mathbf{F}}\subset \mathbf{E}_{\mathcal{F}[\mathbf{F}],\tilde{\mathcal{P}}}$ and a refinement $(\mathcal{F}[\mathbf{F}]',\tilde{\mathcal{P}}'(f))$ of $(\mathcal{F}[\mathbf{F}],\tilde{\mathcal{P}}(f))$ so that 
			\begin{enumerate}
				\item [(i)] $\tilde{\mathcal{P}}'(f)= \mathbf{E}_{\mathbf{F}}\cap \tilde{\mathcal{P}}(f)$ for all $f\in \mathcal{F}[\mathbf{F}]$ and $\tilde{\mathcal{P}}'(f)$ is a refinement of $\tilde{\mathcal{P}}(f)$ for all $f\in \mathcal{F}[\mathbf{F}]'$. Here we have $\mathbf{E}_{\mathbf{F}}=\cup_{f\in \mathcal{F}[\mathbf{F}]}\tilde{\mathcal{P}}'(f)$.
				\item [(ii)] $|\{f\in\mathcal{F}[\mathbf{F}]: J\in \tilde{\mathcal{P}}'(f)\}|\sim M_\mathbf{F}$ for all $J\in \mathbf{E}_{\mathbf{F}}$.
				\item [(iii)] $M_\mathbf{F}\sim \sum_{f\in \mathcal{F}[\mathbf{F}]}|\tilde{\mathcal{P}}'(f)|/|\mathbf{E}_{\mathbf{F}}|\approx \sum_{f\in \mathcal{F}[\mathbf{F}]'}|\tilde{\mathcal{P}}'(f)|/|\mathbf{E}_{\mathcal{F}[\mathbf{F}]',\tilde{\mathcal{P}}'}|$.
			\end{enumerate}

			For each $f\in \mathcal{F}$, we define a new shading 
			\begin{equation}\label{equp1}
				\mathcal{P}_1(f):=\cup_{J\in \tilde{\mathcal{P}}'(f)} (\mathcal{P}(f)\cap J).
			\end{equation}
			Since $|\mathcal{D}_r(\mathcal{P}(f))|\gtrapprox \tilde{\lambda} r^{-1}$, we have $|\tilde{\mathcal{P}}'(f)|\gtrapprox |\tilde{\mathcal{P}}(f)|\gtrapprox \tilde{\lambda} r^{-1}$, hence property \nref{D1} holds. To see property \nref{D2}, note that 
			\[|\mathcal P_1(f)\cap J|=|\mathcal P(f)\cap J|, \quad J\in \tilde{\mathcal{P}}'(f).\]
			Since $J=u_Q(6\delta)$ for some $Q\in \mathcal{D}_r(\mathcal{P}(f))$ was chosen such that $\mathcal{P}(f)\cap Q\subset J$, by \eqref{equ-21} we get \eqref{eqqq-upper}. The $(\delta/r,s)$-set condition of $T_Q(\mathcal P_1(f)\cap J)$ follows from property \nref{B2}.

			It remains to show \nref{D3} and \nref{D4}. Since $\mathcal{P}(f)$ is a uniform $(\delta,\epsilon^2,C;\rho)$-set, $\tilde{\mathcal{P}}'(f)$ is a $(r,\epsilon^2,C\bar{C};\max\{\rho,r\})$-set for some $\bar{C}\lessapprox 1$. Note that $\max\{\rho,r\}\in [r,r^\eta]$ since $\rho\in[\delta,\delta^\eta]$. We aim to apply induction hypothesis at scale $r$. To proceed, fix any $f_\mathbf{F}\in \mathcal{F}[\mathbf{F}]'$ and define the rescaling map
			\begin{equation}\label{eq-TF}
				T_\mathbf{F}(f)=\frac{f-f_\mathbf{F}}{\delta/r}, \quad f\in \mathcal{F}. 
			\end{equation}
			We also abuse notation to define \[T_\mathbf{F}(x,y)=\Big(x, \frac{y-f_\mathbf{F}(x)}{\delta/r}\Big), \quad (x,y)\in \R^2.\] Then $T_\mathbf{F}(\mathcal{F}[\mathbf{F}]')\subset B(0,O(\mathfrak{T}))$ is a transversal family by Lemma \ref{lem1} and $T_\mathbf{F}(J)$ is roughly a cube of side length $\sim r$ for each $J\in \tilde{\mathcal{P}}'(f)$ as $f\in \mathcal{F}\subset B_{C^2}(1)$. Moreover, $T_\mathbf{F}(\tilde{\mathcal{P}}'(f))$ is a $(r,\epsilon^2,CC_1;\max\{\rho,r\})$-set for some $C_1\lessapprox 1$. By applying induction hypothesis to the rescaled configuration $(T_\mathbf{F}(\mathcal{F}[\mathbf{F}]'),T_\mathbf{F}(\tilde{\mathcal{P}}'(f)))_r$, we get
			\begin{equation}\label{equ-richness-J}
				\begin{split}
					M_\mathbf{F} & \approx\sum_{f\in \mathcal{F}[\mathbf{F}]'}|\tilde{\mathcal{P}}'(f)|/|\mathbf{E}_{\mathcal{F}[\mathbf{F}]',\tilde{\mathcal{P}}'}|\lessapprox  C_{\mathfrak{T},t,\epsilon}^{-1} C^{\eta^{-2}}r^{-\epsilon}\cdot \tilde{\lambda}^{-1/2} r^{(1-t)/2}\gamma_{\tilde{\mathcal{P}}',t^\ast}^{1/2}\\
					&\leq C_{\mathfrak{T},t,\epsilon}^{-1} C^{\eta^{-2}}r^{-\epsilon}\cdot \tilde{\lambda}^{-1/2} r^{(1-t)/2}\gamma_{\tilde{\mathcal{P}},t^\ast}^{1/2},
				\end{split}
			\end{equation}
			where we use $\gamma_{\tilde{\mathcal{P}}',t^\ast} \leq \gamma_{\tilde{\mathcal{P}},t^\ast}$. Recalling property (ii) and (iii) above, this means
			\begin{equation}\label{eq-richness2}
				|\{f\in\mathcal{F}[\mathbf{F}]: J\in \tilde{\mathcal{P}}'(f)\}|\lessapprox C_{\mathfrak{T},t,\epsilon}^{-1} C^{\eta^{-2}}r^{-\epsilon}\cdot \tilde{\lambda}^{-1/2} r^{(1-t)/2}\gamma_{\tilde{\mathcal{P}},t^\ast}^{1/2},\quad J\in \mathbf{E}_{\mathbf{F}}.
			\end{equation}

			Since $\mathcal{P}(f)$ is uniform and $|\tilde{\mathcal{P}}'(f)|\gtrapprox |\tilde{\mathcal{P}}(f)|$, $(\mathcal{F},\mathcal{P}_1)_\delta$ is a refinement of $(\mathcal{F},\mathcal{P})_\delta$. By dyadic pigeonholing, there exist a subset $\mathcal{F}_1[\mathbf{F}]\subset \mathcal{F}[\mathbf{F}]$ for each $\mathbf{F}\in\mathcal{D}_{\delta/r}(\mathcal{F})$ and a constant
			\begin{equation}\label{equ-constantM}
				M\lessapprox C_{\mathfrak{T},t,\epsilon}^{-1} C^{\eta^{-2}}r^{-\epsilon}\cdot \tilde{\lambda}^{-1/2} r^{(1-t)/2}\gamma_{\tilde{\mathcal{P}},t^\ast}^{1/2}
			\end{equation}
			such that
			\begin{itemize}
				\item [(i)] if $M_\mathbf{F}\sim M$, then $\mathcal{F}_1[\mathbf{F}]=\mathcal{F}[\mathbf{F}]$, otherwise $\mathcal{F}_1[\mathbf{F}]=\emptyset$;
				
				\item [(ii)] let $\mathcal{F}_1:=\cup_{\mathbf{F}\in \mathcal{D}_{\delta/r}(\mathcal{F})}\mathcal{F}_1[\mathbf{F}]$, then $(\mathcal{F}_1,\mathcal{P}_1)_\delta$ is a refinement of $(\mathcal{F},\mathcal{P})_\delta$.
			\end{itemize}
			Since $|\mathcal{P}(f)|\sim \lambda \delta^{-1}$ for all $f\in\mathcal{F}$, we have $|\mathcal{F}_1|\gtrapprox |\mathcal{F}|$ which proves \nref{D4}. Moreover, for each $J=u_Q(6\delta)\in \cup_{f\in \mathcal{F}_1[\mathbf{F}]} \tilde{\mathcal{P}}'(f)$ with $\mathcal{F}_1[\mathbf{F}]\neq\emptyset$, it follows from \eqref{eq-richness2} that
			\begin{equation}\label{equ-24}
				|\{f\in\mathcal{F}_1[\mathbf{F}]: J\in \tilde{\mathcal{P}}'(f)\}|\sim M.
			\end{equation}
			Since \[\{f\in \mathcal{F}_1: J\in \tilde{\mathcal{P}}'(f)\}\subset \{f\in \mathcal{F}_1: \Gamma_{f|_{I_Q}}\subset u_Q(3\delta)\}\]
			while $\{f\in \mathcal{F}_1: \Gamma_{f|_{I_Q}}\subset u_Q(3\delta)\}$ can intersect at most $\lesssim_\mathfrak{T}1$ dyadic $\delta/r$-cubes $\mathbf{F}$, see \cite[Claim A.15]{OPY2025} for the proof, combining \eqref{equ-constantM} and \eqref{equ-24} gives \nref{D3} and finishes the proof of Claim \ref{claimrefine1}.
		\end{proof}
		
		\subsection*{Step 3: constructing local configurations.}
		
		In this step, we localize the refined configuration $(\mathcal{F}_1,\mathcal{P}_1)_\delta$ from {\bf Step 2} to a fixed $r$-cube $Q\in\mathcal{D}_r(\mathbf{E}_{\mathcal{F}_1,\mathcal{P}_1})$. The goal is to decompose the portion of the configuration inside $Q$ into a controlled collection of smaller configurations, and after some refining process each has two useful properties: first, the associated shadings remain quantitatively $s$-dimensional after rescaling; second, the relevant incidence multiplicities--namely, the number of curve segments in a suitable refinement of $\mathbb{S}_Q$ that intersect each $p\in \mathcal{P}_1$--are essentially constant. This will allow us, in the next step, to  apply Lemma \ref{lem-main} to some refinement of each local configuration.

		For each $Q\in\mathcal{D}_r(\mathbf{E}_{\mathcal{F}_1,\mathcal{P}_1})$, recall $\mathbb S_Q$ constructed at the beginning of {\bf Step 2}, then define
		\begin{equation}\label{def:SQ}
			\mathbb{S}_Q':=\{u\in\mathbb{S}_Q: u(6\delta) \in \tilde{\mathcal{P}}'(f) \text{ for some } f\in\mathcal{F}_1\}
		\end{equation}
		and
		\[\mathcal{F}_Q':=\{f_u:u\in\mathbb{S}_Q'\},\]
		where $f_u\in \mathcal{F}$ satisfies $u=\Gamma_{f_u|_{I_Q}}$. For each $u\in \mathbb S_Q'$, write
		\begin{equation}\label{eq-f1}
			\mathcal{F}_1(u):=\{f\in \mathcal{F}_1: u(6\delta) \in \tilde{\mathcal{P}}'(f) \}.
		\end{equation}
		Let $\mathcal{F}_1(Q):=\{f\in\mathcal{F}_1: u(6\delta) \in \tilde{\mathcal{P}}'(f) \text{ for some } u\in\mathbb{S}_Q'\}$, then we claim that \[\mathcal{F}_1(Q)=\sqcup_{u\in \mathbb S_Q'}\mathcal{F}_1(u).\] 
		It is clear that $\mathcal{F}_1(Q)=\cup_{u\in \mathbb S_Q'}\mathcal{F}_1(u)$, then we prove the disjointness. Indeed, by our construction in \eqref{equ-curvedrectangle}, for each $Q\in \mathcal{D}_r(\mathcal{P}(f))$, we take only one $u_Q$ such that 
		\[u_Q(6\delta)\in \tilde{\mathcal{P}}(f).\]
		This means for fixed $Q\in\mathcal{D}_r(\mathcal{P}_1(f))$, there is only one $u\in \mathbb{S}'_Q$ such that $f\in\mathcal{F}_1(u)$, i.e., $\{\mathcal{F}_1(u):u\in \mathbb{S}'_Q\}$ are disjoint. Write $K(u):=|\mathcal{F}_1(u)|$. By property \nref{D3} we know
		\begin{equation}\label{eq-k(u)}
			K(u)\sim_\mathfrak{T} M.
		\end{equation}
		For each $u\in \mathbb S_Q'$, we enumerate \[\mathcal{F}_1(u)=\{f_1(u), f_2(u), \cdots, f_{K(u)}(u)\}.\] Write $\mathcal{F}_{k,Q}:=\{f_k(u): u\in \mathbb S_Q'\}$, where $\{f_k(u)\}=\emptyset$ if $k> K(u)$. Then we have \[\mathcal{F}_1(Q)=\sqcup_{k}\mathcal{F}_{k,Q}. \]
		
		Now, for each $1\leq k \leq \max_{u\in \mathbb{S}_Q'} K(u)$ and $u\in \mathbb{S}_Q'$, we define
		\begin{equation}\label{equ-41}
			\mathcal{P}_{Q,k}(u):=\left\{
			\begin{array}{lll}
				\mathcal{P}_1(f_k(u))\cap Q, \quad &\text{if}~~k\leq K(u),\\
				\emptyset, \quad &\text{otherwise}.
			\end{array}\right.
		\end{equation}
		Write $K(Q):=\min_{u\in \mathbb{S}_Q'} K(u)$. We have $K(Q)\sim_\mathfrak{T} M$ and $\mathcal{P}_{Q,k}(u)\neq\emptyset$ if $k\leq K(Q)$. Then, 
		\begin{equation}\label{equ-70}
			\begin{split}
				\sum_{p\in Q \cap \mathbf{E}_{\mathcal{F}_1,\mathcal{P}_1}}|\mathcal{F}_1(p)|&\leq\sum_{f\in \mathcal{F}_1(Q)} |\mathcal{P}_1(f)\cap Q|\sim\sum_{k}\sum_{f\in \mathcal{F}_{k,Q}}|\mathcal{P}_1(f)\cap Q|\\
				&=\sum_{k}\sum_{u\in \mathbb{S}_{Q}'}|\mathcal{P}_{Q,k}(u)|\lessapprox \sum_{k\leq K(Q)}\sum_{p\in \mathbf{E}_{\mathbb{S}_{Q}',\mathcal{P}_{Q,k}}}|\mathbb{S}_Q'(p)|,
			\end{split}
		\end{equation}
		where $\mathbb{S}_Q'(p):=\{u\in \mathbb{S}_Q': p\in \mathcal{P}_{Q,k}(u)\}$. Recall property \nref{D2}, then for all $Q\in \mathcal{D}_r(\mathbf{E}_{\mathcal{F}_1,\mathcal{P}_1})$, $u\in \mathbb{S}_Q'$ and $k\leq K(Q)$, we have
		\begin{itemize}
			\item [(i)] \phantomsection \label{B1} $|\mathcal{P}_{Q,k}(u)|\leq (\delta/r)^{-(s+10\eta_1)}$,
			
			\item [(ii)] \phantomsection \label{B2} $T_Q(\mathcal{P}_{Q,k}(u))$ is a $(\delta/r,s, (\delta/r)^{-10\eta_1})$-set. 
		\end{itemize}
		Then, for each $k\leq K(Q)$, we apply Lemma \ref{refine1} and use the properties (i), (ii) above to get a refinement $(\mathbb{S}_Q'',\mathcal{P}_{Q,k}'')_\delta$ of $(\mathbb{S}_Q',\mathcal{P}_{Q,k})_\delta$ such that
		\begin{itemize}
			\item [(J1)] \phantomsection \label{J1} For each $u\in \mathbb{S}_Q''$, $T_Q(\mathcal{P}_{Q,k}''(u))$ is a $(\delta/r,s, (\delta/r)^{-10\eta_1})$-set and
			\[|\mathcal{P}''_{Q,k}(u)|\leq (\delta/r)^{-(s+10\eta_1)}.\]
			
			\item [(J2)] \phantomsection \label{J2} For all $p\in \mathbf{E}_{\mathbb{S}_Q'',\mathcal{P}_{Q,k}''}:=\cup_{u\in \mathbb{S}_Q''}\mathcal{P}_{Q,k}''(u)$, we have
			\begin{equation}\label{equ-density11}
				|\mathbb{S}_Q''(p)|=|\{u\in \mathbb{S}_Q'': p\in \mathcal{P}_{Q,k}''(u)\}|\lessapprox \frac{\sum_{u\in \mathbb{S}_{Q,k}''}|\mathcal{P}_{Q,k}''(u)|}{|\mathbf{E}_{\mathbb{S}_Q'',\mathcal{P}_{Q,k}''}|}.
			\end{equation}
		\end{itemize}
		Since $|\mathcal{P}_{Q,k}(u)|$ are approximately constant for all $u\in \mathbb S_Q'$, we infer $|\mathbb{S}_Q''|\gtrapprox |\mathbb{S}_Q'|$. Moreover, after a further refinement, we may assume that $\mathcal{P}_{Q,k}''(u)$ is uniform for each $u\in \mathbb{S}_Q''$.

		\subsection*{Step 4: applying Lemma \ref{lem-main}.}
		Our plan is to find a refinement of $(\mathbb{S}_Q'',\mathcal{P}_{Q,k}'')_\delta$ and then apply Lemma \ref{lem-main} to this refinement after rescaling. Let $(x_0,y_0)$ denote the lower-left corner of $Q$ and define 
		\[\widetilde{\mathcal{F}}_Q'':=\{(f_u)_{(x_0,y_0),r}: u\in\mathbb S_Q''\}.\]
		By Lemma \ref{lem2}, $\widetilde{\mathcal{F}}_Q''$ is $(4\mathfrak{T}+1)$-transversal on $[-2,2]$. By the same argument as in \cite[Claim A.24]{OPY2025}, one checks that $\widetilde{\mathcal{F}}_Q''$ is $\sim_\mathfrak{T} \tfrac{\delta}{r}$-separated. Now for each $\bar{f}=(f_u)_{(x_0,y_0),r}\in \widetilde{\mathcal{F}}_Q''$, we define the shading
		\[\widetilde{\mathcal{P}}_{Q,k}''(\bar{f}):=T_Q(\mathcal{P}_{Q,k}''(u)) \subset \mathcal{D}_{\delta/r}.\]
		By \nref{J1}, \nref{J2}, the rescaled configuration $(\widetilde{\mathcal{F}}_Q'',\widetilde{\mathcal{P}}_{Q,k}'')_{\delta/r}$ satisfies:
		\begin{itemize}
			\item [(i)] for each $\bar{f}\in \widetilde{\mathcal{F}}_Q''$, $\widetilde{\mathcal{P}}_{Q,k}''(\bar{f})$ is a $(\delta/r,s, (\delta/r)^{-10\eta_1})$-set;
			
			\item [(ii)] for all $\bar{p}\in T_Q(\mathbf{E}_{\mathbb{S}_Q'',\mathcal{P}_{Q,k}''})=\{T_Q(p): p\in \mathbf{E}_{\mathbb{S}_Q'',\mathcal{P}_{Q,k}''}\}$, we have
			\begin{equation}\label{equ-density11}
				|\widetilde{\mathcal{F}}_Q''(\bar{p})|=|\{\bar{f}\in \widetilde{\mathcal{F}}_Q'': \bar{p}\in \widetilde{\mathcal{P}}_{Q,k}''(\bar{f})\}|\lessapprox \frac{\sum_{\bar{f}\in \widetilde{\mathcal{F}}_Q''}|\widetilde{\mathcal{P}}_{Q,k}''(\bar{f})|}{|T_Q(\mathbf{E}_{\mathbb{S}_Q'',\mathcal{P}_{Q,k}''})|}.
			\end{equation}
		\end{itemize}
		Next, we use Lemma \ref{lem-uniformsubset} to get a uniform refinement $\mathcal{F}_Q^\ast\subset \widetilde{\mathcal{F}}_Q''$ and define 
		\[\mathbb{S}_Q^\ast:=\{u: (f_u)_{(x_0,y_0),r}\in \mathcal{F}_Q^\ast\}\subset \mathbb{S}_Q''.\]
		For each $\bar{f}\in \widetilde{\mathcal{F}}_Q''$, define the shading
		\begin{equation}\label{equ-72}
			\mathcal{P}_{Q,k}^\ast(\bar{f}):=\left\{
			\begin{array}{lll}
				\widetilde{\mathcal{P}}_{Q,k}''(\bar{f}), \quad &\text{if}~~\bar{f}\in \mathcal{F}_Q^\ast,\\
				\emptyset, \quad &\text{otherwise}.
			\end{array}\right.
		\end{equation}
		In the sequel, we will abuse notation to use (after rescaling back)
		\begin{equation}\label{equ-73}
			\mathcal{P}_{Q,k}^\ast(u):=\left\{
			\begin{array}{lll}
				\mathcal{P}_{Q,k}''(u), \quad &\text{if}~~u\in \mathbb{S}_Q^\ast,\\
				\emptyset, \quad &\text{otherwise}.
			\end{array}\right.
		\end{equation}
		Since $|\widetilde{\mathcal{P}}_{Q,k}''(\bar{f})|=|\mathcal{P}_{Q,k}''(u)|$ is approximately constant, $(\mathcal{F}_Q^\ast,\mathcal{P}_{Q,k}^\ast)_{\delta/r}$ is a refinement of $(\widetilde{\mathcal{F}}_Q'',\widetilde{\mathcal{P}}_{Q,k}'')_{\delta/r}$. Moreover, the following properties are inherited from $(\widetilde{\mathcal{F}}_Q'',\widetilde{\mathcal{P}}_{Q,k}'')_{\delta/r}$.
		\begin{itemize}
			\item [(i)] For each $\bar{f}\in \mathcal{F}_Q^\ast$, $\mathcal{P}_{Q,k}^\ast(\bar{f})$ is a $(\delta/r,s, (\delta/r)^{-10\eta_1})$-set.
			
			\item [(ii)] For all $\bar{p}\in \mathbf{E}_{\mathcal{F}_Q^\ast,\mathcal{P}_{Q,k}^\ast}=\cup_{\bar{f}\in \mathcal{F}_Q^\ast}\mathcal{P}_{Q,k}^\ast(\bar{f})$, we have (note $|\mathcal{F}_Q^\ast|\gtrapprox |\widetilde{\mathcal{F}}_Q''|$)
			\begin{equation}\label{equ-density12}
				|\mathcal{F}_Q^\ast(\bar{p})|=|\{\bar{f}\in \mathcal{F}_Q^\ast: \bar{p}\in \mathcal{P}_{Q,k}^\ast(\bar{f})\}|\lessapprox \frac{\sum_{\bar{f}\in \mathcal{F}_{Q}^\ast}|\mathcal{P}_{Q,k}^\ast(\bar{f})|}{|\mathbf{E}_{\mathcal{F}_Q^\ast,\mathcal{P}_{Q,k}^\ast}|}.
			\end{equation}
		\end{itemize}
		Having obtained the suitable refinement $(\mathcal{F}_Q^\ast,\mathcal{P}_{Q,k}^\ast)_{\delta/r}$, we are now in a position to apply Lemma \ref{lem-main} at scale $\delta/r$.

		Let $\bar{\eta}, \bar{\delta}>0$ be the constants given by Lemma \ref{lem-main} with $\nu=\epsilon^2$. Recall $\eta_0=\eta_1^{2\eta_1^{-1}}$ and $\eta_1=\eta\epsilon^3$, then write
		\[h(\eta):=\eta_0\eta_1^{-1}=(\eta\epsilon^3)^{2/(\eta\epsilon^3)-1}.\]
		Since $r\geq \delta^{1-h(\eta)}$ and $\eta<\epsilon^4$ as selected in {\bf Step 1}, we take
		\[\delta/r\leq \delta_0^{h(\eta)}< \delta_0^{h(\epsilon^4)}<\bar{\delta}.\]
		Consequently, we choose $\delta<\delta_0<\bar{\delta}^{1/h(\epsilon^4)}$, $\eta<\bar{\eta}$ and note $\epsilon<1/100$, which ensures
		\[\delta/r<\bar{\delta}, \quad 10\eta_1<\bar{\eta}.\]
		Then, by applying Lemma \ref{lem-main} to $(\mathcal{F}_Q^\ast,\mathcal{P}_{Q,k}^\ast)_{\delta/r}$, we obtain a scale $\Delta=\Delta_{Q,k}\in [\delta,r]$ such that the following is true (after rescaling back).
		\begin{itemize}
			\item [(P1)] \phantomsection \label{P1} Let $N_{Q,k}:=\max_{p\in \mathbf{E}_{\mathbb{S}_Q^\ast,\mathcal{P}_{Q,k}^\ast}} |\mathbb{S}_Q^\ast(p)|$, where $\mathbb{S}_Q^\ast(p)=\{u\in \mathbb{S}_Q^\ast: p\in \mathcal{P}_{Q,k}^\ast(u)\}$. Recall 
			\[f(s,t)=\left\{
			\begin{array}{lll}
				0, \quad &\text{if}~~t\in(0,s],\\
				(s-t)/2, &~\text{if}~t\in[s,2-s],\\
				s-1, &~\text{if}~t\in [2-s,2].
			\end{array}\right.\]
			Then we have
			\begin{equation}\label{equ-multiplicity}
				N_{Q,k}\leq \Big(\frac{\delta}{r}\Big)^{-\epsilon^2} \Big(\frac{\delta}{\Delta}\Big)^{-O(1)}\Big(\frac{\delta}{r}\Big)^{f(s,t)}.
			\end{equation}
			
			\item [(P2)] \phantomsection \label{P2} If $\Delta/r \geq (\delta/r)^{1-\sqrt{\eta}}$, then there exist a refinement $\mathcal{S}_Q\subset \mathbb{S}_Q^\ast$ and an integer $d_{Q}\geq1$ such that for all $u\in\mathcal{S}_Q$ there is a cube family $\mathcal{Q}(u)\subset \mathcal{D}_\Delta\cap u(O(\Delta))$ satisfying
			\begin{itemize}
				\item [(a)] $|\mathcal{Q}(u)|\sim d_{Q}$.
				\item [(b)] For each $Q_\Delta\in\mathcal{Q}(u)$, we have
				\begin{equation}\label{equ-largecardinality}
					|\mathbf{E}_{\mathcal{S}_Q^\ast,\mathcal{P}_{Q,k}^\ast}\cap Q_\Delta|\geq  \Big(\frac{\delta}{\Delta}\Big)^{\epsilon^2} \Big(\frac{\delta}{r}\Big)^{-2s+f(s,t)}\Big(\frac{\Delta}{r}\Big)^{s-f(s,t)}\cdot d_{Q}^{-1}.
				\end{equation}
			\end{itemize}
		\end{itemize}
		For each $u\in \mathbb S_Q'$, we define
		\begin{equation}\label{equ-77}
			\mathbb{P}_{Q,k}(u):=\left\{
			\begin{array}{lll}
				\mathcal{P}_{Q,k}^\ast(u), \quad &\text{if}~~u\in \mathcal{S}_Q,\\
				\emptyset, \quad &\text{otherwise}.
			\end{array}\right.
		\end{equation}
		Since $|\mathcal{S}_Q|\gtrapprox |\mathbb S_Q^\ast|$ and $\mathcal{P}_{Q,k}^\ast(u)$ is approximately constant,  $(\mathcal{S}_Q,\mathbb{P}_{Q,k})_{\delta}$ is a refinement of $(\mathbb S_Q^\ast,\mathcal{P}_{Q,k}^\ast)_{\delta}$. Note also
		\[(\mathbb S_Q^\ast,\mathcal{P}_{Q,k}^\ast)_{\delta} \xrightarrow{\text{refinement of}} (\mathbb{S}_Q'',\mathcal{P}_{Q,k}'')_\delta \xrightarrow{\text{refinement of}} (\mathbb{S}_Q',\mathcal{P}_{Q,k})_\delta.\]
		Now recall \eqref{equ-70}, then we infer
		\begin{equation}\label{equ-74}
			\begin{split}
				\sum_{p\in Q \cap \mathbf{E}_{\mathcal{F}_1,\mathcal{P}_1}}|\mathcal{F}_1(p)|&\lessapprox \sum_{k\leq K(Q)}\sum_{p\in \mathbf{E}_{\mathbb{S}_{Q}',\mathcal{P}_{Q,k}}}|\mathbb{S}_Q'(p)|\lessapprox \sum_{k\leq K(Q)}\sum_{p\in \mathbf{E}_{\mathbb{S}_{Q}^\ast,\mathcal{P}_{Q,k}^\ast}}|\mathbb{S}_Q^\ast(p)|\\
				&\lessapprox \sum_{k\leq K(Q)}\sum_{p\in \mathbf{E}_{\mathcal{S}_Q,\mathbb{P}_{Q,k}}}|\mathcal{S}_Q(p)|.
			\end{split}
		\end{equation}
		By dyadic pigeonholing, we can find a sub-family $\mathcal{B}_r\subset \mathcal{D}_r(\mathbf{E}_{\mathcal{F}_1,\mathcal{P}_1})$, a uniform scale $\Delta$, a uniform integer $d\geq1$ and an index set $\mathcal{K}_Q\subset [1,K(Q)]\cap \mathbb Z$ for each $Q\in\mathcal{B}_r$ such that
		\begin{itemize}
			\item [(R1)] \phantomsection \label{R1} for all $Q\in\mathcal{B}_r$ and $k\in \mathcal{K}_Q$, we have $\Delta_{Q,k}\sim\Delta$ and $d_{Q}\sim d$;
			
			\item [(R2)] \phantomsection \label{R2} we have
			\begin{equation}\label{equ-75}
				\sum_{p\in \mathbf{E}_{\mathcal{F}_1,\mathcal{P}_1}}|\mathcal{F}_1(p)|\lessapprox \sum_{Q\in\mathcal{B}_r}\sum_{k\in \mathcal{K}_Q}\sum_{p\in \mathbf{E}_{\mathcal{S}_Q,\mathbb{P}_{Q,k}}}|\mathcal{S}_Q(p)|=\sum_{Q\in\mathcal{B}_r}\sum_{k\in \mathcal{K}_Q}\sum_{u\in \mathcal{S}_Q}|\mathbb{P}_{Q,k}(u)|.
			\end{equation}
		\end{itemize}
		
		To proceed, we use $\{(\mathcal{S}_Q,\mathbb{P}_{Q,k}): Q\in \mathcal{B}_r, k\in \mathcal{K}_Q\}$ to construct a new configuration. For each $f\in\mathcal{F}_1$, let $\mathcal{B}_r(f)$ be those $r$-cubes in $\mathcal{B}_r$ intersecting $\mathcal{P}_1(f)$. By construction in {\bf Step 2}, for each $f\in \mathcal{F}_1$ and each $Q\in \mathcal B_r(f)\neq\emptyset$, there is a unique segment $u\in S'_Q$ such that $u(6\delta)\in \tilde{\mathcal P}'(f)$, i.e., $f\in \mathcal{F}_1(u)$ (recall \eqref{eq-f1}). Since $\mathcal F_1(u)$ has been enumerated as
		\[\mathcal F_1(u)=\{f_1(u),\dots,f_{K(u)}(u)\},\]
		this determines a unique index $k\leq K(u)=|\mathcal{F}_1(u)|$ such that
		\[f=f_k(u).\]
		However, we do not know if $(u,k)\in \mathcal{S}_Q\times \mathcal{K}_Q$ while we will need to apply properties \nref{R1} and \nref{R2}. Therefore, a sub-family of $\mathcal{F}_1$ is introduced precisely to retain those $f$ for which such a choice is possible:
		\begin{equation}\label{equ-79}
			\mathcal{F}_2:=\{f\in\mathcal{F}_1: \text {there is } Q\in \mathcal{B}_r(f)\neq \emptyset \text { such that }f=f_k(u) \text { with } (u,k)\in  \mathcal{S}_Q\times \mathcal{K}_Q \},
		\end{equation}
		and for each $f\in\mathcal{F}_2$, we define a new shading
		\begin{equation}\label{equ-78}
			\mathcal{P}_2(f):=
			\bigcup_{Q\in \mathcal{B}_r(f)} \bigcup_{\substack{(k,u)\in  \mathcal{K}_Q\times\mathcal{S}_Q\\ f=f_k(u)}}\mathbb{P}_{Q,k}(u).
		\end{equation}
		Note that if $Q\in \mathcal{B}_r(f)\neq \emptyset$ and $(k,u)\in  \mathcal{K}_Q\times\mathcal{S}_Q$ such that $f=f_k(u)$, we have
		\[\mathcal{P}_2(f)\cap Q=\mathbb{P}_{Q,k}(u)=\mathcal{P}_{Q,k}^\ast(u)=\mathcal{P}_{Q,k}''(u)\subset \mathcal{P}_{Q,k}(u)=\mathcal{P}_1(f_k(u))\cap Q.\]
		By \eqref{equ-75}, we infer that $(\mathcal{F}_2,\mathcal{P}_2)_\delta$ is a refinement of $(\mathcal{F}_1,\mathcal{P}_1)_\delta$. Up to a further pigeonholing, we may assume that $|\mathcal{P}_2(f)|\gtrapprox |\mathcal{P}_1(f)|$ for all $f\in\mathcal{F}_2$. Since $|\mathcal{P}_1(f)\cap Q|$ are approximately constant, we also have \[|B_r(f)|\geq|\mathcal{D}_r(\mathcal{P}_2(f))|\gtrapprox |\mathcal{D}_r(\mathcal{P}_1(f))|, \quad f\in\mathcal{F}_2.\]
		
		We close {\bf Step 4} with the following Claim.
		
		\begin{claim}\label{claimkeyl}
			The refinement $(\mathcal{F}_2,\mathcal{P}_2)_\delta$ of $(\mathcal{F}_1,\mathcal{P}_1)_\delta$ satisfy the following properties.
			
			\begin{itemize}
				\item [(T1)] \phantomsection \label{T1} For each $p\in\mathbf{E}_{\mathcal{F}_2,\mathcal{P}_2}$, let $\mathcal{F}_2(p):=\{f\in \mathcal{F}_2: p\in \mathcal{P}_2(f)\}$, then we have 
				\begin{equation}\label{equ-76}
					|\mathcal{F}_2(p)|\lessapprox M \Big(\frac{\delta}{r}\Big)^{-\epsilon^2} \Big(\frac{\delta}{\Delta}\Big)^{-O(1)}\Big(\frac{\delta}{r}\Big)^{f(s,t)}.
				\end{equation}
				
				\item [(T2)] \phantomsection \label{T2} If $\Delta/r \geq (\delta/r)^{1-\sqrt{\eta}}$, then for all $f\in\mathcal{F}_2$ and $Q\in\mathcal{D}_r(\mathcal{P}_2(f))$, there is a cube family $\mathcal{Q}(f,Q)\subset \mathcal{D}_\Delta\cap \Gamma_f(O(\Delta))$ with $|\mathcal{Q}(f,Q)|\sim d$ such that
				\begin{itemize}
					\item [(a)] For each $Q_\Delta\in\mathcal{Q}(f,Q)$, 
					\begin{equation}\label{equ-largecardinality1}
						|\mathbf{E}_{\mathcal{F},\mathcal{P}}\cap Q_\Delta|\gtrapprox \Big(\frac{\delta}{\Delta}\Big)^{\epsilon^2} \Big(\frac{\delta}{r}\Big)^{-2s+f(s,t)}\Big(\frac{\Delta}{r}\Big)^{s-f(s,t)}\cdot d^{-1}.
					\end{equation}
					\item[(b)] Write $\widetilde{\mathcal{P}}_2(f):=\cup_{Q\in \mathcal{D}_r(\mathcal{P}_2(f))}\mathcal{Q}(f,Q)$ for each $f\in\mathcal{F}_2$. Then 
					\begin{equation}\label{equ-tilde}
						|\widetilde{\mathcal{P}}_2(f)|\gtrapprox \tilde{\lambda}d (\Delta/r)\cdot \Delta^{-1}.
					\end{equation}
					Moreover, we have
					\begin{equation}\label{equ-density}
						\gamma_{\widetilde{\mathcal{P}}_2,t^\ast}\lesssim \Big(\frac{\delta}{\Delta}\Big)^{-t^\ast}\Big(\frac{\delta}{r}\Big)^{s-O(\eta_1)}\cdot d\cdot \gamma_{\mathcal{P}_2,t^\ast}.
					\end{equation}
				\end{itemize}
			\end{itemize}
		\end{claim}
		
		\begin{proof}
			To see \nref{T1}, take $Q\in \mathcal{B}_r$ such that $p\in Q$, then 
			\begin{align}
				&f \in \mathcal{F}_2(p)   \iff p\in \mathcal{P}_2(f)\\
				&\iff p\in \mathbb{P}_{Q,k}(u), \text{ where } (k,u)\in  \mathcal{K}_Q\times\mathcal{S}_Q \text{ such that } f=f_k(u)\\
				&\iff u\in \mathbb{S}_Q^\ast(p)=\{u\in \mathbb{S}_Q^\ast: p\in \mathcal{P}_{Q,k}^\ast(u)\}, \text{ where } (k,u)\in  \mathcal{K}_Q\times\mathcal{S}_Q \text{ such that } f=f_k(u).
			\end{align}
			This means each $f\in \mathcal{F}_2(p)$ corresponds to a pair $(k,u)\in  \mathcal{K}_Q\times\mathcal{S}_Q$. By \eqref{eq-k(u)} and \eqref{equ-multiplicity}, we infer
			\[|\mathcal{F}_2(p)|\lesssim \max_{u\in \mathbb S_Q'}K(u)\cdot |\mathbb S_Q^\ast(p)|\lessapprox M \Big(\frac{\delta}{r}\Big)^{-\epsilon^2} \Big(\frac{\delta}{\Delta}\Big)^{-O(1)}\Big(\frac{\delta}{r}\Big)^{f(s,t)}.\]
			
			Property \nref{T2} essentially follows from property \nref{P2}. For each $f\in \mathcal{F}_2$ and each $Q\in\mathcal{D}_r(\mathcal{P}_2(f))$, there is a unique pair $(u,k)\in \mathcal{S}_Q\times \mathcal{K}_Q$ such that $f=f_k(u)$. By \nref{P2} and \nref{R1}, for this $u\in \mathcal{S}_Q$, there is a cube family 
			\[\mathcal{Q}(u)\subset \mathcal{D}_\Delta\cap u(O(\Delta))\]
			satisfying $|\mathcal{Q}(u)|\sim d$ and \eqref{equ-largecardinality}. Since $f\in \mathcal{F}_1(u)$, we have $\mathcal{Q}(u)\subset \Gamma_f(O(\Delta))$. We take $\mathcal{Q}(f,Q):=\mathcal{Q}(u)$ and conclude (a) in \nref{T2}. It remains to show (b) in \nref{T2}. By Remark \ref{re-portion}, we know that $\cup\mathcal{Q}(f,Q)$ contains $\gtrapprox1$ portion of \[\mathcal{P}_2(f)\cap Q=\mathbb{P}_{Q,k}(u)=\mathcal{P}^\ast_{Q,k}(u)=\mathcal{P}''_{Q,k}(u).\] Since $\mathcal{P}''_{Q,k}(u)$ is uniform, we infer \[d\cdot |\mathcal{P}_2(f)\cap Q_\Delta|\gtrapprox |\mathbb{P}_{Q,k}(u)|\geq (\delta/r)^{-s+O(\eta_1)},\quad Q_\Delta\in \mathcal{Q}(f,Q).\] Then we deduce from Lemma \ref{lem-rdelta} that
			\begin{equation}
				\gamma_{\widetilde{\mathcal{P}}_2,t^\ast}\lesssim \Big(\frac{\delta}{\Delta}\Big)^{-t^\ast}\Big(\frac{\delta}{r}\Big)^{s-O(\eta_1)}\cdot d\cdot \gamma_{\mathcal{P}_2,t^\ast}.
			\end{equation}
			Moreover, since $\mathcal{D}_r(\mathcal{P}_2(f))|\gtrapprox|\mathcal{D}_r(\mathcal{P}_1(f))|\gtrapprox \tilde{\lambda} r^{-1}$ for each $f\in \mathcal{F}_2$, we obtain
			\begin{equation}
				|\widetilde{\mathcal{P}}_2(f)|\gtrapprox \tilde{\lambda}d (\Delta/r)\cdot \Delta^{-1}.
			\end{equation}
		\end{proof}

		\subsection*{Step 5: estimating $|\mathbf{E}_{\mathcal{F},\mathcal{P}}|$}  In this step, we will use Claim \ref{claimkeyl} to estimate $|\mathbf{E}_{\mathcal{F},\mathcal{P}}|$. We consider the following two cases.
		
		{\bf{Case I: $\Delta/r\leq (\delta/r)^{1-\sqrt{\eta}}$}.} Recall that $\eta_1=\eta\epsilon^3$ and $\tilde{\lambda}:=\lambda (r/\delta)^{1-s-9\eta_1}$. 
		
		When $t\geq s$, we use \eqref{equ-76}, \eqref{equ-constantM} and \nref{B1} to deduce
		\begin{equation}\label{equ-81}
			\begin{split}
				|\mathcal{F}_2(p)|& \lessapprox C_{\mathfrak{T},t,\epsilon}^{-1} C^{\eta^{-2}}r^{-\epsilon}\cdot \tilde{\lambda}^{-1/2} r^{(1-t)/2}\gamma_{\tilde{\mathcal{P}},t^\ast}^{1/2}\cdot  \Big(\frac{\delta}{r}\Big)^{-\epsilon^2} \Big(\frac{\delta}{\Delta}\Big)^{-O(1)}\Big(\frac{\delta}{r}\Big)^{f(s,t)}\\
				&\lessapprox C_{\mathfrak{T},t,\epsilon}^{-1} C^{\eta^{-2}}r^{-\epsilon}\cdot \lambda^{-1/2} \cdot \Big(\frac{\delta}{r}\Big)^{-\epsilon^2-O(\sqrt{\eta})} \delta^{(1-t)/2}\gamma_{\mathcal{P},t^\ast}^{1/2},
			\end{split}
		\end{equation}
		where we use $(\delta/\Delta)^{-O(1)}\leq (\delta/r)^{-O(\sqrt{\eta})}$ since $\Delta/r\leq (\delta/r)^{1-\sqrt{\eta}}$.
		
		When $t\leq s$, since $|\mathcal{P}_1(f)\cap Q|\geq (\delta/r)^{-s+O(\eta_1)}$ for each $Q\in \mathcal{D}_r(\mathcal{P}_1(f))$, we infer from Lemma \ref{lem-rdelta} that
		\begin{equation}\label{equ-80}
			\gamma_{\widetilde{\mathcal{P}},t^\ast} \leq \Big(\frac{\delta}{r}\Big)^{-O(\eta_1)} \Big(\frac{\delta}{r}\Big)^{-t^\ast}\Big(\frac{\delta}{r}\Big)^{s} \gamma_{\mathcal{P}_1,t^\ast}= \Big(\frac{\delta}{r}\Big)^{-O(\eta_1)} \Big(\frac{\delta}{r}\Big)^{s-t}\gamma_{\mathcal{P},t^\ast}.
		\end{equation} 
		Using \eqref{equ-80} in \eqref{equ-81}, we get the same estimate for $|\mathcal{F}_2(p)|$.
		
		Since $(\mathcal{F}_2,\mathcal{P}_2)$ is a refinement of $(\mathcal{F},\mathcal{P})$, \eqref{equ-81} implies
		\begin{equation}\label{equ-82}
			\begin{split}
				|\mathbf{E}_{\mathcal{F},\mathcal{P}}|&\geq |\mathbf{E}_{\mathcal{F}_2,\mathcal{P}_2}| \gtrapprox r^{\epsilon}\Big(\frac{\delta}{r}\Big)^{\epsilon^2+O(\sqrt{\eta})} \cdot C_{\mathfrak{T},t,\epsilon} C^{-\eta^{-2}}\lambda^{1/2} \delta^{(t-1)/2}\gamma_{\mathcal{P},t^\ast}^{-1/2}\sum_{f\in\mathcal{F}}|\mathcal{P}(f)|\\
				&=\Big(\frac{\delta}{r}\Big)^{\epsilon^2+O(\sqrt{\eta})-\epsilon} \cdot C_{\mathfrak{T},t,\epsilon} \delta^{\epsilon} C^{-\eta^{-2}}\lambda^{1/2} \delta^{(t-1)/2}\gamma_{\mathcal{P},t^\ast}^{-1/2}\sum_{f\in\mathcal{F}}|\mathcal{P}(f)|.
			\end{split}
		\end{equation}
		Since $r\geq \delta^{1-\eta_0\eta_1^{-1}}$, this establishes \eqref{equu-15} by choosing $O(\sqrt{\eta})<\epsilon^2$ so that
		\[\Big(\frac{\delta}{r}\Big)^{\epsilon^2+O(\sqrt{\eta})-\epsilon}\geq1\]
		noting that $\epsilon<1/100$.

		{\bf{Case II: $\Delta/r> (\delta/r)^{1-\sqrt{\eta}}$}.} Recall that $\mathcal{F}$ is a $(\delta,t)$-KT set. By Lemma \ref{lemma-KT}, we can extract a $(\Delta,t)$-KT
		set $\widetilde{\mathcal{F}}_2\subset \mathcal{F}_2$ such that $\Delta^t |\widetilde{\mathcal{F}}_2|\gtrapprox \delta^t|\mathcal{F}_2|$. For all $f\in\widetilde{\mathcal{F}}_2$, recall that $\mathcal{P}_2(f)$ is a refinement of $\mathcal{P}(f)$. Since $\mathcal{P}(f)$ is a $(\delta,\epsilon^2,C;\rho)$-set, $\mathcal{P}_2(f)$ is a $(\delta,\epsilon^2,CC_4;\rho)$-set for some $C_4\lessapprox1$. Since $|\mathcal{P}_2(f)\cap Q|$ are approximately constant for all $Q\in \mathcal{D}_r(\mathcal{P}_2(f))$, we infer that or $\mathcal{D}_r(\mathcal{P}_2(f))$ is a $(r,\epsilon^2,CC_5;\max\{\rho,r\})$-set for some $C_5\lessapprox1$. Recall the definition of $\widetilde{\mathcal{P}}_2(f)$ from (b) in Claim \ref{claimkeyl}. Since 
		\[|\widetilde{\mathcal{P}}_2(f) \cap Q|=|\mathcal{Q}(f,Q)|\sim d, \quad f\in \widetilde{\mathcal{F}}_2,\]
		$\widetilde{\mathcal{P}}_2(f)$ is a $(\Delta,\epsilon^2,CC_6;\max\{\rho,r\})$-set for some $C_6\lessapprox1$.
		
		Moreover, since $\rho\in [\delta,\delta^\eta]$ and $r\leq \delta^\eta$, we have $\max\{\rho,r\}\in [\Delta, \Delta^{\eta}]$. From \eqref{equ-tilde}, we know $\widetilde{\mathcal{P}}_2(f)$ is $\gtrapprox \tilde{\lambda}d (\Delta/r)$-dense. Now applying induction hypothesis at scale $\Delta$ to $(\widetilde{\mathcal{F}}_2,\widetilde{\mathcal{P}}_2)_\Delta$ gives
		\begin{equation}\label{equ-83}
			\begin{split}
				|\mathbf{E}_{\widetilde{\mathcal{F}}_2,\widetilde{\mathcal{P}}_2}| & \geq C_{\mathfrak{T},t,\epsilon} \Delta^{\epsilon} (CC_6)^{-\eta^{-2}} (\tilde{\lambda} d\Delta/r)^{1/2} \Delta^{(t-1)/2}\gamma_{\widetilde{\mathcal{P}}_2,t^\ast}^{-1/2} \cdot \sum_{f\in \widetilde{\mathcal{F}}_2}|\widetilde{\mathcal{P}}_2(f)|\\
				& \gtrapprox C_{\mathfrak{T},t,\epsilon} \Delta^{\epsilon} (CC_6)^{-\eta^{-2}} (\tilde{\lambda} d r^{-1})^{3/2} \Delta^{t/2}\gamma_{\widetilde{\mathcal{P}}_2,t^\ast}^{-1/2} |\widetilde{\mathcal{F}}_2|.
			\end{split}
		\end{equation}
		By using $\tilde{\lambda}=\lambda (r/\delta)^{1-s-9\eta_1}$, $\Delta^t |\widetilde{\mathcal{F}}_2|\gtrapprox \delta^t|\mathcal{F}_2|$ and $|\mathcal{P}_2(f)|\approx \lambda \delta^{-1}$, we get
		\begin{equation}\label{equ-84}
			\begin{split}
				|\mathbf{E}_{\widetilde{\mathcal{F}}_2,\widetilde{\mathcal{P}}_2}| & \gtrapprox C_{\mathfrak{T},t,\epsilon} \Delta^{\epsilon} (CC_6)^{-\eta^{-2}} \lambda^{3/2} (r/\delta)^{3(1-s)/2-O(\eta_1)}(d r^{-1})^{3/2} \Delta^{-t/2}\gamma_{\widetilde{\mathcal{P}}_2,t^\ast}^{-1/2} \cdot \Delta^t|\widetilde{\mathcal{F}}_2|\\
				&\gtrapprox C_{\mathfrak{T},t,\epsilon} \Delta^{\epsilon} (CC_6)^{-\eta^{-2}} \Big(\frac{\delta}{r}\Big)^{3s/2+O(\eta_1)}
				\Big(\frac{\delta}{\Delta}\Big)^{t/2}d^{3/2}\cdot \lambda^{1/2} \delta^{(t-1)/2}\gamma_{\widetilde{\mathcal{P}}_2,t^\ast}^{-1/2}\sum_{f\in\mathcal{F}_2}|\mathcal{P}_2(f)|.
			\end{split}
		\end{equation}
		By using \eqref{equ-largecardinality1}, we have
		\begin{equation}\label{equ-85}
			\begin{split}
				|\mathbf{E}_{\mathcal{F},\mathcal{P}}|\gtrapprox & C_{\mathfrak{T},t,\epsilon} \Delta^{\epsilon} (CC_6)^{-\eta^{-2}} \Big(\frac{\delta}{r}\Big)^{3s/2+O(\eta_1)}
				\Big(\frac{\delta}{\Delta}\Big)^{t/2}d^{3/2}\cdot \lambda^{1/2} \delta^{(t-1)/2}\gamma_{\widetilde{\mathcal{P}}_2,t^\ast}^{-1/2}\sum_{f\in\mathcal{F}_2}|\mathcal{P}_2(f)|\\
				& \cdot \Big(\frac{\delta}{\Delta}\Big)^{\epsilon^2} \Big(\frac{\delta}{r}\Big)^{-2s+f(s,t)}\Big(\frac{\Delta}{r}\Big)^{s-f(s,t)}\cdot d^{-1}\\
				\gtrapprox & \Delta^{\epsilon} C_6^{-\eta^{-2}}\Big(\frac{\delta}{r}\Big)^{O(\eta_1)}
				\Big(\frac{\delta}{\Delta}\Big)^{\epsilon^2} \Big(\frac{\delta}{\Delta}\Big)^{t/2} \Big(\frac{\delta}{r}\Big)^{-s/2+f(s,t)} \Big(\frac{\Delta}{r}\Big)^{s-f(s,t)}d^{1/2}\\
				& \cdot C_{\mathfrak{T},t,\epsilon}C^{-\eta^{-2}} \lambda^{1/2} \delta^{(t-1)/2}\gamma_{\widetilde{\mathcal{P}}_2,t^\ast}^{-1/2}\sum_{f\in\mathcal{F}}|\mathcal{P}(f)|.
			\end{split}
		\end{equation}
		
		In the sequel, we will apply \eqref{equ-85} and treat the cases $t^\ast> s$ and $t^\ast\leq s$ separately.
		
		{\bf Case $t^\ast> s$.} Let $w\in [\Delta,1]$, $f\in\widetilde{\mathcal{F}}_2$ and $Q_1\subset \mathcal{D}_w(\widetilde{\mathcal{P}}_2(f))$ such that
		\[\gamma_{\widetilde{\mathcal{P}}_2,t^\ast}\sim |\widetilde{\mathcal{P}}_2(f)\cap Q_1|_\Delta\cdot \Big(\frac{\Delta}{w}\Big)^{t^\ast}.\]
		If $w>r$, we use $|\mathcal{Q}(f,Q)|\sim d$ (here $Q\in \mathcal{D}_r(\mathcal{P}_2(f))$), $t^\ast>s$ and $\ga_{\tilde \cP, t^\ast}(f)\lesssim \ga_{\cP,t^\ast}(f)$ to obtain
		\begin{equation}\label{equ-88}
			\begin{split}
				\gamma_{\widetilde{\mathcal{P}}_2,t^\ast}&\lesssim |\mathcal{P}_2(f)\cap Q_1|_r\Big(\frac{r}{w}\Big)^{t^\ast}\cdot \Big(\frac{\Delta}{r}\Big)^{t^\ast}d\\
				&\lesssim \ga_{\tilde \cP, t^\ast}(f) \cdot \Big(\frac{\Delta}{r}\Big)^{s}d \lesssim \Big(\frac{\Delta}{r}\Big)^{s}d \cdot\gamma_{\mathcal{P},t^\ast}.
			\end{split}
		\end{equation}
		If $w\leq r$, take $Q\in \mathcal{D}_r(\mathcal{P}_2(f))$ such that $Q_1\subset Q$, then by property \nref{D2} and 
		\[|\mathcal{P}_2(f)\cap Q|\approx |\mathcal{P}_1(f)\cap Q|,\]
		we infer that $T_Q(\mathcal{P}_2(f)\cap Q)$ is a $(\delta/r,s,(\delta/r)^{-O(\eta_1)})$-set. This implies
		\begin{equation}\label{equ-91}
			|\mathcal{P}_2(f)\cap Q|_w\geq \Big(\frac{\delta}{r}\Big)^{O(\eta_1)}\Big(\frac{w}{r}\Big)^{-s}.
		\end{equation}
		Since $Q\in \mathcal{D}_r(\mathcal{P}_2(f))$, there are $(u.k)\in \mathcal{S}_Q\times\mathcal{K}_Q$ such that $f=f_k(u)$ and
		\[\mathcal{P}_2(f)\cap Q=\mathbb{P}_{Q,k}(u)=\mathcal{P}''_{Q,k}(u).\]
		Since $\mathcal{P}''_{Q,k}(u)$ is uniform and $|\cup\mathcal{Q}(f,Q)\cap \mathcal{P}_2(f)|\gtrapprox |\mathcal{P}_2(f)\cap Q|$, we have
		\[d|\mathcal{P}''_{Q,k}(u)\cap Q_\Delta|\gtrapprox |\mathcal{P}_2(f)\cap Q|_w|\mathcal{P}''_{Q,k}(u)\cap Q_1|, \quad Q_\Delta\in \mathcal{Q}(f,Q),\]
		which implies
		\begin{equation}\label{eq-uniQ}
			d\gtrapprox |\mathcal{P}_2(f)\cap Q|_w|\mathcal{P}_2(f)\cap Q_1|_\Delta, \quad Q_1\subset Q.
		\end{equation}
		Using \eqref{eq-uniQ}, \eqref{equ-91} and $\gamma_{\mathcal{P},t^\ast}\geq1$, we get
		\begin{equation}\label{equ-92}
			\gamma_{\widetilde{\mathcal{P}}_2,t^\ast} \lessapprox \Big(\frac{\delta}{r}\Big)^{-O(\eta_1)} \Big(\frac{w}{r}\Big)^{s}d\cdot \Big(\frac{\Delta}{w}\Big)^{t^\ast}\leq \Big(\frac{\delta}{r}\Big)^{-O(\eta_1)} \Big(\frac{\Delta}{r}\Big)^{s}d\cdot \gamma_{\mathcal{P},t^\ast}.
		\end{equation}
		Hence \eqref{equ-92} holds for both cases. By substituting \eqref{equ-92} into \eqref{equ-85}, we infer
		\begin{equation}\label{equ-93}
			\begin{split}
				|\mathbf{E}_{\mathcal{F},\mathcal{P}}|\gtrapprox & \Delta^{\epsilon} C_6^{-\eta^{-2}}\Big(\frac{\delta}{r}\Big)^{O(\eta_1)}
				\Big(\frac{\delta}{\Delta}\Big)^{\epsilon^2} \Big(\frac{\delta}{\Delta}\Big)^{t/2} \Big(\frac{\delta}{r}\Big)^{-s/2+f(s,t)} \Big(\frac{\Delta}{r}\Big)^{s/2-f(s,t)}\\
				& \cdot C_{\mathfrak{T},t,\epsilon}C^{-\eta^{-2}} \lambda^{1/2} \delta^{(t-1)/2}\gamma_{\mathcal{P},t^\ast}^{-1/2}\sum_{f\in\mathcal{F}}|\mathcal{P}(f)|.
			\end{split}
		\end{equation}
		Since $t^\ast>s$, we have $f(s,t)=(s-t)/2$. Using this in \eqref{equ-93} gives
		\begin{equation}\label{equ-94}
			\begin{split}
				|\mathbf{E}_{\mathcal{F},\mathcal{P}}|\gtrapprox C_6^{-\eta^{-2}}\Big(\frac{\delta}{\Delta}\Big)^{-\epsilon} \Big(\frac{\delta}{r}\Big)^{O(\eta_1)}
				\Big(\frac{\delta}{\Delta}\Big)^{\epsilon^2} \cdot C_{\mathfrak{T},t,\epsilon}\delta^\epsilon C^{-\eta^{-2}} \lambda^{1/2} \delta^{(t-1)/2}\gamma_{\mathcal{P}_2,t^\ast}^{-1/2}\sum_{f\in\mathcal{F}}|\mathcal{P}(f)|.
			\end{split}
		\end{equation}
		Recall that $\Delta/r> (\delta/r)^{1-\sqrt{\eta}}$ and $\eta_1=\eta\epsilon^3$, then we deduce
		\begin{equation}\label{equ-95}
			\begin{split}
				|\mathbf{E}_{\mathcal{F},\mathcal{P}}|\gtrapprox C_6^{-\eta^{-2}}\Big(\frac{\delta}{\Delta}\Big)^{-\epsilon+\epsilon^2+O(\sqrt{\eta})}\cdot C_{\mathfrak{T},t,\epsilon}\delta^\epsilon C^{-\eta^{-2}} \lambda^{1/2} \delta^{(t-1)/2}\gamma_{\mathcal{P}_2,t^\ast}^{-1/2}\sum_{f\in\mathcal{F}}|\mathcal{P}(f)|.
			\end{split}
		\end{equation}
		Since $r\geq \delta^{1-\eta_0\eta_1^{-1}}$ and $\Delta/r> (\delta/r)^{1-\sqrt{\eta}}$,
		\[\Big(\frac{\delta}{\Delta}\Big)^{-\epsilon+\epsilon^2+O(\sqrt{\eta})}\geq \Big(\frac{\delta}{r}\Big)^{(-\epsilon+\epsilon^2+O(\sqrt{\eta}))\sqrt{\eta}}\geq \delta^{(-\epsilon+\epsilon^2+O(\sqrt{\eta}))\sqrt{\eta}\eta_0\eta_1^{-1}}.\]
		Recall our choices for $\eta,\delta_0$ from {\bf Step 1} and {\bf Step 4}. Finally, by taking \[\eta<\min\{\eta'/2,\epsilon^4, \bar{\eta}\}\]
		so that $-\epsilon+\epsilon^2+O(\sqrt{\eta})<\epsilon/3$ and taking 
		\[\delta<\delta_0<\min\{\delta_1,\delta_2,\bar{\delta}^{1/h(\epsilon^4)}\}\]
		sufficiently small, we obtain \eqref{equu-15}.

		{\bf Case $t^\ast \leq s$.} In this case we have
		\begin{equation}\label{equ-96}
			f(s,t)=\left\{
			\begin{array}{lll}
				0, \quad &\text{if}~~t^\ast=t,\\
				s-1, &~\text{if}~t^\ast=2-t.
			\end{array}\right.
		\end{equation}
		By substituting \eqref{equ-density} into \eqref{equ-85} and using \eqref{equ-96}, we also get \eqref{equ-95} and complete the whole proof.
		
	\end{proof}

	
	\section{Fourier decay}\label{sec-Fourierdecay}\label{sec4}
	
	Before the statement of the main result, let us recall the following definition.
	
	\begin{definition}\label{def:frostman}
		Let $u\in[0,2]$ and $C\geq 1$. A Borel measure $\mu$ on $\mathbb R^2$ is called a $(u,C)$\textit{-Frostman measure} if $$\mu(B(x,r))\leq C r^u$$ for all $x\in \mathbb R^2$ and $r>0$. We simply say $\mu$ is a $u$-Frostman measure when $C>0$ is irrelevant.
	\end{definition}

	In this part, we let $\mathbf{c}\geq \mathbf{d}>0$ and consider any fixed convex function satisfying
	\begin{equation}\label{eq-abc}
		\|f\|_{C^3([-6,6])}\leq \mathbf{c}\quad \text{and} \quad \min_{x\in [-6,6]}|f''(x)|\geq \mathbf{d}.
	\end{equation}
	Moreover, we will always use the notation \[\Gamma:=\{(x,f(x)):x\in [-1,1]\}.\]

	We aim to prove Theorem \ref{thm-L6decay} which we recall below.
	
	\begin{thm}\label{thm-L7decay}
		Let $s\in (0, 2/3]$. Let $\mu$ be an $s$-Frostman measure supported on $\Gamma$. Then for any $\epsilon \in (0,1)$ and $R\geq1$ we have
		\begin{equation}\label{equ-decay}
			\|\hat{\mu}\|_{L^6(B_R)}^6 \leq C(\mathbf{c},\mathbf{d},s,\epsilon) R^{2-5s/2+\epsilon}.
		\end{equation}
	\end{thm}
	The key tool for proving Theorem \ref{thm-L7decay} is an incidence estimate we will establish in Subsection \ref{subs4.1}.

	\subsection{An incidence estimate}\label{subs4.1}
	We start with a definition for transversal families.
	\begin{definition}[Rectangular KT-condition for transversal families]\label{def:rectangleofF}
		Let $\tau\in [0,2]$, $C\geq1$ and $\mathfrak{T}\geq1$. Let $\mathcal{F}\in B_{C^2}(1)$ be a $\mathfrak{T}$-transversal family on $[-2,2]$. We say that $\mathcal{F}$ is a rectangular $(\delta,\tau,C)$-KT set, if for any $x\in [-2,2]$, $I_r\in \mathcal{D}_r([-1,1])$ and $I_{r'}\in\mathcal{D}_{r'}([-1,1])$, we have
		\[|A_x(\mathcal{F})\cap (I_{r'}\times I_r)|_\delta \leq C \big(\frac{\sqrt{rr'}}{\delta}\big)^{\tau}, \quad \delta\leq r', r.\]
		Here $A_x(f)=(f(x),f'(x))$ is defined as in Lemma \ref{lem.indentify}. We simply say $\mathcal{F}$ is a rectangular $(\delta,\tau)$-KT set if $C$ is irrelevant.
	\end{definition}
	
	We aim to prove the following incidence estimate between a cube family and a transversal family.
	
	\begin{thm}\label{thm-recincidence}
		Let $s\in (0,1]$, $C_1,C_2\geq1$ and $\mathfrak{T}\geq1$. For any $\epsilon>0$, there exists $\delta_0=\delta_0(\mathfrak{T},s,\epsilon)>0$ such that the following holds for all $\delta\in (0,\delta_0]$.
		
		Let $\mathbb{F}\subset B_{C^2}(1)$ be a $\delta$-separated $\mathfrak{T}$-transversal family over $[-2,2]$. Assume that 
		\begin{itemize}
			\item $\mathbb{F}$ is a rectangular $(\delta,2s,C_1)$-KT set;
			\item for each $g\in \mathbb{F}$, assume that there exists a $(\delta,\sigma,C_2)$-KT cube family 
			\[\mathbb{P}(g)\subset\{p\in \mathcal{D}_\delta:p\cap \Gamma_g\neq \emptyset\}.\]
			Here $\sigma=\min\{2s,2-2s\}$.
		\end{itemize}
		Write $\mathbb{P}=\cup_{g\in \mathbb{F}}\mathbb{P}(g)$, then
		\begin{equation}\label{equ-recincidd}
			\mathcal{I}(\mathbb{F},\mathbb{P}):=\sum_{g\in \mathbb{F}}|\mathbb{P}(g)|\leq C_1^{2/3}C_2^{1/3}\delta^{-2s/3-\epsilon}|\mathbb{P}|^{2/3} |\mathbb{F}|^{1/3}.
		\end{equation}
	\end{thm}

	To prove Theorem \ref{thm-recincidence}, we need to apply Theorem \ref{main-twoends} which we restate below to track the dependence of the Katz-Tao constant of $\mathcal{F}$.
	
	\begin{thm}\label{thm-trackKTC}
		Let $t\in (0,2)$, $\mathfrak{T}\geq1$ and $\lambda\in [0,1]$. For any $\epsilon\in (0,1)$, there exists a small constant $\delta_0=\delta_0(t,\epsilon,\mathfrak{T})>0$ such that the following holds for any $\delta\in (0,\delta_0]$.
		
		Let $(\mathcal{F},\mathcal{P})_\delta$ be a given configuration. Here $\mathcal{F}\subset B_{C^2}(1)$ is $\mathfrak{T}$-transversal over $[-2,2]$. Assume that
		\begin{itemize}
			\item $\mathcal{F}$ is $\delta$-separated and a $(\delta,t,C)$-KT set with $C\geq1$;
			\item for each $g\in \mathcal{F}$, $\mathcal{P}(g)$ is $(\epsilon_1, \epsilon_2)$-two-ends and $\lambda$-dense.
		\end{itemize}
		Write $t^\ast=\{t,2-t\}$. Then,
		\begin{equation}\label{eq-main-twoends}
			|\mathbf{E}_{\mathcal{F},\mathcal{P}}|\geq \delta^{\epsilon+t\epsilon_1/2}  C^{-1}\delta^{(t-1)/2} \ga_{\cP,t^\ast}^{-1/2}\cdot \lambda^{1/2}\sum_{g\in\mathcal{F}} |\mathcal{P}(g)|.
		\end{equation}
	\end{thm}
	
	\begin{proof}
		First, after dyadic pigeonholing, we may assume that $|\mathcal{P}(g)|$ are roughly constant for all $g\in \mathcal{F}$. By Lemma \ref{lemma-KT}, there exists a $(\delta,t,1)$-KT subset $\mathcal{F}_1\subset\mathcal{F}$ with $|\mathcal{F}_1|\gtrsim |\mathcal{F}|/C$, where $C\geq1$ is the Katz-Tao constant of $\mathcal{F}$.
		
		Consequently, we apply Theorem \ref{main--twoends} to $(\mathcal{F}_1,\mathcal{P})_\delta$ to obtain what we desire, after observing that
		\[\sum_{g\in\mathcal{F}_1} |\mathcal{P}(g)| \gtrsim \sum_{g\in\mathcal{F}} |\mathcal{P}(g)|/C.\]
	\end{proof}

	To apply the two-ends estimate \eqref{eq-main-twoends}, we will need the following lemma. It says that although $\mathbb{P}(g)$ (in Theorem \ref{thm-recincidence}) generally does not satisfy the two-ends condition, we can always find a two-ends sub-family of $\mathbb{P}(g)$. The statement (and proof) are essentially \cite[Lemma 4.7]{doprod} so we omit to prove it here.
	
	\begin{lemma}\label{lem-fixtwoends}\label{2endsredu}
		Let $(\mathbb{F},\mathbb{P})_\delta$ be a given configuration. Assume that for each $g\in\mathbb{F}$, the shading $\mathbb{P}(g)$ is a $(\delta,s)$-KT set. Then for any $\epsilon\in (0,\sqrt{s})$ and $g\in \mathbb{F}$, there exists an interval $J_g\in \mathcal{D}_{L_g}([-1,1])$ such that
		\begin{itemize}
			\item [(1)] the length $L_g=|J_g|\gtrsim (\delta^s|\mathbb{P}(g)|)^{\tfrac{1}{s-\epsilon^2}}$;
			\item [(2)] write $\mathbb{P}(J_g):=\{p\in\mathbb{P}(g): \pi_x(p)\subset J_g\}$, then $|\mathbb{P}(J_g)|\geq L_g^{\epsilon^2}|\mathbb{P}(g)|$ and
			\begin{equation}\label{equ-IP}
				|\mathbb{P}(J_g)\cap B(x,L_g(\delta/L_g)^\epsilon)|_\delta\leq (\delta/L_g)^{\epsilon^3}|\mathbb{P}(J_g)|, \quad \forall x\in \R.
			\end{equation}
		\end{itemize}
	\end{lemma}

	\begin{proof}[Proof of Theorem \ref{thm-recincidence}]
		The proof is organized in three steps.
		
		{\bf Step 1: initial reduction.} It suffices to show \eqref{equ-recincidd} when $\epsilon <\sqrt{\sigma/2}$. Indeed, by the $(\delta,\sigma, C_2)$-KT condition of $\mathbb{P}(g)$, we have the trivial estimate
		\begin{equation}\label{esti1}
			\mathcal{I}(\mathbb{F},\mathbb{P})\leq C_2\delta^{-\sigma}|\mathbb{F}|.
		\end{equation}
		On the other hand, for each $p\in\mathbb{P}$, note that
		\[\mathbb{F}(p):=\{g\in\mathbb{F}:p\in\mathbb{P}(g)\}\subset \{g\in\mathbb{F}:p\cap \Gamma_g\neq \emptyset\}\]
		and $A_{x_p}(\mathbb{F}(p))$ is contained in a $O(\delta\times 1)$-rectangle, where $x_p$ is the center of $\pi_x(p)$. By the rectangular $(\delta,2s,C_1)$-KT condition of $\mathbb{F}$, we get another trivial estimate
		\begin{equation}\label{esti2}
			\mathcal{I}(\mathbb{F},\mathbb{P})=\sum_{p\in\mathbb{P}}|\mathbb{F}(p)|\lesssim C_1\delta^{-s}|\mathbb{P}|.
		\end{equation}
		Interpolating the two bounds \eqref{esti1} and \eqref{esti2}, we obtain
		\[\mathcal{I}(\mathbb{F},\mathbb{P})\lesssim C_1^{2/3}C_2^{1/3}\delta^{-2s/3-\sigma/3}|\mathbb{P}|^{2/3} |\mathbb{F}|^{1/3}.\]
		If $\epsilon\geq \sqrt{\sigma/2}$, this easily implies \eqref{equ-recincidd}.
		
		Therefore, in the sequel we fix $\epsilon\in (0,\sqrt{\sigma/2})$. There will be various losses of the form $(1/\delta)^{O(\epsilon)}$ that will be hidden in the notation $\les$ and $\approx$. The constants in $\sim$, $\gtrsim$ and $\lesssim$ may also depend on $\mathfrak{T}$.
		
		To proceed, we pigeonhole a subset $\mathbb{F}_1\subset\mathbb{F}$ and a number $N_0\in 2^\N$ such that $|\PP(g)|\sim N_0$ for all $g\in \mathbb{F}_1$ and
		\[N_0|\mathbb{F}_1|\sim\mathcal{I}(\mathbb{F}_1,\PP)\approx \mathcal{I}(\mathbb{F},\PP),\]
		where we recall $\mathbb{P}=\cup_{p\in\mathbb{P}}\mathbb{P}(g)$. Moreover, we may assume that $|\mathbb{P}(g)|>\delta^{-\epsilon}$ for all $g\in\mathbb{F}_1$. Otherwise, we deduce from $|\mathbb{P}(g)|\leq\delta^{-\epsilon}$ that
		\begin{equation}\label{esti3}
			\mathcal{I}(\mathbb{F},\mathbb{P})\lessapprox \delta^{-\epsilon}|\mathbb{F}|\leq C_2\delta^{-\epsilon}|\mathbb{F}|.
		\end{equation}
		Interpolating the two bounds \eqref{esti2} and \eqref{esti3}, we obtain
		\[\mathcal{I}(\mathbb{F},\mathbb{P})\lessapprox C_1^{2/3}C_2^{1/3}\delta^{-2s/3-\epsilon/3}|\mathbb{P}|^{2/3} |\mathbb{F}|^{1/3}.\]
		By choosing $\delta$ small enough, this implies \eqref{equ-recincidd}.
		
		Next, we apply Lemma \ref{2endsredu} with $\epsilon$ replaced by $\epsilon/2$ to each $g\in \mathbb{F}_1$ to produce an interval $J_g\in \mathcal{D}_{L_g}([-1,1])$ with length (use $\epsilon\geq \sqrt{2}\epsilon^2$ and $\sigma-\epsilon^2\leq 1$)
		\[L_g\gtrsim (\delta^\sigma|\mathbb{P}(g)|)^{\tfrac{1}{\sigma-\epsilon^2}}\geq \delta^{\tfrac{\sigma-\epsilon}{\sigma-\epsilon^2}}\geq \delta^{1-O(\epsilon^2)}\]
		such that the collection $\PP(J_g)$ is $(\epsilon/2,\epsilon^3/8)$-two-ends and $|\mathbb{P}(J_g)|\geq L_g^{\epsilon^2}|\mathbb{P}(g)|$. This means that the sub-family $\mathbb{P}(J_g)$
		after being rescaled by $1/L_g$, forms an $(\epsilon/2,\epsilon^3/8)$-two-ends subset at scale $\delta/L_g,$ that is \eqref{equ-IP} holds. We note that $|\mathbb{P}(g)|\geq \delta^{-\epsilon}$ is mainly used to get the lower bound for $L_g$.

		We may also assume that $L_g\sim L$, $|\mathbb{P}(J_g)|\sim N$ for all $g\in \mathbb{F}_1$, for some fixed dyadic parameters $L>\delta^{1-O(\epsilon^2)}$ and $N\approx N_0$.
		The parameter $N_0$ will not be mentioned again, as it is replaced by $N$. Note that
		\begin{equation}\label{ jutgugtg9 yh =390h yioh0yi0h-}
			\mathcal{I}(\mathbb{F},\PP)\approx N|\mathbb{F}_1|.
		\end{equation}
		We rename $\PP(g)$ to refer to the smaller collection $\PP(J_g)$, so in the sequel we assume that $\PP(g)$ is $(\epsilon/2,\epsilon^3/8)$ two-ends with
		\begin{equation}
			\label{;lckjfvjvjiobiytpio}
			\pi_x(\mathbb{P}(g))\subset J_g \quad \text{ and } \quad |\PP(g)|\sim N,\;\;\forall g\in\mathbb{F}_1.
		\end{equation}

		{\bf Step 2: local decomposition.}
		We decompose $\mathbb{F}_1$ according to $J_g$ determined by Lemma \ref{lem-fixtwoends}. For each $R\in \mathcal{D}_L(\mathbb{P})$, define
		\[\mathbb{F}_R:=\{g\in \mathbb{F}_1: R\in \mathcal{D}_L(\mathbb{P}(g))\}.\]
		Note that we assume $\pi_x(\mathbb{P}(g))\subset J_g$. Since $g\in \mathbb{F}\subset B_{C^2}(1)$, the graph of $g$ above $J_g$ intersects one and at most two $R\in \mathcal{D}_L$, i.e., $g$ belongs to at most two $\mathbb{F}_R$, thus
		\[|\mathbb{F}_1|\sim \sum_{R\in \mathcal{D}_L(\mathbb{P})}|\mathbb{F}_R|.\]

		Now we do the same construction as in Section \ref{sec3}. For each $R\in\mathcal{D}_L(\mathbb{P})$, let \[\Gamma_R=\{\Gamma_{g|_{I_R}}: g\in \mathbb{F}_R\}.\] 
		Here $I_R=\pi_x(R)$ and $\Gamma_{g|_{I_R}}$ means the graph of $g$ above $I_R$. Two curve segments $\Gamma_{g|_{I_{R}}},\Gamma_{h|_{I_{R}}}$ are \emph{comparable} if $|g(x) - h(x)| \leq 3\delta$ for all $x \in I_{R}$. Let $\mathbb{S}_{R}$ be a maximal set of \emph{incomparable} curve segments in $\Gamma_R$. 
		Since each $u\in\mathbb{S}_R$ is a truncated graph of one function in $\mathbb{F}_R$, we use $g_u$ to denote this function. Recall that for each $R\in \mathcal{D}_L(\mathbb{P})$ and $g\in\mathbb{F}_R$, we have 
		\[|\{u\in \mathbb{S}_R:\Ga_{g|_{I_R}}\subset u(3\delta)\}|\sim_{\mathfrak{T}} 1.\]
		This is property (ii) for incomparable segments stated on page 11. By this approximate identity and the fact that $|\mathbb{P}(g)|\sim N$ for all $g\in\mathbb{F}_1$, we get the following claim.
		
		\begin{claim}
			We have
			\begin{equation}
				\mathcal{I}(\mathbb{F}_1,\PP)=\sum_{g\in \mathbb{F}_1}|\mathbb{P}(g)|\sim \sum_{R \in \mathcal{D}_L(\mathbb{P})}\sum_{u\in \mathbb{S}_R}\sum_{\substack{g\in u(3\delta)\cap \mathbb{F}_R}}|\PP(g)|,
			\end{equation}
			where $u(3\delta)\cap \mathbb{F}_R:=\{g\in \mathbb{F}_R: \Ga_{g|_{I_R}}\subset u(3\delta)\}$.
		\end{claim}
		We pick $M\in 2^\N$ such that
		\begin{equation}\label{equ-in11}
			\mathcal{I}(\mathbb{F}_1,\PP)\approx \sum_{R \in \mathcal{R}^\ast}\sum_{u\in \mathbb{S}^\ast_R}\sum_{\substack{g\in u(3\delta)\cap \mathbb{F}_R}}|\PP(g)|,
		\end{equation}
		where $\mathcal{R}^\ast\subset \mathcal{D}_L(\mathbb{P})$ is pigeonholed such that \[\mathbb{S}_R^\ast=\{u\in \mathbb{S}_R:\;|u(3\delta)\cap \mathbb{F}_R|\sim M\}, \quad R\in \mathcal{R}^\ast.\]
		Write
		$\mathbb{S}^\ast=\cup_{R\in\mathcal{R}^\ast}\mathbb{S}_R^\ast$. Note that due to \eqref{;lckjfvjvjiobiytpio}, \eqref{equ-in11} and the choice of $M$ we have
		\begin{equation}
			\label{cjuuig8 rtui90 yti h9yi 09683}
			|\mathbb{S}^\ast|\sim \sum_{R\in\mathcal{R}^\ast}|\mathbb{S}_R^\ast|\approx \frac{|\mathbb{F}_1|}{M}.
		\end{equation}
		Now for $R\in\mathcal{R}^\ast$ and $u\in \mathbb{S}_R^\ast$, recall $u$ is a part of the graph of $g_u\in \mathbb{F}_R\cap u(3\delta)$, then define
		\[\mathbb{P}(u):=\mathbb{P}(g_u).\]
		By using \eqref{equ-in11}, \eqref{cjuuig8 rtui90 yti h9yi 09683} and noting $|\mathbb{F}_R\cap u(3\delta)|\sim M$, we infer
		\begin{equation}\label{equ-icd}
			\mathcal{I}(\mathbb{F},\mathbb{P}_1)\approx M\cdot \mathcal{I}(\mathbb{P},\mathbb{S}^\ast):=M \cdot\sum_{u \in \mathbb{S}^\ast}|\mathbb{P}(u)|.
		\end{equation}
		Moreover, for each $u \in \mathbb{S}_R^\ast$ $(R \in \mathcal{R}^\ast)$, $\mathbb{P}(u)$ is a $(\delta,\sigma,C_2)$-KT set and an $(\epsilon/2,\epsilon^3/8)$-two-ends set, since the same properties hold for each $\mathbb{P}(g_u)$.

		{\bf Step 3: applying two-ends Furstenberg inequality.} We will first estimate the incidence in each $R\in \mathcal{R}^\ast$: \[\mathcal{I}(\mathbb{P}^\ast_R,\mathbb{S}^\ast_R):=\sum_{u\in \mathbb{S}^\ast_R}|\mathbb{P}(u)|,\]
		where $\mathbb{P}^\ast_R=\cup_{u\in\mathbb{S}_R^\ast}\mathbb{P}(u)$. Recall that for any $u\in\mathbb{S}_R^\ast$, $\mathbb{P}(u)=\mathbb{P}(g_u)$ can intersect at most two $R\in\mathcal{D}_L$. Therefore, $\mathbb{P}^\ast_R$ lies in the union of three $L$-cubes which have the same $x$-projection as $R$ and satisfies
		\begin{equation}\label{ewq-pro1}
			\sum_{R\in \mathcal{R}^\ast}|\mathbb{P}^\ast_R|\lesssim \sum_{R\in \mathcal{D}_L(\mathbb{P})}|\mathbb{P}\cap R|=|\mathbb{P}|.
		\end{equation}

		Our plan is to apply Theorem \ref{thm-trackKTC} to $(\mathbb{P}^\ast_R,\mathbb{S}^\ast_R)$ after rescaling. Now let $(x_0,y_0)\in R$ be the center of $R\in\mathcal{R}^\ast$, then define 
		\[T(u):=(g_u)_{(x_0,y_0),L}=\frac{g_u(L\cdot+x_0)-y_0}{L}, \quad u\in\mathbb{S}^\ast_R,\]
		and
		\[T(\mathbb{S}_R^\ast):=\{T(u):u\in\mathbb{S}_R^\ast\}.\]
		We also abuse notation to define
		\[T(x,y)=\frac{(x,y)-(x_0,y_0)}{L}, \quad (x,y)\in\R^2.\]
		By Lemma \ref{lem2} and since segments in $S_R^\ast$ are incomparable, we know that $T(\mathbb{S}_R^\ast)$ is a $\sim\delta/L$-separated transversal family, see also \cite[Claim A.24]{OPY2025}. Moreover, $T(\mathbb{P}_R^\ast)$ is a set of $\delta/L$-cubes. Before applying Theorem \ref{thm-trackKTC}, let us verify the following properties.
		
		\begin{claim}\label{clll2}
			The rescaled configuration $(T(\mathbb{S}_R^\ast),T(\mathbb{P}_R^\ast))_{\delta/L}$ satisfies the following properties.
			\begin{itemize}
				\item [(i)] $|\mathbb{F}_1|\cdot\sum_{R\in \mathcal{R}^\ast} |T(\mathbb{P}_R^\ast)|\lessapprox M |\mathbb{P}|\sum_{R\in \mathcal{R}^\ast} |T(\mathbb{S}_R^\ast)|$.
				
				\item [(ii)] $T(\mathbb{S}_R^\ast)$ is a $(\delta/L, 2s, O(\tfrac{C_1}{ML^s}))$-KT set.
				
				\item [(iii)] For each $\bar{g}=T(u)\in T(\mathbb{S}_R^\ast)$, write \[T(\mathbb{P}_R^\ast)(\bar{g}):=\{T(p)\in T(\mathbb{P}_R^\ast): p\in \mathbb{P}(g_u)\}.\] Then $T(\mathbb{P}_R^\ast)(\bar{g})$ is $\gtrapprox N(\delta/L)$-dense, $(\delta/L,\sigma,C_2)$-KT and $(\epsilon/2,\epsilon^3/8)$-two-ends.
			\end{itemize}
		\end{claim}
		
		\begin{proof}
			The first property follows directly from \eqref{ewq-pro1} and \eqref{cjuuig8 rtui90 yti h9yi 09683}.
			
			Property (ii) will be derived from the rectangular Katz-Tao condition of $\mathbb{F}$. Similar to Claim \ref{claimm1}, it suffices to show
			\[|\{v\in\mathbb{S}_R^\ast: v\subset u(10r)\}|\lesssim \tfrac{C_1}{ML^s} \big(\frac{r}{\delta}\big)^{2s}, \quad \forall u\in \mathbb{S}_R^\ast, \quad \forall r\in [\delta,L].\]
			Since $|\mathbb{F}_R\cap v(3\delta)|\sim M$ for each $v\in \mathbb{S}_R^\ast$, we have
			\[|\{v\in\mathbb{S}_R^\ast: v\subset u(10r)\}| \lesssim M^{-1} |\mathbb{F}_1\cap  u(20r)|_\delta, \]
			where $\mathbb{F}_1\cap  u(10r)=\{g\in\mathbb{F}_1:\Ga_{g|_{\pi_x(u)}}\subset u(20r)\}$. Moreover, $\mathbb{F}_1\cap  u(20r)$ is contained in a $C^2$-ball of radius $\sim_\mathfrak{T} r/L$ around $g_u$, see \cite[Claim A.15]{OPY2025} for the proof. Take any $z_0\in \pi_x(u)$, then we notice that 
			\[A_{z_0}(\mathbb{F}_1\cap  u(20r))=\{A_{z_0}(g):g\in \mathbb{F}_1\cap  u(20r)\}\]
			can be covered by an $O(r\times r/L)$-rectangle, thus by the rectangular KT-condition  of $\mathbb{F}_1$ we infer
			\begin{equation}
				|\{v\in\mathbb{S}_R^\ast: v\subset u(10r)\}| \lesssim \frac{C_1}{M} \big(\frac{\sqrt{r^2/L}}{\delta}\big)^{2s}=\frac{C_1}{ML^s} \big(\frac{r}{\delta}\big)^{2s},
			\end{equation}
			as desired.
			
			Property (iii) is automatic as the same holds for $\mathbb{P}(g_u)$, noting that $|\mathbb{P}(g)|\approx N$ for all $g\in\mathbb{F}_1$ and the KT-condition of $\mathbb
			P(g_u)$ is preserved under rescaling.   
		\end{proof}

		Given Claim \ref{clll2}, we are able to apply Theorem \ref{thm-trackKTC} to the pair $(T(\mathbb{S}_R^\ast),T(\mathbb{P}_R^\ast))_{\delta/L}$ with $\epsilon$ replaced by $\epsilon/10$ and after rescaling back we get
		\[|\mathbb{P}_R^\ast|\gtrapprox \Big(\frac{\delta}{L}\Big)^{\epsilon/10+s\epsilon/2} C_2^{-1/2}\Big(\frac{N\delta}{L}\Big)^{1/2} \Big(\frac{\delta}{L}\Big)^{(2s-1)/2}(ML^s/C_1) \sum_{u\in \mathbb{S}_R^\ast}|\mathbb{P}(u)|.\]
		Here we use that $(2s)^\ast=\sigma$ and $\gamma_{\mathbb{P}, \sigma}\lesssim C_2$ since $\mathbb{P}(g)$ is $(\delta,\sigma,C_2)$-KT. Note that we choose $\delta$ small such that
		\[\frac{\delta}{L}\leq \delta^{O(\epsilon)}\leq \delta_1,\]
		where $\delta_1$ is the constant determined by Theorem \ref{thm-trackKTC}. Then we infer
		\[N^{3/2} M |\mathbb{S}_R^\ast|\lessapprox \Big(\frac{\delta}{L}\Big)^{-3\epsilon/5} C_1C_2^{1/2}\delta^{-s} |\mathbb{P}_R^\ast|.\]
		With summation in $R$ and use \eqref{ewq-pro1}, \eqref{cjuuig8 rtui90 yti h9yi 09683} we obtain
		\[N^{3/2} |\mathbb{F}_1| \lessapprox \Big(\frac{\delta}{L}\Big)^{-3\epsilon/5}C_1C_2^{1/2}\delta^{-s} |\mathbb{P}|.\]
		Finally, by \eqref{ jutgugtg9 yh =390h yioh0yi0h-} we get the desired inequality
		\[\mathcal{I}(\mathbb{F},\PP)\approx N|\mathbb{F}_1|\lessapprox \delta^{-2\epsilon/5} \Big(C_1C_2^{1/2}\delta^{-s}\frac{|\mathbb{P}|}{|\mathbb{F}_1|}\Big)^{2/3}|\mathbb{F}_1|=\delta^{-2s/3-2\epsilon/5}(C_1|\mathbb{P}|)^{2/3}(|C_2\mathbb{F}_1|)^{1/3}.\]
		Taking $\delta>0$ small enough, we conclude \eqref{equ-recincidd} and complete the proof.
		
	\end{proof}

	Next, recall Definition \ref{def:vertical}, then we introduce the `rectangular' KT-condition for subsets in the plane. This should be compared with \cite[Definition 1.4]{doprod}.
	
	\begin{definition}\label{def:recKT}
		Let $\delta\in (0,1]$ and $s\in [0,2]$. Let $f$ be any $C^3$-function satisfying
		\[\|f\|_{C^3([-6,6])}\leq \mathbf{c}\quad \text{and} \quad \min_{x\in [-6,6]}|f''(x)|\geq \mathbf{d}.\]
		A bounded subset $E\subset \R^2$ is called an \emph{$f$-rectangular $(\delta,s,C)$-KT set} if the following non-concentration condition holds. 
		
		For any $\de \leq r',r$ and $(a,b)\in [-2,2]\times\R$, let $\gamma$ be a curve segment of the graph of 
		\[x\mapsto f(x-a)+b\]
		with arc length $r$ and let $\mathbf{R}=\gamma(r')$ be the vertical $r$-neighbourhood of $\gamma$, then we have
		\[|E\cap\mathbf{R}|_\delta\leq C \Big(\frac{\sqrt{r'r}}{\delta}\Big)^{s}.\]
		We simply say $E$ is an \emph{$f$-rectangular $(\delta,s)$-KT set} if the constant $C$ is irrelevant.
	\end{definition}
	
	To prove Theorem \ref{thm-L7decay}, we need a dual form of Theorem \ref{thm-recincidence} for translations of a fixed convex function. Recall that $f\in C^3([-6,6])$ satisfies
	\[\|f\|_{C^3([-6,6])}\leq \mathbf{c}\quad \text{and} \quad \min_{x\in [-6,6]}|f''(x)|\geq \mathbf{d}.\]
	Let $P\subset [-2,2]\times \R$ be a $\delta$-separated subset and
	\begin{equation}\label{eq-defofFtrans}
		\mathcal{F}=\{f_{a,b}:(a,b)\in P\},
	\end{equation}
	where $f_{a,b}=b-f(a-\cdot)$. By the same argument as in \cite[Example 1.5]{OPY2025}, $\mathcal{F}$ is a $C_{\mathbf{c},\mathbf{d}}$-transversal family on $[-2,2]$. We also call $P$ the parameter set of $\mathcal{F}$.  
	
	The following lemma shows that the $f$-rectangular KT-condition of parameter set implies the rectangular KT-condition of $\mathcal{F}$ which is defined as \eqref{eq-defofFtrans}. 
	
	\begin{lemma}\label{lem:dual-rectKT-oneway}
		Let $\tau\in (0,1]$. Let $\mathcal{F}$ be defined as in \eqref{eq-defofFtrans} with parameter set $P\subset [-2,2]\times \R$. Assume that $P$ is an $f$-rectangular $(\delta,\tau,C)$-KT set in the sense of Definition \ref{def:recKT}. Then $\mathcal F$ is a rectangular $(\delta,\tau,C_{\mathbf{c},\mathbf{d}}C)$-KT set in the sense of Definition \ref{def:rectangleofF}.
	\end{lemma}
	
	\begin{proof}
		Fix $x\in[-2,2]$, and write
		\[A_x(a,b):=A_x(f_{a,b})=(f_{a,b}(x),f_{a,b}'(x))=(b-f(a-x),\,f'(a-x)).\]
		To prove that $\mathcal F$ is a rectangular $(\delta,\tau,C_{\mathbf{c},\mathbf{d}}C)$-KT set, it suffices to show that for every pair of intervals $I_{r'}$, $I_r$ with $|I_{r'}|\le r'$, $|I_r|\le r$,
		\[|A_x(P)\cap (I_{r'}\times I_r)|_\delta\lesssim_{\mathbf{c},\mathbf{d}} C\Bigl(\frac{\sqrt{rr'}}{\delta}\Bigr)^\tau.\]
		Let $c$ be the center of $I_{r'}$. If $(a,b)\in P$ satisfies
		\[A_x(a,b)\in I_{r'}\times I_r,\]
		then
		\[b-f(a-x)\in I_{r'},\qquad f'(a-x)\in I_r.\]
		The first condition implies
		\[|b-(f(a-x)+c)|\le r'.\]
		Since $|f''| \in [\mathbf{d},\mathbf{c}]$, the second condition implies that $a$ belongs to an interval $J$ with $|J|\lesssim_{\mathbf{c},\mathbf{d}} r$. Hence $A_x^{-1}(I_{r'}\times I_r)$ is contained in a curved rectangle
		\[\bar{R}=\{(a,b): a\in J,\ |b-(f(a-x)+c)|\le r'\}.\]
		Consequently, by using the rectangular KT-condition of $P$,
		\[|A_x(P)\cap (I_{r'}\times I_r)|_\delta\lesssim |P\cap \bar{R}|_\delta\le_{\mathbf{c},\mathbf{d}} C\Bigl(\frac{\sqrt{rr'}}{\delta}\Bigr)^\tau,\]
		as required.
	\end{proof}

	Now we state the dual incidence estimate of Theorem \ref{thm-recincidence}.
	
	\begin{cor}\label{cor-recincidence}
		Let $s\in (0,1]$, $C_1,C_2\geq1$ and $\mathbf{c},\mathbf{d}>0$. For any $\epsilon>0$, there exists $\delta_0=\delta_0(\mathbf{c},\mathbf{d},s,\epsilon)>0$ such that the following holds for all $\delta\in (0,\delta_0]$.
		
		Let $\mathcal{F}$ be defined as in \eqref{eq-defofFtrans}. Let $\mathbb{P}\subset \mathcal{D}_\delta([-1,1]^2)$. Assume that
		\begin{itemize}
			\item $\mathbb{P}$ is an $f$-rectangular $(\delta,2s,C_1)$-KT set in the sense of Definition \ref{def:recKT}.
			\item For each $p\in \mathbb{P}$, assume that there exists a $(\delta,\sigma,C_2)$-KT sub-family 
			\[\mathbb{F}(p)\subset\{f_{a,b}\in \mathcal{F}:p\cap \Gamma_{a,b}\neq \emptyset\},\]
			where $\Gamma_{a,b}$ is the graph of $f_{a,b}$ and $\sigma=\min\{2s,2-2s\}$. 
		\end{itemize}
		
		Write $\mathbb{F}=\cup_{p\in \mathbb{P}}\mathbb{F}(p)$, then
		\begin{equation}\label{equ-correcincid111}
			\mathcal{I}(\mathbb{F},\mathbb{P}):=\sum_{p\in \mathbb{P}}|\mathbb{F}(p)|\leq C_1^{2/3}C_2^{1/3}\delta^{-2s/3-\epsilon}|\mathbb{F}|^{2/3} |\mathbb{P}|^{1/3}.
		\end{equation}
	\end{cor}
	\begin{proof}
		First, we use point-curve duality to swap the roles of $\mathbb{F}$ and $\mathbb{P}$. More explicitly, for any $p \in \mathbb{P}$ with lower-left corner $(x_p,y_p)$, we have 
		\begin{align}
			p \cap \Ga_{a,b} \neq \emptyset & \Longrightarrow \exists~ (x,y)\in p \quad \text{such that}\quad |y - (b - f(a-x))| \leq \de \\
			&\Longrightarrow |b - f^\ast_{x_p,y_p}(a)| \lesssim_{\mathbf{c}} \de \label{align-bfstar},
		\end{align}
		where $f^\ast_{x_p,y_p}(\cdot)=f(\cdot-x_p)+y_p$. Write
		\[\pi(\mathbb{F}(p)):=\{(a,b): f_{a,b}\in \mathbb{F}(p)\}.\] 
		For each $(a,b)\in \pi(\mathbb{F}(p))$, let $q(a,b,p)\in\mathcal{D}_\delta$ be the dyadic $\delta$-cube containing $(a,f^\ast_{x_p,y_p}(a))$.
		Then
		\[q(a,b,p)\cap \Gamma_{f^\ast_{x_p,y_p}}\neq\varnothing.\]
		Define
		\[\mathcal{G}:=\{f^\ast_{x_p,y_p}:p\in\mathbb{P}\},\]
		and for each $f^\ast_{x_p,y_p}\in\mathcal{G}$, define
		\[\mathcal{Q}(f^\ast_{x_p,y_p}):=\{q(a,b,p):(a,b)\in \pi(\mathbb{F}(p))\}, \quad \mathcal{Q}:=\bigcup_{p\in\mathbb{P}}\mathcal{Q}(f^\ast_{x_p,y_p}).\]
		
		The new configuration $(\mathcal{G},\mathcal{Q})_\delta$ satisfies the following properties.
		\begin{enumerate}
			\item $\mathcal{G}$ is a rectangular $(\delta,2s,O_{\mathbf{c},\mathbf{d}}(C_1))$-KT set with $|\mathcal{G}|=|\mathbb{P}|$.
			\item For each $f^\ast_{x_p,y_p}\in\mathcal{G}$, $\mathcal{Q}(f^\ast_{x_p,y_p})$ is a $(\delta,\sigma,O_{\mathbf{c},\mathbf{d}}(C_2))$-KT set and $|\mathcal{Q}|\lesssim_{\mathbf{c}}|\mathbb{F}|$.
			\item For each $p\in\mathbb{P}$, \[|\mathbb{F}(p)|=|\pi(\mathbb{F}(p))|\lesssim_{\mathbf{c}}|\mathcal{Q}(f^\ast_{x_p,y_p})|.\]
		\end{enumerate}
		Property (1) follows from Lemma \ref{lem:dual-rectKT-oneway}. Next we prove property (2). To show the KT-condition for each $\mathcal{Q}(f^\ast_{x_p,y_p})$, it suffices to show
		\[|\mathcal{Q}(f^\ast_{x_p,y_p})\cap B(z,r)|_\delta \lesssim_{\mathbf{c},\mathbf{d}} C_2 \bigg(\frac{r}{\delta}\bigg)^\sigma, \quad z\in\R^2, r\in [\delta,1].\]
		For each cube $q\in\mathcal{Q}(f^\ast_{x_p,y_p})\cap B(z,r)$, there exists $(a,b)\in \pi(\mathbb{F}(p))$ such that $q=q(a,b,p)$. Since $q$ intersects $B(z,r)$ and contains $(a,f^\ast_{x_p,y_p}(a))$, we have
		\[(a,f^\ast_{x_p,y_p}(a))\in B(z,r+O(\delta)).\]
		Using \eqref{align-bfstar}, we get
		\[(a,b)\in B(z,r+O_{\mathbf{c}}(\delta))\subset B(z, C_{\mathbf{c}}r).\]
		Since $\pi(\mathbb{F}(p))$ is a $(\delta,\sigma,O_{\mathbf{c},\mathbf{d}}(C_2))$-KT set, 
		\[|\mathcal{Q}(f^\ast_{x_p,y_p})\cap B(z,r)|_\delta\leq|\pi(\mathbb{F}(p))\cap B(z,C_{\mathbf{c}}r)|_\delta\lesssim_{\mathbf{c},\mathbf{d}}C_2\left(\frac r\delta\right)^\sigma.\]
		To show $|\mathcal{Q}|\lesssim_{\mathbf{c}}|\mathbb{F}|$, take $q\in\mathcal{Q}$ with $q=q(a,b,p)$ for some $(a,b)\in \pi(\mathbb{F}(p))$. Let $q(a,b)\in\mathcal{D}_\delta$ be the dyadic $\delta$-cube containing $(a,b)$. Since $q(a,b,p)$ contains $(a,f^\ast_{x_p,y_p}(a))$ and
		\[|b-f^\ast_{x_p,y_p}(a)|\lesssim_{\mathbf{c}}\delta,\]
		$q(a,b,p)$ lies in the $O_{\mathbf{c}}(\delta)$-neighborhood of $q(a,b)$. Consequently,
		\[\mathcal{Q}\subset\bigcup_{f_{a,b}\in\mathbb{F}}\mathcal{N}_{\mathbf{c}}(q(a,b)),\]
		where $\mathcal{N}_{\mathbf{c}}(q(a,b))$ denotes the collection of dyadic $\delta$-cubes whose distance from $q(a,b)$ is at most
		$O_{\mathbf{c}}(\delta)$. Since each such neighborhood contains only $O_{\mathbf{c}}(1)$ dyadic $\delta$-cubes, 
		\[|\mathcal{Q}|\leq\sum_{f_{a,b}\in\mathbb{F}} |\mathcal{N}_{\mathbf{c}}(q(a,b))|\lesssim_{\mathbf{c}}|\mathbb{F}|.\]
		To prove property (3), we only need to show that each $q\in \mathcal{Q}(f^\ast_{x_p,y_p})$ corresponds to at most $O_{\mathbf{c}}(1)$ many $(a,b)\in\pi(\mathbb{F}(p))$. Indeed, if several points $(a,b)\in \pi(\mathbb{F}(p))$ give the same cube $q$, then the points
		$(a,f^*_{x_p,y_p}(a))$ all lie in the same $\delta$-cube $q$. Together with
		\[|b-f^\ast_{x_p,y_p}(a)|\lesssim_{\mathbf{c}}\delta,\]
		this implies that all such $(a,b)$ lie in an $O_{\mathbf{c}}(\delta)$-ball. Since $\pi(\mathbb{F}(p))$ is $\delta$-separated, 
		\[|\{(a,b)\in\pi(\mathbb{F}(p)):q(a,b,p)=q\}|\lesssim_{\mathbf{c}}1.\]
		
		By propoerty (3),
		\[\mathcal{I}(\mathbb{F},\mathbb{P})=\sum_{p\in\mathbb{P}}|\mathbb{F}(p)|\lesssim_{\mathbf{c}}\sum_{p\in\mathbb{P}}|\mathcal{Q}(f^\ast_{x_p,y_p})|.\]
		Consequently, we apply Theorem \ref{thm-recincidence} to $(\mathcal{G},\mathcal{Q})$ with $\epsilon$ replaced by $\epsilon/2$, then obtain
		\[\mathcal{I}(\mathbb{F},\mathbb{P})\lesssim_{\mathbf{c}}\sum_{g\in\mathcal{G}}|\mathcal{Q}(g)| \lesssim_{\mathbf{c},\mathbf{d}} C_1^{2/3}C_2^{1/3}\delta^{-2s/3-\epsilon/2}|\mathcal{Q}|^{2/3} |\mathcal{G}|^{1/3}.\]
		Using that $|\mathcal{Q}|\lesssim_{\mathbf{c}}|\mathbb{F}|$ and $|\mathcal{G}|=|\mathbb{P}|$, we conclude \eqref{equ-correcincid111} by choosing $\delta$ small enough.
		
	\end{proof}

	\subsection{Proof of Theorem \ref{thm-L7decay}} 
	Given finite subsets $S_1, S_2, S_3\subset \R^2$ and $\delta>0$, define 
	\[\mathcal{E}_{3,\delta}(S_1,S_2,S_3):=|\{(s_1,\cdots,t_3)\in (S_1\times S_2\times S_3)^2:|s_1+s_2+s_3-t_1-t_2-t_3|\leq 2\delta\}|.\]

	The main idea is the same as in \cite[Section 4]{doprod}. We first apply the incidence estimate in Corollary \ref{cor-recincidence} to establish the following energy estimate, then Theorem \ref{thm-L7decay} will be derived from Theorem \ref{thm-section5main}. 
	
	\begin{thm}\label{thm-section5main}
		Let $s\in (0,2/3]$. For any $\epsilon\in(0,1)$, there exists $\delta_0=\delta_0(\mathbf{c},\mathbf{d},s,\epsilon)>0$ such that the following holds for all $\de\in (0,\delta_0]$. 
		
		Let $S_1, S_2, S_3\subset \Gamma$ be three $\delta$-separated subsets lying above three intervals $I_1, I_2, I_3\subset [-1,1]$ respectively. Assume additionally that 
		\begin{itemize}
			\item $S_i$ is a $(\delta,s)$-KT set for all $i=1,2,3$;
			\item $\dist(I_i,I_j)\gtrsim 1$ for any $i\neq j$.
		\end{itemize}
		Then
		\begin{equation}\label{equ-energy}
			\mathcal{E}_{3,\delta}(S_1,S_2,S_3)\leq\delta^{-s-\epsilon} (|S_1|\cdot|S_2|\cdot |S_3|)^{5/6}.
		\end{equation}
	\end{thm}

	To prove Theorem \ref{thm-section5main}, we first reduce the energy estimate to an incidence estimate between a cube family and a transversal family. The following reduction is performed under the same assumptions as in Theorem \ref{thm-section5main}. Similar reductions for parabolas can be found in \cite{O1} and \cite{doprod}.

	Given $(t_1,t_2,t_3)\in S_1\times S_2\times S_3$, let
	\[\mathcal{N}(t_1,t_2,t_3):=\big|\big\{(s_1,s_2,s_3)\in S_1\times S_2\times S_3:\bigg|\sum_{i=1}^3s_i-\sum_{j=1}^3t_j\bigg|\leq 2\delta\big\}\big|.\]
	Then we have
	\begin{equation}\label{equ-energ}
		\mathcal{E}_{3,\delta}(S_1,S_2,S_3)=\sum_{(t_1,t_2,t_3)\in S_1\times S_2\times S_3}\mathcal{N}(t_1,t_2,t_3).
	\end{equation}
	Fix $(t_1,t_2,t_3)\in S_1\times S_2\times S_3$ and suppose that $t_1+t_2+t_3=(a,b)$. Then for any $s_i=(z_i,f(z_i))\in S_i$ with $|\sum_{i=1}^3s_i-\sum_{j=1}^3t_j|\leq 2\delta$, we have
	\begin{equation}\label{equ-relation}
		\left\{
		\begin{array}{lll}
			|z_1+z_2+z_3-a|\leq 2\delta,\\
			|f(z_1)+f(z_2)+f(z_3)-b|\leq2\delta.
		\end{array}
		\right.
	\end{equation}
	Write \[z_3=a-z_1-z_2+e, \quad |e|\leq 2\delta.\] 
	Then \[f(z_1)+f(z_2)=b-f(a-z_1-z_2+e)+e', \quad |e|, ~|e'|\leq 2\delta.\]
	By using the mean value theorem, we infer
	\[|f(a-z_1-z_2+e)-f(a-z_1-z_2)|\leq \|f'\|_{L^\infty([-6,6])}\cdot 2\delta\leq 2\mathbf{c}\delta,\]
	which implies
	\begin{equation}\label{equ-taylor}
		|f(z_1)+f(z_2)-(b-f(a-z_1-z_2))|\leq (2+2\mathbf{c})\delta.
	\end{equation}
	Here we recall $\|f\|_{C^3([-6,6])}\leq \mathbf{c}$. Define 
	\[f_{a,b}(t)=b-f(a-t) \quad \text{and} \quad \Phi(z_1,z_2)=(z_1+z_2,f(z_1)+f(z_2))=s_1+s_2,\]
	then \eqref{equ-taylor} means that $\Phi(z_1,z_2)$ lies in the vertical $C_\mathbf{c}\delta$-neighborhood $\Gamma_{a,b}(C_\mathbf{c}\delta)$ of the graph of $f_{a,b}$ (recall Definition \ref{def:vertical}). Here $C_\mathbf{c}\geq1$ is a constant depending on $\mathbf{c}$, which may vary from line to line in the following.
	
	Quickly note that since $\dist(I_1,I_1)\gtrsim1$, and by the convexity of $f,$ we have
	\begin{equation}\label{eq.derivative}
		|D\Phi(z_1z_2)| = |f'(z_1) - f'(z_2)| \sim_{\mathbf{c},\mathbf{d}} 1, \quad (z_1,z_2)\in I_1\times I_2
	\end{equation}
	where we recall $\min_{x\in [-6,6]}|f''(x)|\geq \mathbf{d}$. Therefore $\Phi$ is bi-Lipschitz on $I_1\times I_2$.

	Let $\mathcal{P}^\ast\subset\mathcal{D}_\delta$ be the set of dyadic $\delta$-cubes intersecting 
	\[\begin{split}
		\Phi(\pi_x(S_1),\pi_x(S_2)):&=\{\Phi(z_1,z_2):(z_1,z_2)\in \pi_x(S_1)\times \pi_x(S_2)\}\\
		&=S_1+S_2
	\end{split}\]
	Here $\pi_x$ is the $x$-projection, hence $\pi_x(S_i)\subset I_i$. For convenience, we introduce the following definition.
	
	\begin{definition}\label{def:reptable}
		Let $(a,b)\in\R^2$ and $p\in\mathcal{P}^\ast=\mathcal{D}_\delta(S_1+S_2)$. We say $(a,b)$ is \emph{$p$-representable} if there exists a triple $(s_1,s_2,s_3)\in S_1\times S_2\times S_3$ such that $s_1+s_2\in p$ and \[|s_1+s_2+s_3-(a,b)|\leq 2\delta.\]
	\end{definition}

	We aim to prove the following lemma.
	\begin{lemma}
		For fixed $(t_1,t_2,t_3)\in S_1\times S_2\times S_3$, we have
		\begin{equation}\label{equ-66}
			\mathcal{N}(t_1,t_2,t_3)\sim_{\mathbf{c},\mathbf{d}} |\{p\in\mathcal{P}^\ast: (t_1+t_2+t_3) \text{ is } \text{$p$-representable}\}|.
		\end{equation}
	\end{lemma}
	
	\begin{proof}
		Recall that $\mathcal{N}(t_1,t_2,t_3)$ is the number of $(s_1,s_2,s_3)\in S_1\times S_2\times S_3$ such that 
		\begin{equation}\label{eq-tsdef}
			\bigg|\sum_{i=1}^3s_i-\sum_{j=1}^3t_j\bigg|\leq 2\delta,
		\end{equation}
		where $s_i=(z_i,f(z_i))$. Fix one such triple $(s_1,s_2,s_3)$. First, there is a unique cube $p\in\mathcal{P}^\ast$ such that $\Phi(z_1,z_2)\in p$. Since we also have \eqref{eq-tsdef}, this means that $t_1+t_2+t_3$ is $p$-representable. Now each triple $(s_1,s_2,s_3)$ corresponds to one $p\in\mathcal{P}^\ast$ and each $p\in \mathcal{P}^\ast$ contains $O_{\mathbf{c},\mathbf{d}}(1)$ many $\Phi(z_1,z_2)$ since $S_i$ are $\delta$-separated and $\Phi$ is bi-Lipschitz, thus
		\[\begin{split}
			\mathcal{N}(t_1,t_2,t_3)&\lesssim |\{(z_1,z_2)\in\pi_x(S_1)\times\pi_x(S_2): \\&~~~~~~(t_1+t_2+t_3) \text{ is } \text{$p$-representable with }\Phi(z_1,z_2)\in p\in\mathcal{P}^\ast\}|\\
			&\lesssim_{\mathbf{c},\mathbf{d}} |\{p\in\mathcal{P}^\ast: (t_1+t_2+t_3) \text{ is } \text{$p$-representable}\}|.
		\end{split}\]
		The opposite inequality $\gtrsim_{\mathbf{c},\mathbf{d}}$ follows by the definition of $p$-representable.
	\end{proof}
	
	\begin{remark}\label{remk11}
		From the argument above, we actually have \[\{p\in\mathcal{P}^\ast: (a,b) \text{ is } \text{$p$-representable}\}\subset\{p\in\mathcal{P}^\ast:p\subset \Gamma_{a,b}(C_\mathbf{c}\delta)\}.\]
	\end{remark}

	In the following, when $t_1+t_2+t_3=(a,b)$, we define 
	\[\Gamma_{t_1,t_2,t_3}:=\Gamma_{t_1+t_2+t_3}=\Gamma_{a,b} \quad \text{and} \quad f_{t_1,t_2,t_3}:=f_{t_1+t_2+t_3}=f_{a,b}.\]
	Here $\Gamma_{a,b}$ is the graph of $f_{a,b}$. Two curves $\Gamma_{t_1,t_2,t_3}$ and $\Gamma_{s_1,s_2,s_3}$ are said to be \emph{comparable}, written
	\[\Gamma_{t_1,t_2,t_3}\sim \Gamma_{s_1,s_2,s_3},\]
	if
	\[|t_1+t_2+t_3-s_1-s_2-s_3|\leq 2\delta.\]
	Let $\mathcal{G}$ be a maximal set of incomparable curves in \[\{\Gamma_{t_1,t_2,t_3}: (t_1,t_2,t_3)\in S_1\times S_2\times S_3\}\] and let $\mathcal{F}^\ast=\{f_{a,b}: \Gamma_{a,b}\in\mathcal{G}\}$ be the corresponding function family. Now combining \eqref{equ-energ} and \eqref{equ-66}, we infer
	\[\mathcal{E}_{3,\delta}(S_1,S_2,S_3)\sim_{\mathbf{c},\mathbf{d}}\sum_{\Gamma_{a,b}\in \mathcal{G}}\sum_{\substack{(t_1,t_2,t_3)\in S_1\times S_2\times S_3\\\Gamma_{t_1,t_2,t_3}\sim \Gamma_{a,b}}}|\{p\in\mathcal{P}^\ast: (a,b) \text{ is } \text{$p$-representable}\}|.\]
	To proceed, we need to estimate the number of $(t_1,t_2,t_3)\in S_1\times S_2\times S_3$ such that \[\Gamma_{t_1,t_2,t_3}\sim\Gamma_{a,b}\] (equivalently $|t_1+t_2+t_3-(a,b)|\leq 2\delta$). This is the same as $\mathcal{N}(t_1,t_2,t_3)$, hence 
	\begin{equation}\label{equ-reductionincidence}
		\mathcal{E}_{3,\delta}(S_1,S_2,S_3)\sim_{\mathbf{c},\mathbf{d}}\sum_{\Gamma_{a,b}\in \mathcal{G}}|\{p\in\mathcal{P}^\ast: (a,b) \text{ is } \text{$p$-representable}\}|^2.
	\end{equation}
	In the sequel, we give some basic properties of $(\mathcal{F}^\ast,\mathcal{P}^\ast)_\delta$. The proofs of Lemma \ref{lem-p1} and Lemma \ref{lem-p0} can be found in \cite[Example 1.5]{OPY2025}.

	\begin{lemma}\label{lem-p1}
		$\mathcal{F}^\ast$ is a transversal family on $[-2,2]$ with constant depending only on $\mathbf{c},\mathbf{d}$.
	\end{lemma}

	\begin{lemma}\label{lem-p0}
		For any $f_{a_1,b_1}, f_{a_2,b_2}\in \mathcal{F}^\ast$, we have
		\[\|f_{a_1,b_1}-f_{a_2,b_2}\|_{C^2([-2,2])}\sim_{\mathbf{c},\mathbf{d}} |(a_1,b_1)-(a_2,b_2)|.\]
	\end{lemma}
	
	\begin{lemma}\label{lem-p2}
		For all $p\in \mathcal{P}^\ast=\mathcal{D}_\delta(S_1+S_2)$, \[\mathcal{F}^\ast(p):=\{f_{a,b}\in\mathcal{F}^\ast: (a,b)\text{ is } \text{$p$-representable}\}\] is a $(\delta,s,O_{\mathbf{c},\mathbf{d}}(1))$-KT set in $C^2([-2,2])$.
	\end{lemma}
	\begin{proof}
		Let $P^\ast:=\{(a,b):f_{a,b}\in \mathcal{F}^\ast\}$. By Lemma \ref{lem-p0}, it suffices to prove that 
		\[P^\ast(p):=\{(a,b): f_{a,b}\in \mathcal F^\ast(p)\}\]
		is a $(\delta,s, O_{\mathbf{c},\mathbf{d}}(1))$-KT set. Fix $p\in \mathcal{P}^\ast$, and let
		\[Z(p):=\{(s_1,s_2)\in S_1\times S_2: s_1+s_2\in p\}.\]
		Since $\Phi$ is bi-Lipschitz and $S_i$ are $\delta$-separated, we have $|Z(p)|\lesssim_{\mathbf{c},\mathbf{d}} 1$.
		
		For each $(s_1,s_2)\in Z(p)$, define
		\[C_{s_1,s_2}:=s_1+s_2+S_3=\Big\{s_1+s_2+s_3: s_3\in S_3\Big\}.\]
		Since $S_3$ is a $(\delta,s)$-KT set, so is $C_{s_1,s_2}$. Let $(a,b)\in P^\ast(p)$, then $(a,b)$ is $p$-representable. Thus there is a triple $(s_1,s_2,s_3)$ such that $s_1+s_2\in p$ and \[|s_1+s_2+s_3-(a,b)|\leq 2\delta.\]
		This means $(a,b)\in C_{s_1,s_2}(2\delta)$, hence
		\[P^\ast(p)\subset \bigcup_{(s_1,s_2)\in Z(p)} C_{s_1,s_2}(2\delta).\]
		Since $|Z(p)|\lesssim_{\mathbf{c},\mathbf{d}} 1$ and $P^\ast$ is $2\delta$-separated, it follows from the $(\delta,s)$-KT condition of $C_{s_1,s_2}$ that $P^\ast(p)$ is a $(\delta,s,O_{\mathbf{c},\mathbf{d}}(1))$-KT set.
	\end{proof}

	\begin{lemma}\label{lem-p3}
		The cube family $\mathcal{P}^\ast$ is an $f$-rectangular $(\delta,2s,O_{\mathbf{c},\mathbf{d}}(1))$-KT set with 
		\[|\mathcal{P}^\ast|\leq |S_1||S_2|.\]
	\end{lemma}
	\begin{proof}
		Take any curved rectangle $\mathbf R=\gamma(r')$ as in Definition \ref{def:recKT}, where 
		\[\gamma=\{(t,f_{a,b}^\ast(t)):t\in I_\gamma\}, \quad |I_\gamma|\lesssim r\]
		with $f_{a,b}^\ast(t)=f(t-a)+b$ for some $(a,b)\in [-2,2]\times \R$. Without loss of generality, we assume $r\geq r'$. Let $I_1, I_2$ be the intervals in Theorem \ref{thm-section5main}. Since each $p\in\mathcal{P}^\ast$ contains one and at most $O_{\mathbf{c},\mathbf{d}}(1)$ points in $\Phi(\pi_x(S_1),\pi_x(S_2))$, it suffices to count the number of points $\Phi(z_1,z_2)$ contained in $\mathbf{R}$. Let
		\[U:=I_1\times I_2,\quad E:=\pi_x(S_1)\times \pi_x(S_2) \quad \text{and} \quad\Phi_U:=\Phi|_U .\]
		Since $E\subset U$ and $\Phi_U$ is injective, we have
		\[|\Phi(E)\cap \mathbf R|=|E\cap \Phi_U^{-1}(\mathbf R\cap \Phi(U))|.\]
		Set
		\[\mathbf W:=\Phi_U^{-1}(\mathbf R\cap \Phi(U)).\]
		It remains to show that
		\[|E\cap \mathbf W|\lesssim_{\mathbf{c},\mathbf{d}} \bigg(\frac{\sqrt{rr'}}{\delta}\bigg)^{2s}.\]

		To proceed, we take $(x,y)\in \mathbf W$, then 
		\[x+y\in I_\gamma \quad \text{and} \quad |f(x)+f(y)-f(x+y-a)-b|\leq r'.\]
		Since $\dist(I_1,I_2)\gtrsim1$, We may assume that
		\[\dist(a,I_1)\gtrsim 1.\]
		Otherwise, we have $\dist(a,I_2)\gtrsim 1$ which can be treated by the same argument as below. Now we claim that for each dyadic $r'$-cube $\mathbf{I}\subset I_1$ intersecting $\pi_x(\mathbf W)$, the diameter
		\begin{equation}\label{eq-clami1}
			\diam\big(\{y\in I_2: (x,y)\in\mathbf{W} \text{ for some } x\in\mathbf{I}\}\big)\lesssim_{\mathbf{c},\mathbf{d}} r'.
		\end{equation}
		Indeed, write $F(x,y)=f(x)+f(y)-f(x+y-a)-b$, then
		\[|\partial_yF(x,y)|=|f'(y)-f'(x+y-a)|\gtrsim_{\mathbf{d}}1, \quad x\in I_1,\]
		where we use $f''\geq \mathbf{d}$ and $\dist(a,I_1)\gtrsim1$. If $(x_1,y_1),(x_2,y_2)\in \mathbf W\cap (\mathbf{I}\times I_2)$, then
		\[|F(x_1,y_1)|\leq r'\quad \text{and} \quad |F(x_2,y_2)|\leq r'.\]
		Using the triangle inequality, the mean value theorem, $|\partial_yF|\gtrsim_{\mathbf{d}}1$ and $|\partial_xF|\lesssim_{\mathbf c}1$,
		\[|y_1-y_2|\lesssim_{\mathbf d}|F(x_1,y_1)-F(x_1,y_2)|\lesssim_{\mathbf{c}} r',\]
		which implies \eqref{eq-clami1}.

		Next we show that the $x$-projection of $\mathbf W$ has length
		$O_{\mathbf c,\mathbf d}(r)$. Take $(x_i,y_i)\in \mathbf W$ and write $t_i:=x_i+y_i$ for $i=1,2$, thus $t_i\in I_\gamma$.
		Write \[\tilde F(t,u)=f(u)+f(t-u)-f(t-a)-b,\] then $F(x,y)=\tilde F(x+y,x)$,
		hence
		\[|\tilde F(t_i,x_i)|\leq r',\quad i=1,2.\]
		Connect $(t_1,x_1)$ and $(t_2,x_2)$ by the line segment
		\[(t_\theta,u_\theta)=\big((1-\theta)t_1+\theta t_2,(1-\theta)x_1+\theta x_2 \big) \in I_\gamma \times I_1,\quad 0\leq \theta\leq 1.\]
		Then we have $t_\theta-u_\theta\in I_2$ and 
		\[|\partial_u\tilde F(t_\theta,u_\theta)|=|f'(u_\theta)-f'(t_\theta-u_\theta)|\gtrsim_{\mathbf d}1.\]
		Since $(u_\theta,t_\theta-u_\theta)\in I_1\times I_2$ and $|f''|\geq \mathbf{d}$, $\partial_u\tilde F(t_\theta,u_\theta)$ has a fixed sign for $\theta\in [0,1]$. Hence
		\[\left|\int_0^1 \partial_u\tilde F(t_\theta,u_\theta)\,d\theta\right|\gtrsim_{\mathbf d}1.\]
		On the other hand, we have $|\partial_t\tilde F|\lesssim_{\mathbf c}1$. By the fundamental theorem of calculus,
		\[\tilde F(t_2,x_2)-\tilde F(t_1,x_1)=(x_2-x_1)\int_0^1 \partial_u\tilde F(t_\theta,u_\theta)\,d\theta +(t_2-t_1)\int_0^1 \partial_t\tilde F(t_\theta,u_\theta)\,d\theta .\]
		Consequently,
		\[|\tilde F(t_2,x_2)-\tilde F(t_1,x_1)|\geq C_{\mathbf d}|x_2-x_1|-C_{\mathbf c}|t_2-t_1|.\]
		Since $|\tilde F(t_i,x_i)|\leq r'$ for $i=1,2$, we infer
		\[|\tilde F(t_2,x_2)-\tilde F(t_1,x_1)|\leq 2r'.\]
		Therefore
		\[|x_2-x_1|\lesssim_{\mathbf c,\mathbf d}|t_2-t_1|+r'\lesssim_{\mathbf c,\mathbf d} r,\]
		using $r'\leq r$. Hence the $x$-projection of $\mathbf W$ is contained in an interval of length $O_{\mathbf c,\mathbf d}(r)$.
		
		Now cover $\mathbf W$ by $r'$-dyadic squares. Let $\mathcal I$ be the family of $r'$-dyadic intervals $\mathbf{I}$ intersecting the $x$-projection of $\mathbf W$. Then the union of the intervals in $\mathcal I$ is contained in an interval of length $O_{\mathbf c,\mathbf d}(r)$. For each $\mathbf{I}\in\mathcal I$, let
		\[\mathcal{K}_\mathbf{I}=\{\mathbf{J}\in\mathcal{D}_{r'}: (\mathbf{I}\times \mathbf{J})\cap \mathbf W\neq \emptyset\}.\]
		By \eqref{eq-clami1}, we have $|\mathcal{K}_\mathbf{I}|\lesssim_{\mathbf{c},\mathbf{d}}1$. Therefore, by the KT-conditions of $S_1$ and $S_2$, we infer
		\[\begin{split}
			|E\cap \mathbf W|&\leq\sum_{\mathbf{I}\in \mathcal{I}}\sum_{\mathbf{J}\in \mathcal{K}_\mathbf{I}} |\pi_x(S_1)\cap \mathbf{I}|\,|\pi_x(S_2)\cap \mathbf{J}|\\
			&\lesssim_{\mathbf{c},\mathbf{d}} \bigg(\frac{r'}{\delta}\bigg)^s\sum_{\mathbf{I}\in \mathcal{I}}|\pi_x(S_1)\cap \mathbf{I}|\lesssim_{\mathbf{c},\mathbf{d}} \bigg(\frac{r'}{\delta}\bigg)^s\bigg(\frac{r}{\delta}\bigg)^s,
		\end{split}\]
		which proves the rectangular KT-condition for $\mathcal{P}^\ast$. Finally, $|\mathcal P^*|\leq |S_1||S_2|$ follows directly from the definition
		of $\mathcal P^*$.
	\end{proof}

	After those reductions, Theorem \ref{thm-section5main} now follows easily from Corollary \ref{cor-recincidence}.
	
	\begin{proof}[Proof of Theorem \ref{thm-section5main}]
		Recall that $\mathcal{P}^\ast$ is the set of dyadic $\delta$-cubes intersecting $S_1+S_2$
		and $\mathcal{G}$ is the set of incomparable curves in
		\[\{\Gamma_{t_1,t_2,t_3}: (t_1,t_2,t_3)\in S_1\times S_2\times S_3\}.\]
		By \eqref{equ-reductionincidence}, we have
		\begin{equation}\label{equ-98}
			\mathcal{E}_{3,\delta}(S_1,S_2,S_3)\sim_{\mathbf{c},\mathbf{d}}\sum_{\Gamma_{a,b}\in \mathcal{G}}|\{p\in\mathcal{P}^\ast: (a,b) \text{ is } \text{$p$-representable}\}|^2=:\sum_{g\in \mathcal{F}^\ast}|\mathcal{P}^\ast(g)|^2.
		\end{equation}
		From Remark \ref{remk11} we also know
		\[\{p\in\mathcal{P}^\ast: (a,b) \text{ is } \text{$p$-representable}\}\subset\{p\in\mathcal{P}^\ast:p\subset \Gamma_{a,b}(C_\mathbf{c}\delta)\}.\]
		After replacing $\mathcal{P}^\ast$ by $C_{\mathbf{c}}\delta$-cubes, we may assume that 
		\[\{p\in\mathcal{P}^\ast: (a,b) \text{ is } \text{$p$-representable}\}\subset\{p\in\mathcal{P}^\ast:p\cap \Gamma_{a,b}\neq\emptyset\}.\]
		Recalling Lemma \ref{lem-p2} and Lemma \ref{lem-p3}, we have:
		\begin{itemize}
			\item For each $p\in \mathcal{P}^\ast$, $$\mathcal{F}^\ast(p):=\{f_{a,b}\in\mathcal{F}^\ast: (a,b)\text{ is } \text{$p$-representable}\}$$ is a $(\delta,s,O_{\mathbf{c},\mathbf{d}}(1))$-KT set in $C^2([-2,2])$.
			\item $\mathcal{P}^\ast$ is an $f$-rectangular $(\delta,2s,O_{\mathbf{c},\mathbf{d}}(1))$-KT set with $|\mathcal{P}^\ast|\lesssim |S_1||S_2|$.
		\end{itemize}

		To proceed, we pigeonhole a sub-family $\mathcal{F}_1^\ast\subset \mathcal{F}^\ast$ and $M\in 2^\N$ such that
		\[|\mathcal{P}^\ast(g)|\sim M \quad \text{for all}\quad g\in \mathcal{F}_1^\ast\]
		and
		\begin{equation}\label{equ-99}
			\mathcal{E}_{3,\delta}(S_1,S_2,S_3)\lessapprox |\mathcal{F}_1^\ast|M^2.
		\end{equation}
		Note $\sigma\geq s$ since $s\in (0,2/3]$. Hence $\mathcal{F}_1^\ast(p):=\mathcal{F}^\ast(p)\cap\mathcal{F}_1$ is also a $(\delta,\sigma,O_{\mathbf{c},\mathbf{d}}(1))$-KT set. We apply Corollary \ref{cor-recincidence} to $(\mathcal{F}_1^\ast,\mathcal{P}^\ast)$ to get
		\begin{equation}\label{equ-99}
			\mathcal{I}(\mathcal{F}_1^\ast,\mathcal{P}^\ast)=M|\mathcal{F}_1^\ast|\lessapprox \delta^{-2s/3}|\mathcal{F}_1^\ast|^{2/3}\cdot |\mathcal{P}^\ast|^{1/3},
		\end{equation}
		which implies
		\begin{equation}\label{equ-100}
			M^3|\mathcal{F}_1^\ast|\lessapprox \delta^{-2s}|\mathcal{P}^\ast|.
		\end{equation}
		Moreover, we have the trivial estimate
		\begin{equation}\label{equ-101}
			M|\mathcal{F}_1^\ast|=\mathcal{I}(\mathcal{F}_1^\ast,\mathcal{P}^\ast)\leq |\mathcal{P}^\ast|\cdot \max_p|\mathcal{F}^\ast(p)|.
		\end{equation}
		Taking the geometric average of \eqref{equ-100} and \eqref{equ-101}, we obtain
		\begin{equation}\label{equ-102}
			\mathcal{E}_{3,\delta}(S_1,S_2,S_3)\lessapprox |\mathcal{F}^\ast|M^2 \lessapprox \delta^{-s}|\mathcal{P}^\ast|\cdot (\max_p|\mathcal{F}^\ast(p)|)^{1/2}.
		\end{equation}
		Since $|\mathcal{P}^\ast|\lesssim |S_1||S_2|$ and $|\mathcal{F}^\ast(p)|\leq |S_3|$, we infer \[\mathcal{E}_{3,\delta}(S_1,S_2,S_3)\lessapprox \delta^{-s}|S_1||S_2||S_3|^{1/2}.\]
		Since all the argument works for permutations of $(S_1,S_2,S_3)$, we take the geometric average of these three inequalities and conclude that
		\[\mathcal{E}_{3,\delta}(S_1,S_2,S_3)\lessapprox \delta^{-s} (|S_1|\cdot|S_2|\cdot |S_3|)^{5/6}.\]
		We conclude Theorem \ref{thm-section5main} by choosing $\delta<\delta_0(\mathbf{c},\mathbf{d},s,\epsilon)$ small enough.
	\end{proof}
	
	\begin{remark}
		The argument that Theorem \ref{thm-section5main} implies Theorem \ref{thm-L6decay} is almost the same as \cite[Proposition 4.14]{doprod}, with one minor difference concerning rescaling. In \cite[Proposition 4.14]{doprod}, one uses that Frostman measures on pieces of the parabola can be rescaled to Frostman measures supported on the parabola. In our setting, we instead use that Frostman measures supported on convex curves can be rescaled to Frostman measures supported on different convex curves, while preserving the key constants $\mathbf{c}\geq \mathbf{d}>0$.
	\end{remark}
	
	Write $\phi(x)=(x,f(x))$ for each $x\in [-6,6]$. Let $\mu$ be an $s$-Frostman measure on $\Gamma$ defined as in Theorem \ref{thm-L6decay}, then $\mu=\phi_{\#}(\mu_x)$ with $\mu_x$ an $s$-Frostman measure on $[-1,1]$. In the sequel, we identify $\mu_x$ with its pushforward $\mu$ for simplicity and thus write
	\[\hat{\mu}(\xi_1,\xi_2)=\int_{-1}^1 e^{-2\pi i (\xi_1 x+\xi_2 f(x))} d\mu(x).\]
	
	\begin{proof}[Proof of Theorem \ref{thm-L6decay}]
		We only sketch the difference. Define the function family
		\[\mathbb{H}=\big\{h\in C^3([-6,6]): \|h\|_{C^3([-6,6])}\leq 50\mathbf{c} \text{ and } \min_{x\in[-6,6]}|h''(x)| \geq\mathbf{d}\big\}.\]
		Let $C_R$ be the smallest constant such that 
		\begin{equation}\label{eq-smallestc}
			\|\hat{\nu}\|_{L^6(B_R)}^6 \leq C_R \|\nu\|
		\end{equation}
		holds for any $s$-Frostman measure supported on $\Gamma_h$, where $h\in \mathbb{H}$ is arbitrary. It suffices to prove for sufficiently small $\epsilon>0$ and large $R\geq K^2\geq1$, where $K=K(\epsilon)$ will be chosen later. 
		
		Partition $[-1,1]$ into intervals $I\in\cI_k$ of length $1/K$. For $I,I'\in\cI_K$ we write $I\not\sim I'$ if $I$ is not adjacent to $I'$. Let $\mu_{I_i}$ be the restriction of $\mu$ to $I_i$. As \cite[Proposition 4.14]{doprod}, we have
		\begin{equation}\label{equ-110}
			\int_{B_R}|\widehat{\mu}|^6\lesssim \int_{B_R}\sum_{I\in\cI_K}|\widehat{\mu_I}|^6+K^{200}\max_{I_1\not\sim I_2\not\sim I_3}\int_{B_R}|\widehat{\mu_{I_1}}|^2|\widehat{\mu_{I_2}}|^2|\widehat{\mu_{I_3}}|^2.
		\end{equation}
		Using Theorem \ref{thm-section5main} (see \cite[Proposition 4.13]{doprod} where the calculation does not depend on the parabola), the second term can be estimated as
		\begin{equation}\label{equ-secondterm}
			\max_{I_1\not\sim I_2\not\sim I_3}\int_{B_R}|\widehat{\mu_{I_1}}|^2|\widehat{\mu_{I_2}}|^2|\widehat{\mu_{I_3}}|^2 \lesssim  T_{\mathbf{c},\mathbf{d},s,K,\epsilon}R^{2-\frac{5s}2+\epsilon}(\|\mu_{I_1}\|\|\mu_{I_2}\|\|\mu_{I_3}\|)^{5/6},
		\end{equation}
		where $T_{\mathbf{c},\mathbf{d},s,K,\epsilon}>0$ is a constant depending on $\mathbf{c},\mathbf{d},s,K,\epsilon$. Note that estimate \eqref{equ-secondterm} holds for any $s$-Frostman $\mu$ supported on the graph of an function in $\mathbb{H}$. 
		
		The difference appears when analysing the first term. Given $a\in [-1,1]$, let \[I=[a-1/(2K),a+1/(2K)].\] Then
		\[\begin{split}
			|\widehat{\mu_I}(x,y)|&=\left|\int_I e^{-2\pi i (x \xi+yf(\xi))} d\mu(\xi)\right|\\&=\left|\int_{-1}^1 e^{-2\pi i (\tfrac{x+f'(a)y}{2K}\xi+\tfrac{y}{4K^2}f_{a,K}(\xi))}d \mu_I(a+\tfrac{\xi}{2K})\right|=:K^{-s}|\widehat{\nu_I}(\tfrac{x+f'(a)y}{2K},\tfrac{y}{4K^2})|.
		\end{split}\]
		Here $f_{a,K}(\xi)=4K^2(f(a+\tfrac{\xi}{2K})-f(a)-f'(a)\tfrac{\xi}{2K})$ and $\nu_I$ is defined by
		\[\nu_I(A):=K^{s} (\Phi_{a,K})_{\#}\mu_I(a+\tfrac{A}{2K}), \quad \text{where~~}\Phi_{a,K}(x)=(x,f_{a,K}(x)).\]
		Now $f_{a,K}$ is a new convex function, but one can check that $f_{a,K}\in \mathbb{H}$ and $\nu_I$ is an $s$-Frostman measure supported on $\Gamma_{f_{a,K}}$. This means \eqref{eq-smallestc} is still applicable to $\nu_I$.
		
		The remaining is totally the same as \cite[Proposition 4.14]{doprod}. One uses
		\[\|\widehat{\nu_I}\|_{L^6(B_{R/K^2})} \leq C_{R/K^2} \|\nu_I\|\lesssim C_{R/K^2} K^s \mu(I)\]
		to estimate the first term in \eqref{equ-110} and eventually get 
		\[C_{R}\le AC_{R/K^2}K^{4-5s}+B_{\mathbf{c},\mathbf{d},s,K,\epsilon}R^{2-\frac{5s}2+\frac{\epsilon}2}, \]
		where $A\ge 1$ is an absolute constant. Writing $D_R=C_{R}/R^{2-\frac{5s}2+\frac{\epsilon}2}$, we have 
		\[D_R\le AK^{-2\epsilon}D_{R/K^2}+B_{\mathbf{c},\mathbf{d},s,K,\epsilon}.\]
		Take $\epsilon$ sufficiently small and choose $K=A^{\tfrac{1}{2\epsilon}}$ and iterate this inequality $n=[\log R/2\log K]$ many times, then we use $R/K^{2n}\sim 1$ to obtain $D_R\lesssim_{\mathbf{c},\mathbf{d},s,\epsilon} \log R\leq R^{\epsilon/2}$ and thus 
		\[C_R\lesssim_{\mathbf{c},\mathbf{d},s,\epsilon} R^{2-\frac{5s}{2}+\epsilon}. \]
		This completes the whole proof.
		
	\end{proof}

	\bibliographystyle{plain}
	\bibliography{references}
	
\end{document}